\documentclass[a4paper,10pt]{article}
\usepackage[utf8]{inputenc}
\usepackage{amsmath}
\usepackage{amssymb}
\usepackage{amsthm}
\usepackage{fullpage}
\usepackage{dsfont}
\usepackage{xcolor}
\usepackage[ruled,vlined]{algorithm2e}
\usepackage{tikz} 
\usepackage{hyperref}
\usepackage{graphicx}
\usepackage{enumitem}
\usepackage{multirow}
\usepackage{bm}
\usepackage{mathrsfs}
\usepackage{stmaryrd}

\definecolor{corn}{RGB}{255,230,102}
\definecolor{deepsaffron}{RGB}{255,153,51}
\definecolor{ferrarired}{RGB}{255,42,0}

\newcommand{\bx}{{\bf a}}
\newcommand{\nsites}{\kappa}

\newcommand{\R}{\mathds{R}}
\newcommand{\dx}{\delta x}

\newcommand{\mv}{\mathcal{V}}

\newcommand{\mm}{\mathscr{M}}
\newcommand{\mS}{\mathcal{S}}

\newcommand{\divv}{\operatorname{div}}
\newcommand{\me}{\mathcal{E}}

\newcommand{\ba}{{\bf a}}
\newcommand{\si}{a}
\newcommand{\bsi}{{\bf a}}
\newcommand{\dbsi}{\delta\bsi}
\newcommand{\dsi}{\delta\si}

\newcommand{\vp}{\varphi}
\newcommand{\lsf}{\phi}
\newcommand{\blsf}{\bm{\phi}}
\newcommand{\hlsf}{\widehat\phi}
\newcommand{\hblsf}{\widehat\blsf}
\newcommand{\I}{\mathcal{I}}
\newcommand{\K}{\mathcal{K}}
\newcommand{\mL}{\mathcal{L}}
\newcommand{\card}{{|\I|}}
\newcommand{\ck}{{\kappa}}
\newcommand{\C}{\mathcal{C}}
\newcommand{\mM}{\mathcal{M}}
\newcommand{\mE}{\mathcal{E}}
\newcommand{\om}{\omega}

\newcommand{\trp}{Y}
\newcommand{\dop}{X}
\newcommand{\hme}{\me_i^{\rm{int}}}

\DeclareMathOperator{\rank}{rank}
\DeclareMathOperator{\inte}{int}

\SetKwInput{KwInput}{Input}
\SetKwInput{KwOutput}{Output}

\newtheorem{assump}{Assumption}
\newtheorem{remark}{Remark}
\newtheorem*{notation}{Notation}
\newtheorem{example}{Example}

\newtheorem{thm}{Theorem}
\newtheorem{lemma}{Lemma}
\newtheorem{definition}{Definition}

\begin{document}

\title{Sensitivity analysis and tailored design of minimization diagrams\thanks{This work has been partially supported by FAPESP (grants 2013/07375-0, 2016/01860-1, 2018/24293-0, and 2021/05168-3)  and CNPq (grants 302682/2019-8, 304258/2018-0, and 408175/2018-4).}}

\author{
  E. G. Birgin\thanks{Department of Computer Science, Institute of
    Mathematics and Statistics, University of S\~ao Paulo, Rua do
    Mat\~ao, 1010, Cidade Universit\'aria, 05508-090, S\~ao Paulo, SP,
    Brazil. e-mails: egbirgin@ime.usp.br and tcmenezes@usp.br}
  \and
  A. Laurain\thanks{Department of Applied Mathematics, Institute of
    Mathematics and Statistics, University of S\~ao Paulo, Rua do
    Mat\~ao, 1010, Cidade Universit\'aria, 05508-090, S\~ao Paulo, SP,
    Brazil. e-mail: laurain@ime.usp.br}
  \and
  T. C. Menezes\footnotemark[2]
}

\date{December 9, 2021}

\maketitle

\begin{abstract}
Minimization diagrams encompass a large class of diagrams of interest in the literature, such as generalized Voronoi diagrams. We develop an abstract perturbation theory and perform a sensitivity analysis for functions depending on sets defined through intersections of smooth sets, and formulate precise conditions to avoid singular situations. This allows us to define a general framework for solving optimization problems  depending on minimization diagrams. The particular case of Voronoi diagrams is discussed to illustrate the general theory. A variety of numerical experiments is presented. The experiments include constructing Voronoi diagrams with cells of equal size, cells satisfying conditions on the relative size of their edges or their internal angles, cells with the midpoints of pairs of Voronoi and Delaunay edges as close as possible, or cells of varying sizes governed by a given function. Overall, the experiments show that the proposed methodology allows the construction of customized Voronoi diagrams using off-the-shelf well-established optimization algorithms.

\vspace{2mm}
\noindent
\textbf{Keywords:}
minimization diagrams, generalized Voronoi diagrams, nonsmooth shape optimization.

\vspace{2mm}
\noindent
\textbf{AMS subject classification:} 49Q10, 49J52, 49Q12
\end{abstract}


\section{Introduction}

Let $A\subset\R^2$ be an open and bounded set, $\mL = \{1,\dots,\ck\}$ a set of indices,  $\ba = \{\si_i\}_{i\in\mL}$  a set of so-called {\it sites} with $\si_i\in\R^q$, and $\blsf =\{\lsf_i\}_{i\in\mL}$ a set of smooth functions $\lsf_i:\R^2\times \R^q\to\R$. 
Define
$$V_i(\bsi):= \inte\left\{x \in A \text{ such that } \lsf_i(x,a_i) = \min_{k\in\mL}  \lsf_k(x,a_k)  \right\},$$
where $\inte S$ denotes the interior of  $S\subset\R^2$.
The set $\mv(\bsi) := \{V_i(\bsi)\}_{i\in\mL}$ is called {\it minimization diagram} and the sets $V_i(\bsi)$ are called {\it cells} of the diagram.

Minimization diagrams were introduced in \cite{MR824106} and generalize a large class of diagrams of interest in the literature.
They include various types of generalized Voronoi diagrams as particular cases, such as  Euclidean Voronoi diagrams \cite{MR1722997},  power diagrams \cite{MR3419889}, M\"obius diagrams, Apollonius diagrams \cite{Boissonnat2006,wormser:tel-00410850},  multiplicatively weighted Voronoi
diagrams \cite{MR765066} and anisotropic diagrams  \cite{10.5555/3061451.3061482,Budninskiy}.
For the computation of minimization diagrams, we refer to~\cite{MR3041226} and the references therein.
The concept of {\it abstract Voronoi diagrams} has also been introduced in~\cite{MR1036520} where the Voronoi cells are viewed as intersections of regions rather than defined via distance functions. 
The bulk of the literature on this topic is mainly focused on studying the theoretical properties of specific types of generalized Voronoi diagrams and on their efficient computation.

The optimization of Voronoi diagrams has several important applications such as grid generation and optimization in the framework of the finite element method.
In this context,  the optimization usually consists in obtaining  centroidal Voronoi tessellations, see the reviews \cite{MR1722997,NMTMA-3-119} and the references therein; see also~\cite{sieger} for alternative approaches.
Other applications include land-use optimization \cite{Suppakitpaisarn} and inverse problems \cite{Bourne}.
In some cases, the optimization of Voronoi diagrams is based on a sensitivity analysis, which  has been performed in the literature  for specific classes of energies and minimization diagrams such as  centroidal Voronoi tessellation functions \cite{MR1722997}, centroidal power diagrams \cite{MR3419889} and for an inverse problem for Voronoi diagrams in \cite{Bourne}.
The sensitivity analysis developed in the present paper widely generalizes these approaches and provides a rigorous mathematical construction of the bi-Lipschitz mappings required for integration by substitution.
This construction process is key to determine sufficient conditions to avoid singular cases, and to enable the calculation of derivatives of any order and any type of cost functions.  
For instance, in \cite{coveringsecond}, a similar construction allowed to compute second-order derivatives of cost functions defined as domain integrals.

The theoretical part of the present work is structured in three layers of abstraction.
In the first, most abstract layer, a perturbation theory for sets defined as the intersection of subzero level sets of functions is presented and natural conditions to avoid singular situations are provided.
In the second layer, this  theory is applied to obtain a perturbation theory for functions depending on minimization diagrams.
In the third layer, the particular case of Euclidean Voronoi diagrams is discussed;
this serves as an illustration and enables a better understanding of the abstract concepts of the first two layers.
In the first layer, the main result consists in the construction of a bi-Lipschitz mapping between a reference set and the perturbed domain, both defined as intersections of smooth sets.
In \cite{coveringsecond,coveringfirst} a similar but simpler situation has been investigated, where  a  bi-Lipschitz mapping was built to model the small perturbation of sets defined as a union of balls. 
The main ideas of \cite{coveringsecond,coveringfirst}  for building such mapping are generalized here to the much larger class of  sets defined as subzero level sets of smooth functions.
The obtained bi-Lipschitz mapping is a key tool for applying {\it shape calculus} and {\it shape optimization} techniques 
\cite{MR2731611,MR3791463,MR1215733} to compute the shape sensitivity of cost functionals defined as integrals.
Indeed, the calculation of the derivatives of integrals on moving domains requires a change of variables employing this mapping.
The main challenge here is to handle the nonsmoothness of sets defined via intersections. 
In this sense the present work contributes to advance the theory of {\it nonsmooth shape optimization} \cite{LAURAIN2020328,MR3535238}.

The function $x\mapsto \min_{k\in\mL}  \lsf_k(x,a_k)$ is called lower envelope of the set of functions $\blsf$.  
In~\cite{laurainLEM}, a lower-envelope-based numerical method has been developed, and it was shown that this method generalizes the level set method~\cite{MR965860}.  
The theory developed in the present paper shares similarities with  the theory developed in~\cite{laurainLEM}, both being based on a lower envelope approach, but distinguishes itself from~\cite{laurainLEM} in several key aspects. 
Indeed, \cite{laurainLEM} can be seen as a study of time-dependent minimization diagrams via transport equations, aimed at the  tracking of
interfaces motion in multiphase problems, while the present work is a study of the dependence  on the sites $\bsi$ of the implicit interfaces of the diagram cells.
In this sense, these two studies are complementary and contribute to build an abstract theory  of evolving minimization diagrams.
It is interesting to observe that, perhaps unexpectedly, this theory encompasses both the level set method and the optimization of Voronoi diagrams as particular cases, see~\cite{laurainLEM}.

From a practical point of view and by way of illustration, this paper applies the developed theory to the construction of Voronoi diagrams satisfying pre-specified properties. 
The experiments show that it is possible to formulate a priori the desired properties as differentiable functions and that Voronoi diagrams can be obtained by minimizing one or more desirable metrics simultaneously. Moreover, experiments show that the optimization process can be performed using well-established and available optimization methods.

The rest of this work is organized as follows. 
In Section~\ref{sec:abstract}, a perturbation theory for sets defined as the intersection of subzero level sets of functions is described. 
This theory is applied to minimization diagrams in Section~\ref{sec:MinDiags}, and then to the particular case of Voronoi diagrams in Section~\ref{sec:vordiag}.
Section~\ref{numexp} presents numerical experiments for the particular case of Voronoi diagrams.
The calculations of the gradients of the functions used in these numerical experiments are detailed in Appendix~A.
A brief discussion included in Section~\ref{discussion} analyzes alternatives and options that remained unexplored in the computational experiments. Conclusions and lines for future research are provided in the last section.

\begin{notation}
$\|\cdot\|$ denotes the Euclidean norm. 
Given $x, y \in \R^n$, $x \cdot y = x^\top y \in \R$; while $x \otimes y = x y^\top \in \R^{n \times n}$.
We use $y^\perp := R y$, for a vector $y\in\R^2$, where $R$ is a rotation matrix of angle $\pi/2$ with respect to a counterclockwise orientation.  
The transpose of a matrix $M$ is denoted $M^\top$, and $\rank M$ is the rank of $M$, i.e., the  maximum number of linearly independent rows or columns of $M$.
For a finite set $\I$, $|\I|$ denotes the cardinal of $\I$.
For a set $S\subset\R^2$, $\dim S$ denotes its dimension, $\inte S$ its interior, $\overline{S}$ its closure, $|S|$ its perimeter if $S$ is one-dimensional or its area if $S$ is two-dimensional.
We use $B(x,r)$ to denote an open ball of center $x$ and radius $r$.
The gradient with respect to $x\in\R^2$ of a function $\psi:\R^2\to\R$ is denoted $\nabla_x\psi$ and is a column vector.
The divergence of  a  sufficiently smooth vector field $\R^2\ni (x_1,x_2)\mapsto \psi(x_1,x_2)=(\psi_1(x_1,x_2),\psi_2(x_1,x_2))\in\R^2$ is defined by $\divv \psi := \frac{\partial \psi_1}{\partial x_1} + \frac{\partial \psi_2}{\partial x_2}$, and its Jacobian matrix is denoted $D_x \psi$.
The gradient  with respect to $\bsi$ of a function $G:\R^{q\ck}\to\R$ is denoted $\nabla G$.
The Jacobian matrix with respect to $\bsi$ of a function $G:\R^{q\ck}\to\R^n$ is denoted $DG$.
\end{notation}
 
\section{Perturbation theory for sets defined as intersections}\label{sec:abstract}

Given a perturbation $\delta\bsi$ of the sites $\bsi$, our main objective is to build a bi-Lipschitz transformation $T(\cdot,t)$
that maps $V_i(\bsi)$ to the perturbed cell $V_i(\bsi+t\dbsi)$.
In order to handle the constraint $V_i(\bsi)\subset A$, it is convenient to first build  a perturbation theory for sets defined as intersections and for functions depending on a pseudo-time $t$.
For this purpose we use and extend the results of \cite{coveringsecond,coveringfirst,MR3840889,laurainLEM}.
The theory developed here share several similarities with the theory developed in \cite{laurainLEM}.
Indeed, in \cite{laurainLEM} the ``phases'', corresponding to the cells here, are also defined by a minimization diagram. 
A key difference is that in \cite{laurainLEM}, the function $\blsf$ itself corresponds to the control parameter, whereas in the present work the set  $\bsi$ of sites is the control. 
Thus, unlike in \cite{laurainLEM}, we need here to express the perturbation of vertices and edges  (interfaces between cells) in terms of the perturbation $\dbsi$.
Still, several results from  \cite{laurainLEM} can be used or adapted to the present framework.

Let $\K\subset\mathds{N}$ be a finite set of indices  and  $\mathds{I}^r : = \{\I\subset\K\ |\ \card =r\}$. 
Let $\{\hlsf_k\}_{k\in \K}$ be a set of given functions in $\C^\infty(\R^2\times\R,\R)$.
For a subset of indices $\I = \{k_1,k_2, \dots , k_\card\}\subset\K$,
define
\begin{align}\label{def:hblsf}
\hblsf_\I &:= (\hlsf_{k_1}, \hlsf_{k_2},\dots ,\hlsf_{k_\card})^\top\in \C^\infty(\R^2\times\R,\R^{\card}).
\end{align}  

\begin{definition}\label{def2}
For $k\in\K$ and $\I\subset\K$, define
\begin{align}
\label{def:omk}\om_{k}(t) &  := \inte\{x\in \R^2\ |\ \hlsf_{k}(x,t) \leq 0 \},\\
V_{\K}(t) &:= \bigcap_{k\in\K} \om_k(t),\\
\label{gamma_kt}\gamma_k(t) & := \{x\in \R^2\ |\ \hlsf_k(x,t) = 0, \hlsf_j(x,t)<0 \text{ for all } j\in \K\setminus\{k\}  \},\\
\label{eq:malpha0} \mM_{\I}(t) &  := \{x\in \R^2\ |\ \widehat\blsf_\I(x,t) = 0 \},\\
\label{def:mmr}\mM^r(t) &:= \bigcup_{\I\in\mathds{I}^r}\mM_\I(t).
\end{align}
\end{definition}
The goal of this section is to build a bi-Lipschitz mapping satisfying $T(V_{\K},t) = V_{\K}(t)$; in the next sections $V_{\K}(t)$ will play the role of the perturbed cell $V_i(\bsi+t\dbsi)$.
The sets  $\mM_{\I}(t)$ and $\mM^r(t)$ are introduced to study the behavior of the edges and vertices formed by the intersections of the boundaries $\partial\omega_k(t),k\in\K$, and play an important role for the formulation of conditions to avoid singularities in the perturbation theory developed here; see Assumptions~\ref{assump2a} and~\ref{assump2}.
For the sake of simplicity, we use the notation  $\om_k:=\om_k(0)$, $\gamma_k :=\gamma_k(0)$, $V_{\K} =V_{\K}(0)$, $\mM_{\I}:=\mM_{\I}(0)$ and $\mM^r:= \mM^r(0)$.
\begin{assump}\label{assump:bounded}
There exists $\tau_1>0$, $k\in\K$ and an open ball $B\subset\R^2$ such that  $\overline{\om_k(t)}\subset B$ for all $t\in[0,\tau_1]$. 
\end{assump}

The purpose of Assumption~\ref{assump:bounded} is to use  
the uniformly bounded set $\om_k(t)$  to represent the bounded set $A$ in Sections~\ref{sec:MinDiags} and~\ref{sec:vordiag}.
Indeed, the sets defined in \eqref{def:omk} need not be bounded  in general. 
For instance in the particular case of Voronoi diagrams one chooses $\hlsf_{k}(x,t) = \|x_k - (\si_k+t\dsi_k)\|^2 - \|x_i - (\si_i+t\dsi_i)\|^2$ and  $\om_{k}(t)$ is a half-plane; see Section~\ref{sec:vordiag}.

We observe that in general the sets $\om_k(t),k\in\K$ may have nonempty intersections, which is undesirable in applications.
This may actually happen when a set $\mM_\I(t)$ is ``thick'', in the sense that $\dim\mM_\I(t)>1$.
This can be avoided using appropriate conditions that we describe now.

\begin{assump}\label{assump2a}
$\|D_x\widehat\blsf_{\I}(x,0)\| > 0$ for all $x\in\mM_{\I}$ and for all $\I\in\mathds{I}^1$.
\end{assump}
\begin{assump}\label{assump2}
We have $\rank D_x\widehat\blsf_{\I}(x,0)=2$  for all $x\in \mM_\I$ and for all $\I\in\mathds{I}^2$,
and
\begin{align}\label{199}
\mM^3 = \emptyset.
\end{align} 
\end{assump}

\begin{remark}\label{assump_grad_f}
In Assumption~\ref{assump2}, the condition $\rank D_x\widehat\blsf_{\I}(x,0)=2$   for all $x\in \mM_\I$ and for all $\I\in\mathds{I}^2$ is equivalent to 
$\nabla_x \hlsf_j(x,0)^\perp\cdot \nabla_x \hlsf_k(x,0) \neq 0 \text{ for all } x\in \mM_\I \text{ and for all } \I =\{j,k\}\in\mathds{I}^2$.
\end{remark}

Lemma~\ref{lemma003} and Lemma~\ref{lemma003b} below are  straightforward extensions of \cite[Lemma~2]{laurainLEM} and \cite[Lemma~4]{laurainLEM}, respectively, therefore we omit the proof here.
Note however that the definition of $\hblsf_\I$ in \cite{laurainLEM} is slightly different from the definition in \eqref{def:hblsf}, thus the results of \cite[Lemma~2]{laurainLEM} and \cite[Lemma~4]{laurainLEM} need to be adapted to the notation in the present paper.

\begin{lemma}\label{lemma003}
Suppose $|\K|\geq 3$, $B\subset\R^2$ is an open ball,  and  Assumption~\ref{assump2a} holds. 
Then there exists $\tau_1>0$  such that for all $\I  \in \mathds{I}^1$, $\mM_{\I}(t)\cap B$ is either empty or is a one-dimensional, $\C^\infty$-manifold for all $t\in [0,\tau_1]$.
\end{lemma}

\begin{lemma}\label{lemma003b}
Suppose $|\K|\geq 3$, $B\subset\R^2$ is an open ball,  and Assumption~\ref{assump2} holds. 
Then there exists $\tau_1>0$ such that for all $\I  \in \mathds{I}^2$, $\mM_{\I}(t)\cap B$ is either empty or a set of isolated points for all $t\in [0,\tau_1]$.
\end{lemma}

\begin{lemma}\label{lemma02}
Suppose that Assumptions~\ref{assump:bounded} and~\ref{assump2} hold and let  $\I\in\mathds{I}^2$.
Then there exists $\tau_1>0$ such that for all $v\in \mM_\I\cap \overline{V_{\K}}$ there exists  a unique smooth function $z_v:[0,\tau_1]\to \R^{2}$ such that $z_v(0)=v$  and
\begin{align}\label{MI:dec}
\mM_\I(t) \cap \overline{V_{\K}(t)}
= \bigcup_{v\in \mM_\I\cap \overline{V_{\K}}} \{z_v(t)\} \text{ for all }t\in [0,\tau_1]. 
\end{align}
In addition we have
\begin{equation}\label{der:zrho}
z_v'(0)  = - D_x \hblsf_\I(v,0)^{-1}\partial_t \hblsf_\I(v,0) \text{ for all }v\in \mM_\I\cap \overline{V_{\K}}.
\end{equation}
\end{lemma}
\begin{proof}
Due to Assumption~\ref{assump:bounded}, there exists $k\in\K$ such that $\om_k(t)$ is uniformly bounded for all $t\in[0,\tau_1]$, hence $V_{\K}(t)$ is uniformly bounded. 
Thus, in view of Lemma~\ref{lemma003}, $\mM_\I\cap\overline{V_{\K}}$ is a finite set of points.
We have $\hblsf_\I(v,0) = (0,0)^\top$ for all $v\in \mM_\I\cap\overline{V_{\K}}$.
Using Assumption \ref{assump2}, we get that $D_x  \hblsf_\I(v,0)$ is invertible for all $v\in\mM_\I\cap\overline{V_{\K}}$. 
Thus we can apply the implicit function theorem, and for each $v\in \mM_\I\cap\overline{V_{\K}}$ this yields a unique smooth function $z_v:[0,\tau_1]\to\R^2$ such that $z_v(0)=v$ and $\hblsf_\I(z_v(t),t)=(0,0)^\top$ for all $t\in [0,\tau_1]$.
This proves \eqref{MI:dec}.

Since $\hblsf_\I(z_v(t),t)=(0,0)^\top$ for all $t\in [0,\tau_1]$ we get
$$ \partial_t \hblsf_\I(v,0) + D_x \hblsf_\I(v,0) z_v'(0) = (0,0)^\top$$
and then 
$ 
z_v'(0)  = - D_x \hblsf_\I(v,0)^{-1}\partial_t \hblsf_\I(v,0)
$.
\end{proof}

\begin{lemma}\label{lem3a}
Suppose that Assumptions~\ref{assump:bounded} and~\ref{assump2} hold.
Then there exists $\tau_1>0$ and $r>0$ such that 
\begin{align}\label{Z:carac}
\mM^2(t) \cap \overline{V_{\K}(t)} = \bigcup_{v\in \mM^2\cap \overline{V_{\K}}} \{z_v(t)\} \text{ for all }t\in [0,\tau_1],
\end{align}
with $z_v(t)$ given by Lemma~\ref{lemma02}, $z_v(t)\in B(v,r)$ and $B(v,r)\cap B(w,r)=\emptyset$ for all $\{v,w\}\subset \mM^2\cap \overline{V_{\K}}$.
\end{lemma}
\begin{proof}
The functions $z_v$ in \eqref{MI:dec} depend in principle on $\I$. However, we can show that to each $v\in\mM^2 \cap \overline{V_{\K}}$ can be associated a unique function  $z_v$ using Assumption~\ref{assump2}. 
Indeed let $v\in\mM^2 \cap \overline{V_{\K}}$, then there is a unique $\I\in \mathds{I}^2$ such that  $v\in\I$, otherwise we would have $v\in \mM_\I\cap \mM_{\widetilde\I}\cap \overline{V_{\K}} = \mM_{\I\cup\widetilde \I}\cap \overline{V_{\K}}$ for some $\widetilde \I\neq \I$, $\widetilde\I\in \mathds{I}^2$, but this would contradict \eqref{199} since $\I\cup\widetilde \I\in\mathds{I}^r, r\geq 3$. 
Thus, the functions $z_v$ in \eqref{MI:dec} are actually independent  on $\I\in \mathds{I}^2$, and using \eqref{MI:dec} and definition~\eqref{def:mmr} we may write
$$ \mM^2(t) \cap \overline{V_{\K}(t)} = \bigcup_{\I\in\mathds{I}^2}\bigcup_{v\in \mM_\I\cap \overline{V_{\K}}} \{z_v(t)\} \text{ for all }t\in [0,\tau_1],$$
which yields \eqref{Z:carac}.

According to Lemma~\ref{lemma003} and since there exists $k\in\K$ such that $\om_k$ is bounded due to Assumption~\ref{assump:bounded}, 
$\mM^2 \cap \overline{V_{\K}}$ is finite. 
Thus we can choose $\tau_1$ sufficiently smooth so that there exists $r>0$ with the property  $z_v(t)\in B(v,r)$ for all $v\in \mM^2\cap \overline{V_{\K}}$ and $B(v,r)\cap B(w,r)=\emptyset$ for all $\{v,w\}\subset \mM^2\cap \overline{V_{\K}}$.
This yields the result.
\end{proof}
\begin{remark}
In \eqref{Z:carac},  $\mM^2(t) \cap \overline{V_{\K}(t)}$ is the set of vertices of $\overline{V_{\K}(t)}$ and  $z_v(t), v\in \mM^2\cap \overline{V_{\K}}$, are the vertices. This shows that $\mM^2(t) \cap \overline{V_{\K}(t)}$ for sufficiently small $t$, in the sense that the number of vertices stays constant.
\end{remark}

We now state a Lemma which provides a decomposition of the boundary of the cell $V_{\K}(t)$ into edges $\gamma_k(t)$ and vertices  $\mM^2(t) \cap \overline{V_{\K}(t)}$. Note that the properties  $\gamma_k(t)\subset \partial\om_k(t)$ in Lemma~\ref{lem3b} is not true in general and requires Assumption~\ref{assump2a}, otherwise the dimension of  $\gamma_k(t)$ could be greater than one.

\begin{lemma}\label{lem3b}
Suppose that Assumptions~\ref{assump:bounded},~\ref{assump2a} and \ref{assump2} hold. Then there exists $\tau_1>0$ so that,  for all $k\in\K$, $\gamma_k(t)\subset \partial\om_k(t)$, $\gamma_k(t)$ is uniformly bounded on $[0,\tau_1]$ and is a finite union of open, smooth, connected arcs.
In addition, $V_{\K}(t)$ is Lipschitz  and 
\begin{align}\label{Vibd:carac_b}
\partial  V_{\K}(t) 
= \bigcup_{k\in\K} \overline{\gamma_k(t)}
= (\mM^2(t) \cap \overline{V_{\K}(t)})  \cup\bigcup_{k\in\K} \gamma_k(t) \text{ for all } t\in [0,\tau_1].
\end{align}
\end{lemma}
\begin{proof}
Due to Assumption~\ref{assump:bounded}, there exists an open ball $B\subset\R^2$ such that $\gamma_k(t)\subset B$  for all $k\in\K$ and all $t\in [0,\tau_1]$.
Using Assumption~\ref{assump2a} we get $\partial\om_k(t) = \{x\in \R^2\ |\ \hlsf_{k}(x,t) = 0 \} = \mM_{\{k\}}(t)$ and Lemma~\ref{lemma003} yields  $\dim (\mM_{\{k\}}(t)\cap B)=1$ or $\mM_{\{k\}}(t)\cap B =\emptyset$ for all $t\in [0,\tau_1]$.
Thus $\gamma_k(t)\subset \partial\om_k(t)$ in view of \eqref{gamma_kt}.  
The boundary of $\gamma_k(t)$, relatively to $\partial\om_k(t)$, is included in $\mM_{\I}(t)$ for some  $\I\in\mathds{I}^2$ with $\I\ni k$. 
Due to Assumption~\ref{assump2} and Lemma~\ref{lemma003b}, the boundary of $\gamma_k(t)$, relatively to $\partial\om_k(t)$, is a finite set of points, thus $\gamma_k(t)$ is a finite union of open, smooth and connected arcs.

Now we prove the first equality in \eqref{Vibd:carac_b}.
Let $x\in \partial  V_{\K}(t)$, then we must have $\hlsf_k(x,t) = 0$ for some $ k\in\K$ and $\hlsf_j(x,t)\leq 0$ for all  $j\in \K\setminus\{k\}$, otherwise we would have $\hlsf_k(x,t) < 0$ for all $k\in\K$ which would imply $x\in V_{\K}(t)$.  This would be a contradiction since $V_{\K}(t)$ is open.
Since 
\begin{equation}\label{bar_gammak}
\overline{\gamma_k(t)}= \{x\in \R^2\ |\ \hlsf_k(x,t) = 0, \hlsf_j(x,t)\leq 0 \text{ for all } j\in \K\setminus\{k\}  \}, 
\end{equation}
we have $x\in \overline{\gamma_k(t)}$   and this proves $\partial  V_{\K}(t) \subset \bigcup_{k\in\K} \overline{\gamma_k(t)}$. 

Reciprocally, let $x\in \overline{\gamma_k(t)}$ for some $ k\in\K$, then $x\in \overline{V_\K(t)}$ by definition of $\gamma_k(t)$. 
Further, if $x\in V_{\K}(t)$ we would have $B(x,r)\subset V_{\K}(t)$, for some $r>0$, and consequently $\hlsf_k(y,t) \leq 0$ for all $y\in B(x,r)$.
Since $\hlsf_k(x,t) = 0$ due to \eqref{bar_gammak}, we must have $\nabla_x \hlsf_k(x,t) = 0$ which contradicts Assumption~\ref{assump2a} for $t\in [0,\tau_1]$ and $\tau_1$ sufficiently small.
Thus $x\in \partial V_{\K}(t)$ and this proves the first equality in \eqref{Vibd:carac_b}.

Then we prove the following result
\begin{align}\label{018}
\bigcup_{k\in\K} \overline{\gamma_k(t)} 
& = \left(\bigcup_{\I =\{k_1,k_2\}\in \mathds{I}^2} \overline{\gamma_{k_1}(t)} \cap \overline{\gamma_{k_2}(t)}\right) \cup\bigcup_{k\in\K} \gamma_k(t)
=\left(\bigcup_{\I\in \mathds{I}^2} \mM_\I(t)\cap \overline{V_{\K}(t)}\right)
\cup\bigcup_{k\in\K} \gamma_k(t).
\end{align}
We start with the first equality in \eqref{018}.
Suppose $x\in \bigcup_{k\in\K} \overline{\gamma_k(t)}\setminus \bigcup_{k\in\K} \gamma_k(t)$, then, in view of \eqref{bar_gammak}, $\hlsf_{k_1}(x,t) = 0$ and $\hlsf_{k_2}(x,t) = 0$ for some $\{k_1,k_2\}\subset\K$, which proves $\bigcup_{k\in\K} \overline{\gamma_k(t)}\setminus \bigcup_{k\in\K} \gamma_k(t)\subset\bigcup_{\I =\{k_1,k_2\}\in \mathds{I}^2} \overline{\gamma_{k_1}(t)} \cap \overline{\gamma_{k_2}(t)}$. 
The other inclusion is clear.

Now we prove the second equality in \eqref{018}.
Let $x\in \overline{\gamma_{k_1}(t)} \cap \overline{\gamma_{k_2}(t)}$, then $x\in \mM_{\{k_1,k_2\}}(t)$ in view of \eqref{bar_gammak}.
Using the first equality in \eqref{Vibd:carac_b}, we also have $x\in \partial  V_{\K}(t)$ and this yields the first inclusion.
Reciprocally, suppose $x\in \mM_\I(t)\cap \overline{V_{\K}(t)}$ for some $\I=\{k_1,k_2\}\in \mathds{I}^2$, then $\hlsf_{k_1}(x,t) = 0$,  $\hlsf_{k_2}(x,t) = 0$, therefore $\overline{\gamma_{k_1}(t)} \cap \overline{\gamma_{k_2}(t)}$, which proves the other inclusion.

Using \eqref{018} and
\begin{align*}
\left(\bigcup_{\I\in \mathds{I}^2} \mM_\I(t)\cap \overline{V_{\K}(t)}\right)
\cup\bigcup_{k\in\K} \gamma_k(t) 
&=(\mM^2(t) \cap \overline{V_{\K}(t)})  \cup\bigcup_{k\in\K} \gamma_k(t)
\end{align*}
proves the second equality in \eqref{Vibd:carac_b}.

Now we prove that $V_{\K}(t)$ is Lipschitz.
Recall that $V_{\K}(t)$ is Lipschitz if $\partial V_{\K}(t)$ is, in a neighborhood of each of its
points, the graph of a Lipschitz function and $V_{\K}(t)$ is only on one side of its boundary.
Let $x\in \partial V_{\K}(t)$. 
In view of \eqref{Vibd:carac_b}, either $x\in \mM^2(t) \cap \overline{V_{\K}(t)}$ or $x\in \gamma_k(t)$ for some $k\in\K$.
If $x\in \gamma_k(t)$, then we can use the function $\hlsf_k$ to locally describe $\partial V_{\K}(t)$ as the graph of a Lipschitz function, and $V_{\K}(t)$ is only on one side of its boundary since  $\overline{V_{\K}(t)}$ satisfies $\hlsf_{k}(\cdot,t) \leq 0$ in a neighborhood of $x\in\gamma_k(t)$.

Now suppose $x\in \mM^2(t) \cap \overline{V_{\K}(t)}$, i.e., $x$ is a vertex of $\overline{V_{\K}(t)}$.
Then $x\in\overline{\gamma_j(t)} \cap \overline{\gamma_k(t)}$ for some $\{j,k\}\in\mathds{I}^2$.
Since $\partial V_{\K}(t)$ is smooth on both sides of $x$, one just needs to check that the tangent vectors to $\overline{\gamma_j(t)}$ and $\overline{\gamma_k(t)}$ are not collinear at $x$.
If the tangent vectors were collinear, then the normal vectors to $\overline{\gamma_j(t)}$ and $\overline{\gamma_k(t)}$ would also be collinear and this would contradict the condition $\rank D_x\widehat\blsf_{\I}(x,0)=2$ of Assumption~\ref{assump2} (see Remark~\ref{assump_grad_f}), for sufficiently small $\tau_1$.
This shows that $V_{\K}(t)$ is Lipschitz.
\end{proof}

\begin{lemma}\label{lem4}
Suppose that Assumptions~\ref{assump:bounded}, \ref{assump2a} and \ref{assump2} hold.
Then there exists $\tau_1>0$ and a continuous function $T:\partial V_{\K}\times [0,\tau_1]\to \R^2$
such that 
$$T(\gamma_k,t) = \gamma_k(t) \text{ for all } k\in\K \text{ and } T(\partial V_{\K},t) = \partial V_{\K}(t). $$
In addition,  $T(\cdot,t)$  is Lipschitz with constant $1+Ct$ for all $t\in [0,\tau_1]$, where $C$ is independent of $t$.
\end{lemma}
\begin{proof}
Let $k\in\K$ such that $\gamma_k\neq \emptyset$; if $\gamma_k$ is empty for all $k\in\K$ then the result follows trivially.
Using Assumption~\ref{assump2a} and Lemma~\eqref{lem3b} we get $\gamma_k(t)\subset \partial\om_k(t)$ for all $t\in [0,\tau_1]$.
We need to separate two cases, the case where the boundary of $\gamma_k$, relatively to $\partial\om_k$, is empty, and the case where it is not empty. 

Suppose first that the boundary of $\gamma_k$ is not empty.
We have  $\overline{\gamma_k(t)}\subset B$ due to Assumption~\ref{assump:bounded}. 
In view of \eqref{Z:carac}, let $z_v(t),z_w(t)\in\mM^2(t) \cap \overline{V_{\K}(t)}$ be two consecutive vertices of $\overline{\gamma_k(t)}$ with respect to a counterclockwise orientation on $\partial V_{\K}(t)$, and write $z_v := z_v(0)$, $z_w := z_w(0)$ for simplicity.
Then the vertices $z_v(t),z_w(t)$ define a unique connected and relatively open (with respect to $\partial\om_k(t)$) subarc $\gamma(t)\subset\gamma_k(t)$; we will write $\gamma:=\gamma(0)$ for simplicity.
Let $U\subset \R$ be an open interval,  $\xi: U\to \partial\om_k$ be a smooth parameterization of $\partial\om_k$, and $\{s_v,s_w\}$ be such that $\xi(s_v) = z_v$ and $\xi(s_w) = z_w$, $[s_v,s_w]\subset U$ and  $\xi|_{[s_v,s_w]}$ is a parameterization of $\overline{\gamma}$. 
Let $P$ be the projection onto $\partial\om_k$.
For sufficiently small $\tau_1$, this projection is unique for all $x\in\partial\om_k(t)\cap B$, where $B$ is the ball given by Assumption~\ref{assump:bounded}.
Define 
$\lambda(s) := \frac{s - s_v}{s_w - s_v}$, $s(x):=\xi^{-1}(x)$, 
$s_v(t) := s(P(z_v(t)))$, $s_w(t) := s(P(z_w(t)))$,
$\sigma(s,t):=\lambda(s)s_w(t) + (1-\lambda(s))s_v(t)$ 
and 
\begin{align}\label{eq:beta}
\beta(x,t) := \xi(\sigma(s(x),t)) \in\overline{\gamma}.
\end{align}
Note that $\sigma(s,0)=s$, thus $\beta(x,0) = \xi(s(x))=x$.
It can also be checked that  $\beta(z_v,t) = P(z_v(t))$ and $\beta(z_w,t) = P(z_w(t))$.

Since $\gamma_k\subset\partial\om_k$, if $\partial\om_k\cap B=\emptyset$ then $\gamma_k$ is empty.
If $\partial\om_k\cap B$ is not empty, then  according to \cite[Lemma~3.1]{MR3840889}, and using Assumption~\ref{assump:bounded}, there exists   $\hat\alpha\in\C^{\infty}(\partial\om_k\cap B\times [0,\tau_1],\R)$ satisfying $\hat\alpha(y,0) = 0$ for all $y\in \partial\om_k\cap B$ and
\begin{align}\label{117}
\hlsf_k(y + \hat\alpha(y,t)\nabla_x \hlsf_k(y,0),t)=0 \text{ for all } y\in \partial\om_k\cap B \text{ and } t\in [0,\tau_1]. 
\end{align}
Note that  $\nabla_x\hlsf_k(y,0)$ is normal to $\partial\om_k$.
For $\tau_1$ sufficiently small, we have $\beta(x,t)\in \partial\om_k\cap B$ for all $x\in\gamma$ and all $t\in [0,\tau_1]$.
Thus we can define 
\begin{equation}\label{eq:S}
T(x,t) := \beta(x,t) + \hat\alpha(\beta(x,t),t)\nabla_x \hlsf_k(\beta(x,t),0) \text{ on } \gamma\times [0,\tau_1]. 
\end{equation}
Let us also define $\mS:[s_v,s_w]\times [0,\tau_1]\to \R^2$ as $\mS(s,t): = T(\xi(s),t) - \xi(s)$.
Since $P,z_v,z_w,\hat\alpha,\hlsf_k$ are smooth functions on their domain of definition, then by composition  $s_v,s_w,\xi,\lambda,\sigma,\beta$ and $\mS$ are also smooth on their domain of definition.
Also we have $\mS(s,0) = T(\xi(s),0) - \xi(s)= \beta(\xi(s),0) - \xi(s) = 0$ for all $s\in [s_v,s_w]$. Thus, a Taylor expansion provides
$$\partial_s \mS(s,t) = \partial_s \mS(s,0) + t \partial_s\partial_t \mS(s,\eta)=  t \partial_s\partial_t \mS(s,\eta) \text{ with } |\eta|<t.$$
Using the smoothness of $\mS$, this yields
\begin{equation}\label{mS:est}
 \|\partial_s \mS(s,t)\| \leq ct \text{ for all } t\in [0,\tau_1] \text{ and } s\in [s_v,s_w], 
\end{equation}
for some constant $c$ independent of $s$ and $t$.

Now we show that~\eqref{mS:est} implies the existence of a constant $C>0$ such that $x\mapsto S(x,t):=T(x,t) -x$ is Lipschitz on $\overline{\gamma}$ with Lipschitz constant $C t$, i.e.,\
\begin{equation}\label{eq:Slipschitz}
 \|S(x,t) - S(y,t)\| \leq Ct \|x-y\|,\quad  \forall (t,x,y)\in [0,\tau_1]\times \overline{\gamma}^2. 
\end{equation}
Indeed if this were not the case, then there would exist a sequence $(t_n,x_n,y_n)\in  [0,\tau_1]\times \overline{\gamma}^2$ such that
\begin{equation}\label{eq:Sinf}
 \frac{\|S(x_n,t_n) - S(y_n,t_n)\|}{t_n \|x_n-y_n\|} \to\infty\quad \text{ as } n\to +\infty. 
\end{equation}
Suppose that \eqref{eq:Sinf} holds.
In view of~\eqref{eq:S} the numerator $\|S(x_n,t_n) - S(y_n,t_n)\|$ is uniformly bounded on $[0,\tau_1]\times \overline{\gamma}^2$, thus we must have $t_n \|x_n-y_n\|\to 0$.
We suppose  that both $t_n\to 0$ and $\|x_n-y_n\|\to 0$, the other cases follow in a similar way.
Using the compactness of $[0,\tau_1]\times \overline{\gamma}^2$, we can extract a subsequence, still denoted  $(t_n,x_n,y_n)$ for simplicity, that converges towards $(0,x^\star,x^\star)\in [0,\tau_1]\times \overline{\gamma}^2$.
Then we write, recalling that $s(x)=\xi^{-1}(x)$ where $\xi|_{[s_v,s_w]}$ is a parameterization of $\overline{\gamma}$,
$$ \frac{\|S(x_n,t_n) - S(y_n,t_n)\|}{t_n \|x_n-y_n\|}  
=  \underbrace{\frac{\|\mS(s(x_n),t_n) - \mS(s(y_n),t_n)\|}{t_n \|s(x_n) - s(y_n)\|}}_{\text{ bounded using } \eqref{mS:est} \text{ at } s(x^\star)}\underbrace{\frac{\|s(x_n) - s(y_n)\|}{\|x_n-y_n\|}}_{\text{bounded}}.$$
This contradicts~\eqref{eq:Sinf} which proves~\eqref{eq:Slipschitz}.
This proves that $T(\cdot,t)$ is Lipschitz on $\overline{\gamma}$ with constant $1+Ct$ for all $t\in [0,\tau_1]$. 
Then the mapping $T$ is built in the same way on each connected subarc of $\gamma_k$.

Then, taking $y = \beta(x,t)$ in \eqref{117} we get $\hlsf_k(T(x,t),t)=0$ for all $x\in \gamma_k$ and $t\in [0,\tau_1]$.
This proves that $T(\gamma_k,t) \subset \partial\om_k(t)$.
Since $T(\cdot,t)$ is Lipschitz on $\gamma$ with constant $1+Ct$ for all $t\in [0,\tau_1]$, $T(\cdot,t)$ is invertible on $\gamma$ for sufficiently small $\tau_1$, thus $T(\cdot,t)$ is a homeomorphism from $\gamma$ onto $T(\gamma,t)$.

We also have 
$$T(z_v,t) = \beta(z_v,t) + \hat\alpha(\beta(z_v,t),t)\nabla_x \hlsf_k(\beta(z_v,t),0)
= P(z_v(t)) + \hat\alpha(P(z_v(t)),t)\nabla_x \hlsf_k(P(z_v(t)),0)$$ 
and  $T(z_w,t) = P(z_w(t)) + \hat\alpha(P(z_w(t)),t)\nabla_x \hlsf_k(P(z_w(t)),0)$.
In fact we can show $T(z_v,t) = z_v(t)$ and $T(z_w,t) = z_w(t)$ for all $t\in [0,\tau_1]$ and $\tau_1>0$ sufficiently small.
Indeed, suppose that $T(z_v,t_n) \neq z_v(t_n)$ for some sequence $t_n\to 0$.
Then, by definition of the projection $P$ we have
$$z_v(t_n) = P(z_v(t_n)) + \zeta(t_n)\nabla_x \hlsf_k(P(z_v(t_n)),0),$$ 
with $\zeta(t_n)\to 0$ as $n\to\infty$.
Also, we have $\hlsf_k(z_v(t_n),t_n) = \hlsf_k(T(z_v,t_n) ,t_n)=0$ due to \eqref{117} and $z_v(t_n)\in\overline{\gamma_k(t_n)}$.
Thus, there exists $\mu(t_n)$ with 
$$\min\{|\zeta(t_n)|,|\hat\alpha(P(z_v(t_n)),t_n)|\}\leq |\mu(t_n)|\leq \max\{|\zeta(t_n)|,|\hat\alpha(P(z_v(t_n)),t_n)|\}$$
such that  
$$\nabla_x\hlsf_k(P(z_v(t_n)) + \mu(t_n)\nabla_x \hlsf_k(P(z_v(t_n)),0),t_n)\cdot \nabla_x \hlsf_k(P(z_v(t_n)),0)= 0 .$$
Passing to the limit $t_n\to 0$ we get $z_v(t_n)\to z_v$, $P(z_v(t_n))\to z_v$, $\mu(t_n)\to 0$ and
$$\|\nabla_x\hlsf_k(z_v,0)\|^2= 0 ,$$
which contradicts Assumption~\ref{assump:bounded}.
Thus we have indeed $T(z_v,t) = z_v(t)$ and in a similar way $T(z_w,t) = z_w(t)$ for all $t\in [0,\tau_1]$ and $\tau_1>0$ sufficiently small.
Proceeding in the same way for each connected arc $\gamma(t)\subset\gamma_k(t)$, we obtain $T(\partial\gamma_k,t) = \partial\gamma_k(t)$ for all $t\in [0,\tau_1]$, where $\partial\gamma_k(t)$ denotes the boundary of $\gamma_k(t)$ relatively to $\partial\om_k(t)$.


Since $T(\cdot,t)$ is a homeomorphism on $\gamma$, $T(\gamma,t)$ is connected. 
Hence, as $T(\partial\gamma,t) = \partial\gamma(t)$, we must either have $T(\gamma,t)= \gamma(t)$ or $T(\gamma,t)\subset \inte (\gamma(t)^c)$, where the interior and complementary are relative to $\partial\om_k(t)$. 
In the latter case, we have either $T(\gamma,t)\subset \gamma_k(t)$ or $T(\gamma,t)\not\subset \gamma_k(t)$.
If $T(\gamma,t)\subset \gamma_k(t)$, then $T(\gamma,t)$ must be one of the connected component of $\gamma_k(t)$, hence we must have $T(\gamma,t)= \gamma(t)$ due to $T(\partial\gamma,t) = \partial\gamma(t)$.
The case $T(\gamma,t)\not\subset \gamma_k(t)$ is not possible; 
otherwise there would exist a $x\in\gamma$ such that  $\hlsf_j(x,0)<0$ and  $\hlsf_j(T(x,t),t)\geq 0$,  for some $j\in \K\setminus\{k\}$,  and passing to the limit $t\to 0$ we would get $\hlsf_j(x,0) = 0$, which contradicts $x\in\gamma$.
Thus we conclude that $T(\gamma,t) = \gamma(t)$.
Repeating the same argument for each connected component  $\gamma\subset\gamma_k$ we obtain $T(\gamma_k,t) = \gamma_k(t)$.

Now recall that we have supposed in the beginning of the proof that the boundary of $\gamma_k$ is not empty.
In the case where the boundary of $\gamma_k$ is empty, we define the mapping as
\begin{equation}\label{eq:S:empty_bd}
T(x,t) := x + \hat\alpha(x,t)\nabla_x \hlsf_k(x,0) \text{ on } \gamma_k\times [0,\tau_1], 
\end{equation}
i.e., we do not need to use $\beta$.
One can then prove in a similar way that $T(\cdot,t)$ is Lipschitz on $\overline{\gamma_k}$ and that  $T(\gamma_k,t) = \gamma_k(t)$.

Finally, using \eqref{Vibd:carac_b}, we get
$$ T(\partial V_{\K},t)
= T\left(\bigcup_{k\in\K} \overline{\gamma_k},t\right)
= \bigcup_{k\in\K}T\left(\overline{\gamma_k},t\right)
= \bigcup_{k\in\K}\overline{\gamma_k(t)}
= \partial V_{\K}(t). $$
Since  $T(\cdot,t)$ is Lipschitz on $\overline{\gamma_k}$ with constant $1+Ct$ for all $t\in [0,\tau_1]$, and by construction $T(\cdot,t): \partial V_{\K}\to  \partial V_{\K}(t)$ is continuous at the vertices of $\partial V_{\K}$, we obtain that $T(\cdot,t): \partial V_{\K}\to  \partial V_{\K}(t)$  is Lipschitz with constant $1+Ct$ for all $t\in [0,\tau_1]$.
This proves the result.
\end{proof}

\begin{thm}\label{thm01}
Suppose that Assumptions~\ref{assump:bounded}, \ref{assump2a} and \ref{assump2} hold.
Then there exists $\tau_1>0$ and a mapping $T:\overline{V_{\K}} \times [0,\tau_1] \to \R^2$ satisfying
$T(V_{\K},t) = V_{\K}(t)$, $T(\partial V_{\K},t) = \partial V_{\K}(t)$ and such that  $T(\cdot, t): \overline{V_{\K}} \to \overline{V_{\K}(t)}$ is  bi-Lipschitz   for all $t\in [0,\tau_1]$. 
In addition we have
\begin{align}
\label{der_normal} \theta(x)\cdot\nu(x) &
= -\frac{\partial_t \hlsf_k(x,0)}{\|\nabla_x\hlsf_k(x,0)\|}\text{ for all } x\in\gamma_k, \\
\label{der_tangential}\theta(z)\cdot\tau(z) &=  - (D_x \hblsf_\I(z,0)^{-1}\partial_t \hblsf_\I(z,0))\cdot \tau(z)\text{ for all } z\in\mM^2 \cap \overline{V_{\K}},
\end{align}
where $\theta := \partial_t T(\cdot,0)$, $\nu$ is the outward unit normal vector to $V_\K$, and $\tau$ is the tangent vector to $\partial V_{\K}$ with respect to a counterclockwise orientation. 
\end{thm}
\begin{proof}
Let $T:\partial V_{\K}\times [0,\tau_1]\to \partial V_{\K}(t)$ be given by Lemma~\ref{lem4}.
Using Kirszbraun's theorem \cite{Kirszbraun1934} we can extend $x \mapsto T(x,t)$ to a Lipschitz function on $V_{\K}$  with the same Lipschitz constant $1+Ct$.
Since $C$ is independent of $t$, we can choose $\tau_1$ sufficiently small so that $x \mapsto T(x,t)$ is invertible for all $t\in [0,\tau_1]$, and the inverse is also Lipschitz with Lipschitz constant $(1-Ct)^{-1}$.
This shows that $T(\cdot,t):\overline{V_{\K}} \to T(\overline{V_{\K}},t)$ is bi-Lipschitz for all $t\in [0,\tau_1]$.

Now we prove  $T(V_{\K},t) = V_{\K}(t)$.
Suppose first that $\partial V_\K$ has only one connected component.
Since $T(\cdot, t): \overline{V_{\K}} \to T(\overline{V_{\K}},t)$ is  bi-Lipschitz it is a homeomorphism, thus it maps interior points onto interior points and boundary points onto boundary points, which implies that 
$T(V_{\K}, t)$ is the interior of $T(\partial V_{\K}, t)$.
Applying the Jordan curve theorem yields that $V_{\K}(t)$ is the interior of $\partial V_{\K}(t)$, and since $T(\partial V_{\K}, t) = \partial V_\K(t)$ due to Lemma~\ref{lem4}, their interiors coincide and we get  $T(V_{\K}, t) = V_{\K}(t)$.
The case where $\partial V_\K$ has several connected components follows in a similar way.

In view of \eqref{eq:beta} we have $\beta(x,0)=x$ for $x\in\gamma_k$.
Then due to  \eqref{eq:S} we have, for $x\in\gamma_k$,
\begin{align*}
\theta(x) = \partial_t T(x,0) 
&= 
\partial_t \beta(x,0) 
+ \partial_t\hat\alpha(x,0)\nabla_x \hlsf_k(x,0)\\ 
&\quad + \nabla_\Gamma\hat\alpha(x,0)\cdot \partial_t \beta(x,0)\nabla_x \hlsf_k(x,0)
+ \hat\alpha(x,0)D^2_x \hlsf_k(x,0)\partial_t \beta(x,0),
\end{align*}
where $\nabla_\Gamma$ denotes the tangential gradient on $\gamma_k$.
Using $\hat\alpha(x,0)=0$  for all $x\in \partial\om_k\cap B$, and consequently $\nabla_\Gamma\hat\alpha(x,0)=0$  for all $x\in \partial\om_k\cap B$,  where $B$ is the ball given by Assumption~\ref{assump:bounded}, we get
\begin{align}\label{111}
\theta(x) = \partial_t \beta(x,0) 
+ \partial_t\hat\alpha(x,0)\nabla_x \hlsf_k(x,0).
\end{align}
Then, taking the derivative with respect to $t$ at $t=0$ in  \eqref{117} we obtain
$$\partial_t\hat\alpha(x,0) \nabla \hlsf_k(x,0)\cdot \nabla \hlsf_k(x,0) + \partial_t\hlsf_k(x,0) =0,$$ 
hence
$$\partial_t\hat\alpha(x,0)  = - \frac{\partial_t\hlsf_k(x,0)}{\|\nabla \hlsf_k(x,0)\|^2}.$$ 
Since $\beta(x,t)\in\partial\om_k\cap B$ for all $t\in [0,\tau_1]$, $\partial_t \beta(x,0)$ is tangent to $\gamma_k$. 
Using also 
$$\nu(x) = \frac{\nabla_x \hlsf_k(x,0)}{\|\nabla_x \hlsf_k(x,0)\|} \text{ for } x\in\gamma_k,$$ 
in view of \eqref{111} we obtain, for $x\in\gamma_k$,
$$\theta(x)\cdot\nu(x) = \underbrace{\partial_t \beta(x,0) \cdot\nu(x)}_{=0}
+ \partial_t\hat\alpha(x,0)\nabla_x \hlsf_k(x,0)\cdot\nu(x)
= -\frac{\partial_t \hlsf_k(x,0)}{\|\nabla_x\hlsf_k(x,0)\|},$$
which proves \eqref{der_normal}.

Since $\nabla_x \hlsf_k(x,0)$ is normal to $\gamma_k$, using \eqref{111} we get, for $z\in\mM^2 \cap \overline{V_{\K}}$,
$$\theta(z)\cdot\tau(z)  = \partial_t \beta(z,0) \cdot\tau(z) . $$
Now, using the notation of the proof of Lemma~\ref{lem4}, suppose that $z=z_v$.
In view of \eqref{eq:beta} we then have
\begin{align*}
\partial_t \beta(z_v,0) & = \partial_t\sigma(s(z_v),0)\xi'(\sigma(s(z_v),0)) 
= [ \lambda(s(z_v))s'_w(0) + (1-\lambda(s(z_v)))s'_v(0) ]\xi'(\sigma(s(z_v),0)) \\
& = s'_v(0)\xi'(s_v),  
\end{align*}
where we have used $s(z_v) = s_v$, $\lambda(s_v) = 0$ and $\sigma(s_v,0) =s_v$.
Since $s_v(t) = \xi^{-1}(P(z_v(t))) = s(P(z_v(t)))$ and $P(z_v)=z_v$ we get
$s_v'(0) = \nabla_\Gamma s(z_v)\cdot [D_xP(z_v) z'_v(0)] $, thus
\begin{align*}
\partial_t \beta(z_v,0)\cdot\tau(z) 
& = \nabla_\Gamma s(z_v)\cdot [D_xP(z_v) z'_v(0)]\xi'(s_v)\cdot\tau(z)
= \nabla_\Gamma s(z_v)\cdot \xi'(s_v)  [D_xP(z_v) z'_v(0)]\cdot\tau(z),
\end{align*}
where we have used the fact that both $\nabla_\Gamma s(z_v)$ and $\xi'(s_v)$ are tangent to $\gamma_k$ to obtain the last equality.
Since $s(x)=\xi^{-1}(x)$, differentiating $s\circ \xi$ at $s_v$ yields $\nabla_\Gamma s(z_v)\cdot \xi'(s_v) =1$.
One can also show that $D_x P = I - \nu\otimes\nu$ on $\gamma_k$, where $I$ is the identity matrix, see \cite[Section~2.2]{MR3900277}. 
This yields, using $D_x P^\top = D_x P$ on $\gamma_k$,
\begin{align*}
\partial_t \beta(z_v,0)\cdot\tau(z) 
& 
= [D_xP(z_v) z'_v(0)]\cdot\tau(z)
= [D_xP(z_v) \tau(z)]\cdot z'_v(0)
=  z'_v(0)\cdot \tau(z).
\end{align*}
Finally, using \eqref{der:zrho} we get \eqref{der_tangential} for $z=z_v$. 
The same procedure yields \eqref{der_tangential} for any 
$z\in\mM^2 \cap \overline{V_{\K}}$.
\end{proof}

\begin{example}
In order to illustrate the notations and results of this section, we consider the simple example where $\K =\{1,2,3\}$, $\hlsf_{k}(x,t) = \|x - \si_k\|^2 - (r+t\delta r)^2$, $\si_1 = (0,1)$, $\si_2 = (-\sqrt{3}/2, -1/2)$, $\si_3 = (\sqrt{3}/2, -1/2)$ and $\om_k(t)$ is a disk of center $\si_k$ and radius $r+t\delta r$; 
see Figure~\ref{fig:1} for an illustration of the geometric configuration.
Then for $r=1$, $\delta r =1$ and $t>0$, the cell $V_{\K}(t) = \bigcap_{k\in\K} \om_k(t)$ forms a well-known geometric figure called Reuleaux triangle. In this example, each set $\gamma_k$ is an arc of circle with only one connected component, and it is easy to see that \eqref{Vibd:carac_b} is satisfied.
In \eqref{Z:carac}, we have $|\mM^2 \cap \overline{V_{\K}}|=3$ and the points $z_v(t), v\in \mM^2 \cap \overline{V_{\K}}$ are the vertices of the Reuleaux triangle.
We have $\mM_{\{k\}}(t) = \partial\om_k(t)$ for all $k\in \K$.
The set $\mM_{\{i,j\}}(t)$, $\{i,j\}\subset \K$ is composed of the two  intersection points of the circles $\partial B(\si_i,r+t\delta r)$ and $\partial B(\si_j,r+t\delta r)$, thus $\mM^2(t) := \cup_{\I\in\mathds{I}^2}\mM_\I(t)$ is composed of six points. 
Note that three of these six points are the vertices of the Reuleaux triangle, while the other three points are irrelevant for the description of the cell $V_{\K}(t)$. 
We also have $\mM^3 = \cup_{\I\in\mathds{I}^3}\mM_\I = \emptyset$.
Thus we conclude  that Assumptions~\ref{assump:bounded}, ~\ref{assump2a} and \ref{assump2} are satisfied when $r=1$, $\delta r=1$ and $t>0$.
In the singular case $r=1$, $\delta r=1$ and $t=0$, we have $\mM^3 := \cup_{\I\in\mathds{I}^3}\mM_\I = (0,0)\neq\emptyset$, i.e., the three circles intersect at the same point $(0,0)$, thus Assumption~\ref{assump2} is not satisfied in this case.
\end{example}

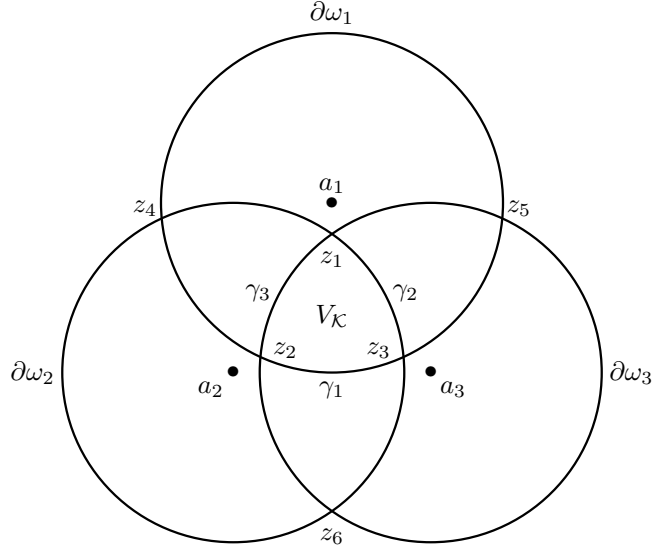
\begin{figure}
\begin{center}
\begin{tikzpicture}[scale=1.5]

\draw[black, line width=0.3mm] (0,1) circle (1.5 cm);
\draw[black, line width=0.3mm] (-0.866,-0.5) circle (1.5 cm);
\draw[black, line width=0.3mm] (0.866,-0.5) circle (1.5 cm);
\draw (0,1) node[above]{$\si_1$} node{$\bullet$};
\draw (-0.866,-0.5) node[below left]{$\si_2$} node{$\bullet$};
\draw (0.866,-0.5) node[below right]{$\si_3$} node{$\bullet$};

\draw (0,2.5) node[above]{$\partial\omega_1$};
\draw (-2.35,-0.5) node[left]{$\partial\omega_2$};
\draw (2.35,-0.5) node[right]{$\partial\omega_3$};

\draw (0,0.65) node[below]{$z_1$};
\draw (-0.6,-0.3) node[right]{$z_2$};
\draw (0.6,-0.3) node[left]{$z_3$};

\draw (-1.45,0.95) node[left]{$z_4$};
\draw (1.45,0.95) node[right]{$z_5$};
\draw (0,-1.8) node[below]{$z_6$};

\draw (0,-0.5) node[below]{$\gamma_1$};
\draw (-0.45,0.2) node[left]{$\gamma_3$};
\draw (0.45,0.2) node[right]{$\gamma_2$};

\draw (0,0) node{$V_\K$};
\end{tikzpicture}
\end{center}
\caption{In this example  $\K =\{1,2,3\}$ and the cell $V_\K$ is a Reuleaux triangle formed by the intersection of three disks $\om_1,\om_2,\om_3$ of centers $\si_1,\si_2,\si_3$.
We have $\mM_{\{1,2\}} = \{z_3,z_4\}$, $\mM_{\{1,3\}} = \{z_2,z_5\}$, $\mM_{\{2,3\}} = \{z_1,z_6\}$ and $\mM_{\{1,2,3\}} =\emptyset$, which means that the intersection of the three circles is empty. 
Thus $\mM^2 = \cup_{\I\in\mathds{I}^2}\mM_\I = \{z_1,z_2,z_3,z_4,z_5,z_6\}$ and $\mM^3 = \cup_{\I\in\mathds{I}^3}\mM_\I = \emptyset$.
Property \eqref{Z:carac} at $t=0$ corresponds to
$\mM^2 \cap \overline{V_{\K}} = \{z_1,z_2,z_3\}$, which means that the points $z_1,z_2,z_3$ are the vertices of the Reuleaux triangle.
The three arcs of circle $\gamma_1,\gamma_2,\gamma_3$ of the decomposition \eqref{Vibd:carac_b} are the three edges of the Reuleaux triangle.
}
\label{fig:1}
\end{figure}

\section{Perturbation of minimization diagrams}\label{sec:MinDiags}

In this section we describe how the abstract theory developed in Section~\ref{sec:abstract} allows us to treat the case of minimization diagrams. 
We start by giving a few particular examples of minimization diagrams taken from \cite{wormser:tel-00410850}.
\begin{example}[Voronoi diagrams]
Let $\si_i\in\R^q$ with $q=2$.
The cells of a Voronoi diagram are defined as $$V_i(\bsi):= \inte\left\{x \in A \text{ such that } \| x - \si_i\|^2  \leq \| x - \si_k\|^2\text{ for all } k\in\mL\setminus\{i\}  \right\}.$$
This corresponds to a minimization diagram for the  particular case $\lsf_k(x,\si_k) = \|x - \si_k\|^2$ for~all~$k\in\mL$. 
\end{example}
\begin{example}[Power diagrams]
Let $\si_i = (c_i,r_i)$ with $c_i\in\R^2$ and $r_i\in\R$, then $\si_i\in\R^q$ with $q=3$.
The cells of a power diagram are defined as $$V_i(\bsi):= \inte\left\{x \in A \text{ such that } \| x - c_i\|^2 - r_i^2 \leq \| x - c_k\|^2 - r_k^2\text{ for all } k\in\mL\setminus\{i\}  \right\}.$$
This corresponds to a minimization diagram for the  particular case $\lsf_k(x,\si_k) = \|x - c_k\|^2 - r_k^2 $ for~all~$k\in\mL$. 
\end{example}
\begin{example}[M\"obius diagrams]
Let $\si_i = (p_i,\lambda_i,\mu_i)$ with $p_i\in\R^2$,  and $\lambda_i, \mu_i\in\R$, then $\si_i\in\R^q$ with $q=4$.
The cells of a M\"obius diagrams are defined as $$V_i(\bsi):= \inte\left\{x \in A \text{ such that } \lambda_i \|x - p_i\|^2 - \mu_i  
\leq 
\lambda_k \|x - p_k\|^2 - \mu_k \text{ for all } k\in\mL\setminus\{i\}   \right\}.$$
This corresponds to a minimization diagram for the  particular case $\lsf_k(x,\si_k) = \lambda_k \|x - p_k\|^2 - \mu_k$, $k\in\mL$.
\end{example}
\subsection{Sensitivity analysis of edges and vertices}\label{sec:sensitivity_vertices}

Let $\delta\ba = \{\dsi_i\}_{i\in\mL}$ be a set of {\it sites perturbations}, with $\dsi_i\in\R^q$.
Through appropriate choices of the sets $\K,\I$, of the function $\hblsf_\I$, and applying Theorem~\ref{thm01}, we are able to describe the perturbations of the cells $V_i(\ba +t\delta\ba)$ via a bi-Lipschitz mapping $T(\cdot,t): V_i(\ba)\to V_i(\ba +t\delta\ba)$.
The purpose of this section is mainly to give a more concrete description of the formulas  \eqref{der_normal}-\eqref{der_tangential} in the particular case of minimization diagrams.
Formula \eqref{der_normal} is used to describe the perturbation of both interior edges and edges on the boundary of $A$, and  \eqref{der_tangential}  is used here to describe  the perturbation of both interior vertices and vertices on the boundary of $A$, see Theorem~\ref{thm01:walls}.

To avoid too many technicalities,  in this section we will assume that $A$ is smooth and bounded and that $A$ is defined as the sublevel set of a given function $\vp\in\C^\infty(\R^2,\R)$, i.e.,
$$ A: = \{ x\in\R^2\ |\ \vp(x)<0\}.$$
The smoothness of $A$ is not a restrictive assumption, as the case of a piecewise smooth $A$ can be readily obtained by considering intersections of smooth sets, and the general theory of Section~\ref{sec:abstract}  could still be applied.

We now show that under appropriate conditions on $\blsf(\bsi)$, the minimization diagram $\mv$ forms a partition of $A$ and the boundary of the cells of the minimization diagram are one-dimensional.
Assumption~\ref{assump1:MDbd} below corresponds to Assumption~\ref{assump2a} with a specific choice of  $\hblsf_{\I}$ and of the set of indices $\K$.
\begin{assump}\label{assump1:MDbd}
Assumption~\ref{assump2a} holds for $\K = \{\kappa+1\}\cup \mL\setminus\{i\}$  for all $i\in\mL$, with
 $\hlsf_{\ck+1}(x , t) = \vp(x)$  and with $\hlsf_{k}(x,t) = \lsf_i(x,\si_i + t\dsi_i) - \lsf_{k}(x,\si_{k} + t\dsi_{k})$ for all $k\in\mL\setminus\{i\}$.  
\end{assump}
\begin{remark}\label{rem:assump1:MDbd}
If Assumption~\ref{assump1:MDbd}  holds, then Assumption~\ref{assump:bounded} is satisfied since $A$ is bounded.
Assumption \ref{assump1:MDbd} also implies that $\|\nabla_x\lsf_i(x,\si_i) - \nabla_x\lsf_j(x,\si_j)\|>0$ for all $x\in \{y\in \R^2\ |\ \lsf_i(y,\si_i) = \lsf_j(y,\si_j) \}$ and for all $\{i,j\}\subset\mL$, and that $\|\nabla_x \vp(x)\|>0$ for all $x\in \partial A$. Note that these are natural and standard assumptions for domains defined as sublevel sets.
\end{remark}

The proof of the following result can be found in \cite[Theorem~1]{laurainLEM}.
It essentially  guarantees that the cells of the diagram do not overlap when Assumption~\ref{assump1:MDbd} holds.
\begin{thm}\label{thm:partitions}
Suppose that Assumption \ref{assump1:MDbd} holds, then  $\dim \partial V_k(\bsi) \leq 1$, $V_k(\bsi) \cap V_\ell(\bsi) =\emptyset$ for all $\{k,\ell\}\subset\mL$ and $\bigcup_{k\in\mL} \overline{V_k(\bsi)} = \overline{A}$. 
\end{thm}

We now study the evolution of the vertices of the minimization diagram.
For $\{i,j,k\}\subset\mL$ let us define $\trp_{ijk}(t):= \overline{V_i(\bsi+t\dbsi)}\cap \overline{V_j(\bsi+t\dbsi)}\cap \overline{V_k(\bsi+t\dbsi)}$ and 
$\dop_{ij}(t) := \overline{V_i(\bsi+t\dbsi)}\cap \overline{V_j(\bsi+t\dbsi)}\cap \partial A$.
The set $\trp_{ijk}(t)$ is a set of {\it interior vertices}, i.e.,  points in $A$ at the intersection of three cells.
The set $\dop_{ij}(t)$ is a set of {\it boundary vertices}, i.e., points on $\partial A$ at the intersection of two cells.
We will write $\trp_{ijk} := \trp_{ijk}(0)$ and $\dop_{ij} := \dop_{ij}(0)$ for simplicity.
The purpose of Assumption~\ref{assump2:MDbd} is to guarantee that $\trp_{ijk}(t)$ and $\dop_{ij}(t)$ are stable with respect to $t$, which means essentially that their cardinality is constant and that they are continuous with respect to $t$.
Assumption~\ref{assump2:MDbd} corresponds to Assumption~\ref{assump2} with a specific choice of the function $\hblsf_{\I}$ and of the set of indices $\K$.

\begin{assump}\label{assump2:MDbd}
Assumption~\ref{assump2} holds for $\K = \{\kappa+1\}\cup \mL\setminus\{i\}$  for all $i\in\mL$, with
 $\hlsf_{\ck+1}(x , t) = \vp(x)$  and with $\hlsf_{k}(x,t) = \lsf_i(x,\si_i + t\dsi_i) - \lsf_{k}(x,\si_{k} + t\dsi_{k})$ for all $k\in\mL\setminus\{i\}$.  
\end{assump}
\begin{remark}\label{assump_grad_f_bd}
If Assumption~\ref{assump2:MDbd} holds, then Assumption~\ref{assump:bounded} is satisfied since $A$ is bounded.
Also, Assumption~\ref{assump2:MDbd} implies  
$$(\nabla_x \lsf_i(v,\si_i) -\nabla_x \lsf_j(v,\si_j))^\perp\cdot (\nabla_x \lsf_i(v,\si_i) -\nabla_x \lsf_k(v,\si_k)) \neq 0$$
for all $v\in \trp_{ijk}$ and any  $\{i,j,k\}\subset \mL$, and  
$$(\nabla_x \lsf_i(v,\si_i) -\nabla_x \lsf_j(v,\si_j))^\perp\cdot \nabla\vp(v) \neq 0$$
for all $v\in  \dop_{ij}$  and any  $\{i,j\}\subset \mL$. 
\end{remark}

\begin{thm}\label{thm:perturbMD1}
Suppose Assumption \ref{assump2:MDbd} holds and let
 $\{i,j,k\}\subset \mL$.
Then $\trp_{ijk}$ is finite and there exists $\tau_1>0$ such that for all $v\in \trp_{ijk}$ there exists  a unique smooth function $z_v:[0,\tau_1]\to \R^{2}$ satisfying $z_v(0)=v$ and
\begin{align}\label{MI:dec_int}
\trp_{ijk}(t) 
= \bigcup_{v\in\trp_{ijk}} \{z_v(t)\} \text{ for all }t\in [0,\tau_1].
\end{align}
In addition  we have
\begin{align}\label{triple_perturb} 
z_v'(0)  
& =
M_v(j,k,i) \dsi_i + M_v(k,i,j) \dsi_j + M_v(i,j,k) \dsi_k  
\end{align}
where
\begin{align}\label{M} 
M_v(i,j,k) := \frac{(\nabla_x \lsf_i(v,\si_i)-  \nabla_x  \lsf_j(v,\si_j))^\perp\otimes \nabla_{\si_k} \lsf_k(v,\si_k)^\top}{Q_v(i,j,k) } 
\end{align}
and
$$ Q_v(i,j,k) := 
\det\begin{pmatrix}
 (\nabla_x \lsf_i(v,\si_i)- \nabla_x \lsf_j(v,\si_j))^\top\\
 (\nabla_x \lsf_i(v,\si_i)- \nabla_x \lsf_k(v,\si_k))^\top
\end{pmatrix}.$$
\end{thm}
\begin{proof}
Applying Lemma~\ref{lemma02} with  $\I = \{j,k\}\subset \K = \{\kappa+1\}\cup\mL\setminus\{i\}$, with $\hlsf_{j}(x,t) = \lsf_i(x,\si_i + t\dsi_i) - \lsf_{j}(x,\si_{j} + t\dsi_{j})$ and $\hlsf_{k}(x,t) = \lsf_i(x,\si_i + t\dsi_i) - \lsf_{k}(x,\si_{k} + t\dsi_{k})$, we get $\trp_{ijk}(t)\subset \mM_{\I}(t)$ and \eqref{MI:dec_int} follows.

Further, we have
$$\hblsf_\I(x,t)=
\begin{pmatrix}
\hlsf_j(x , t)\\
\hlsf_k(x , t) 
\end{pmatrix}
,\qquad
D_x  \hblsf_\I(x,t) =
\begin{pmatrix}
\nabla_x \hlsf_{j}(x,t)^\top\\
\nabla_x \hlsf_{k}(x,t)^\top
\end{pmatrix}.
$$
In view of \eqref{der:zrho} we have, for $v\in \trp_{ijk}$,
\begin{equation*}
z_v'(0)  = - D_x \hblsf_\I(v,0)^{-1}\partial_t \hblsf_\I(v,0).
\end{equation*}
We compute 
\begin{align*}
D_x \hblsf_\I(v,0)^{-1} & =\frac{( -(\nabla_x \lsf_i(v,\si_i) -\nabla_x \lsf_k(v,\si_k))^\perp
\quad  
(\nabla_x \lsf_i(v,\si_i) -\nabla_x \lsf_j(v,\si_j))^\perp)}{\det D_x \hblsf_\I(v,0)}
\end{align*}
and
\begin{align*}
\partial_t \hblsf_\I(v,0) & = 
\begin{pmatrix}
 \nabla_{\si_i} \lsf_i(v,\si_i) \cdot\dsi_i - \nabla_{\si_j} \lsf_j(v,\si_j) \cdot\dsi_j\\
 \nabla_{\si_i} \lsf_i(v,\si_i) \cdot\dsi_i - \nabla_{\si_k} \lsf_k(v,\si_k) \cdot\dsi_k
\end{pmatrix}\\
& =
\begin{pmatrix}
 \nabla_{\si_i} \lsf_i(v,\si_i)^\top\\
 \nabla_{\si_i} \lsf_i(v,\si_i)^\top
\end{pmatrix}
\dsi_i
-
\begin{pmatrix}
 \nabla_{\si_j} \lsf_j(v,\si_j)^\top\\
 0
\end{pmatrix}
\dsi_j
-
\begin{pmatrix}
0\\
\nabla_{\si_k} \lsf_k(v,\si_k)^\top
\end{pmatrix}
\dsi_k\\
& = \left[
\begin{pmatrix}
 1\\
 1
\end{pmatrix}
\otimes 
\nabla_{\si_i} \lsf_i(v,\si_i)^\top
\right]
\dsi_i 
-
\left[
\begin{pmatrix}
 1\\
 0
\end{pmatrix}
\otimes 
\nabla_{\si_j} \lsf_j(v,\si_j)^\top
\right]
\dsi_j
-
\left[
\begin{pmatrix}
 0\\
 1
\end{pmatrix}
\otimes 
\nabla_{\si_k} \lsf_k(v,\si_k)^\top
\right]
\dsi_k.
\end{align*}
Then we compute
\begin{align*}
& ( -(\nabla_x \lsf_i(v,\si_i) -\nabla_x \lsf_k(v,\si_k))^\perp
\quad  
(\nabla_x \lsf_i(v,\si_i) -\nabla_x \lsf_j(v,\si_j))^\perp)
\left[
\begin{pmatrix}
 1\\
 1
\end{pmatrix}
\otimes 
\nabla_{\si_i} \lsf_i(v,\si_i)^\top
\right]\\
& 
\quad = 
( -(\nabla_x \lsf_i(v,\si_i) -\nabla_x \lsf_k(v,\si_k))^\perp
+
(\nabla_x \lsf_i(v,\si_i) -\nabla_x \lsf_j(v,\si_j))^\perp)\otimes \nabla_{\si_i} \lsf_i(v,\si_i)^\top\\
& 
\quad= 
(\nabla_x \lsf_k(v,\si_k) -\nabla_x \lsf_j(v,\si_j))^\perp
\otimes \nabla_{\si_i} \lsf_i(v,\si_i)^\top.
\end{align*}
In a similar way we also have
\begin{align*}
& ( -(\nabla_x \lsf_i(v,\si_i) -\nabla_x \lsf_k(v,\si_k))^\perp
\quad  
(\nabla_x \lsf_i(v,\si_i) -\nabla_x \lsf_j(v,\si_j))^\perp)
\left[-
\begin{pmatrix}
 1\\
 0
\end{pmatrix}
\otimes 
\nabla_{\si_j} \lsf_j(v,\si_j)^\top
\right]\\
& \quad = (\nabla_x \lsf_i(v,\si_i) -\nabla_x \lsf_k(v,\si_k))^\perp
\otimes
\nabla_{\si_j} \lsf_j(v,\si_j)^\top
\end{align*}
and
\begin{align*}
& ( -(\nabla_x \lsf_i(v,\si_i) -\nabla_x \lsf_k(v,\si_k))^\perp
\quad  
(\nabla_x \lsf_i(v,\si_i) -\nabla_x \lsf_j(v,\si_j))^\perp)
\left[-
\begin{pmatrix}
 0\\
 1
\end{pmatrix}
\otimes 
\nabla_{\si_k} \lsf_k(v,\si_k)^\top
\right]\\
& \quad = (\nabla_x \lsf_j(v,\si_j) -\nabla_x \lsf_i(v,\si_i))^\perp
\otimes
\nabla_{\si_k} \lsf_k(v,\si_k)^\top .
\end{align*}
Gathering these results we obtain \eqref{triple_perturb}.
\end{proof}
\begin{remark}
Note that we have $Q_v(i,j,k) = Q_v(k,i,j) = Q_v(j,k,i)$; this can  be checked using the multilinearity of the determinant. Note also that $Q_v(i,j,k)$ is the oriented area of the parallelogram spanned by the vectors $\nabla_x \lsf_i(v,\si_i)- \nabla_x \lsf_j(v,\si_j)$ and $\nabla_x \lsf_i(v,\si_i)- \nabla_x \lsf_k(v,\si_k)$.
\end{remark}

\begin{remark}
A rotation of the index notation for $(i,j,k)$ in \eqref{triple_perturb}, for instance $(i,j,k)\to (k,i,j)$, gives exactly the same expression for $z_v'(0)$, as expected, since the result should be independent of the choice of the indices $i,j,k$.
Also, it can be checked that exchanging the notation for two indices, for instance  $(i,j,k)\to (i,k,j)$ yields the same result for $z_v'(0)$.
As an example, for the first term of $z_v'(0)$ we have 
\begin{align*}
M_v(j,k,i) &= \frac{(\nabla_x \lsf_k(v,\si_k)-  \nabla_x  \lsf_j(v,\si_j))^\perp\otimes \nabla_{\si_i} \lsf_i(v,\si_i)^\top}{Q_v(j,k,i)},\\
M_v(k,j,i) &= \frac{(\nabla_x \lsf_j(v,\si_j)-  \nabla_x  \lsf_k(v,\si_k))^\perp\otimes \nabla_{\si_i} \lsf_i(v,\si_i)^\top}{Q_v(k,j,i)} = M_v(j,k,i),
\end{align*}
where we have used the fact that $Q_v(j,k,i) = Q_v(k,i,j) = - Q_v(k,j,i)$ since  $Q_v(k,i,j)$ is an oriented area.
\end{remark}
Now we consider the case of vertices at the boundary of two cells and located on the boundary of $A$.

\begin{thm}\label{thm:perturbMD2} 

Suppose Assumption \ref{assump2:MDbd} holds and let
 $\{i,j\}\subset \mL$.
Then $\dop_{ij}$ is finite and there exists $\tau_1>0$ such that for all $v\in \dop_{ij}$ there exists  a unique smooth function $z_v:[0,\tau_1]\to \R^{2}$ satisfying $z_v(0)=v$ and
\begin{align}\label{MI:dec_bd}
\dop_{ij}(t) 
= \bigcup_{v\in\dop_{ij}} \{z_v(t)\} \text{ for all }t\in [0,\tau_1].
\end{align}
In addition we have
\begin{align}\label{wprime_bd}
z_v'(0) = \mm_v(j,i) \dsi_i + \mm_v(i,j) \dsi_j 
\end{align}
with 
\begin{equation} \label{calM}
\mm_v(j,i): = \frac{\nabla\vp(v)^\perp \otimes \nabla_{\si_i} \lsf_i(v,\si_i)^\top}{\det\begin{pmatrix}
 (\nabla_x \lsf_i(v,\si_i) -\nabla_x \lsf_j(v,\si_j))^\top\\
 \nabla_x\vp(v)^\top
\end{pmatrix}
}
.
\end{equation}
\end{thm}
\begin{proof}
Applying Lemma~\ref{lemma02} with  $\I =\{j,\ck+1\}\subset\K = \{\kappa+1\}\cup\mL\setminus\{i\}$ and $\hblsf_\I(x,t)=
(\hlsf_j(x , t), \hlsf_{\ck+1}(x , t))^\top$
with  $\hlsf_{j}(x,t) = \lsf_i(x,\si_i + t\dsi_i) - \lsf_{j}(x,\si_{j} + t\dsi_{j})$ and  $\hlsf_{\ck+1}(x , t) = \vp(x)$, we get $\dop_{ij}(t)\subset \mM_{\I}(t)$ and \eqref{MI:dec_bd} follows.

In view of \eqref{der:zrho} we have
\begin{equation*}
z_v'(0)  = - D_x \hblsf_\I(v,0)^{-1}\partial_t \hblsf_\I(v,0)
\text{ with }
D_x  \hblsf_\I(x,t) =
\begin{pmatrix}
\nabla_x \hlsf_{j}(x,t)^\top\\
\nabla_x \vp(x)^\top
\end{pmatrix}.
\end{equation*}
We compute
\begin{align*}
D_x  \hblsf_\I(v,0)^{-1}& =\frac{( -\nabla_x\vp(v)^\perp
\quad  
(\nabla_x \lsf_i(v,\si_i) -\nabla_x \lsf_j(v,\si_j))^\perp)}{\det D_x  \hblsf_\I(v,0)}
\end{align*}
and
\begin{align*}
\partial_t \hblsf_\I(v,0)  & = 
\begin{pmatrix}
 \nabla_{\si_i} \lsf_i(v,\si_i) \cdot\dsi_i - \nabla_{\si_j} \lsf_j(v,\si_j) \cdot\dsi_j\\
 0
\end{pmatrix}\\
& = \left[
\begin{pmatrix}
 1\\
 0
\end{pmatrix}
\otimes 
\nabla_{\si_i} \lsf_i(v,\si_i)^\top
\right]
\dsi_i 
-
\left[
\begin{pmatrix}
 1\\
 0
\end{pmatrix}
\otimes 
\nabla_{\si_j} \lsf_j(v,\si_j)^\top
\right]
\dsi_j.
\end{align*}
Further,
\begin{align*}
 ( -\nabla\vp(v)^\perp
\quad  
(\nabla_x \lsf_i(v,\si_i) -\nabla_x \lsf_j(v,\si_j))^\perp)\left[
\begin{pmatrix}
 1\\
 0
\end{pmatrix}
\otimes 
\nabla_{\si_i} \lsf_i(v,\si_i)^\top
\right]
& =
 -\nabla\vp(v)^\perp \otimes \nabla_{\si_i} \lsf_i(v,\si_i)^\top ,\\
( -\nabla\vp(v)^\perp
\quad  
(\nabla_x \lsf_i(v,\si_i) -\nabla_x \lsf_j(v,\si_j))^\perp)\left[
-
\begin{pmatrix}
 1\\
 0
\end{pmatrix}
\otimes 
\nabla_{\si_j} \lsf_j(v,\si_j)^\top
\right]
& =
 \nabla\vp(v)^\perp \otimes \nabla_{\si_j} \lsf_j(v,\si_j)^\top.
 \end{align*}
Gathering these results we obtain  \eqref{wprime_bd}.
\end{proof}

The following result corresponds to the application of Theorem~\ref{thm01} in the context of minimization diagrams.
For $k\in\mL\setminus\{i\}$,  $E_{ik}(\bsi+t\dbsi) := \overline{V_i(\bsi+t\dbsi)}\cap \overline{V_k(\bsi+t\dbsi)}$ denotes an interior edge of the diagram $\mv(\bsi+t\dbsi)$, while $E_i(\bsi+t\dbsi) := \overline{V_i(\bsi+t\dbsi)}\cap \partial A$ denotes a boundary edge of the diagram.
Note that $E_{ik}(\bsi+t\dbsi)$ and $E_i(\bsi+t\dbsi)$ may have several connected components.

\begin{thm}\label{thm01:walls}
Suppose Assumptions~\ref{assump1:MDbd} and \ref{assump2:MDbd} hold.
Then there exist $\tau_1>0$ and a  mapping $T:\overline{V_i(\bsi)} \times [0,\tau_1] \to \R^2$ satisfying
$T(V_i(\bsi),t) = V_i(\bsi+t\dbsi)$, $T(E_{ik}(\bsi),t)=E_{ik}(\bsi+t\dbsi)$ for all $k\in\mL\setminus\{i\}$, $T(E_i(\bsi),t)=E_i(\bsi+t\dbsi)$,  $T(\partial V_i(\bsi),t) = \partial V_i(\bsi+t\dbsi)$ and $T(\cdot, t): \overline{V_{\K}} \to \overline{V_{\K}(t)}$ is  bi-Lipschitz   for all $t\in [0,\tau_1]$. 
In addition we have
\begin{align}
\label{der_normal:walls} \theta(x)\cdot\nu(x) &
= \frac{ \nabla_{\si_k} \lsf_k(x,\si_k)\cdot\delta\si_k - \nabla_{\si_i} \lsf_i(x,\si_i)\cdot\delta\si_i}{\|\nabla_x\lsf_k(x,\si_k) - \nabla_x\lsf_i(x,\si_i)\|}\text{ for all } x\in E_{ik}(\bsi), \\
\label{der_normal:walls2} \theta(x)\cdot\nu(x) &
= 0\text{ for all } x\in E_i(\bsi),\\
\label{der_tangential:walls}\theta(v)\cdot\tau(v) &=  (M_v(j,k,i) \dsi_i + M_v(k,i,j) \dsi_j + M_v(i,j,k) \dsi_k)\cdot \tau(v)\text{ for all } v\in\trp_{ijk},\\
\label{der_tangential:walls2}\theta(v)\cdot\tau(v) &=  ( \mm_v(j,i) \dsi_i + \mm_v(i,j) \dsi_j)\cdot \tau(v)\text{ for all } v\in\dop_{ij},
\end{align}
where $\theta := \partial_t T(\cdot,0)$, $\nu$ is the outward unit normal vector to $V_\K$, and $\tau$ is the tangent vector to $\partial V_i(\bsi)$ with respect to a counterclockwise orientation. 
\end{thm}
\begin{proof}
The properties of $T$ follow from Lemma~\ref{lem4} and Theorem~\ref{thm01}, considering that $E_{ik}(\bsi+t\dbsi)$ and $E_i(\bsi+t\dbsi)$ both correspond to $\gamma_k(t)$ in Lemma~\ref{lem4}.
Applying \eqref{der_normal} in Theorem~\ref{thm01} with  $$\hlsf_{k}(x,t) = \lsf_i(x,\si_i + t\dsi_i) - \lsf_{k}(x,\si_{k} + t\dsi_{k})$$ we get  \eqref{der_normal:walls}.
Applying \eqref{der_normal} in Theorem~\ref{thm01} with  $\hlsf_{k}(x,t) =\vp(x)$
we get \eqref{der_normal:walls2}.
Then \eqref{der_tangential:walls} follows from applying \eqref{der_tangential} and \eqref{triple_perturb}, and  \eqref{der_tangential:walls2} is an application of \eqref{der_tangential} and \eqref{wprime_bd}.
\end{proof}

\subsection{Application to  integrals on cells}\label{sec:G1}

We now give applications of Theorems~\ref{thm:perturbMD1},  \ref{thm:perturbMD2} and \ref{thm01:walls}.
Let us consider the following standard cost functional defined as a domain integral:
$$G_1(\bsi) := \int_{V_i(\bsi)} f(x) dx,$$
where  $f\in C^1(\overline{A},\R^2)$.
Using Lemma~\ref{lem3b}, we get that $V_i(\bsi+t\dbsi)$ is Lipschitz for all $t\in [0,\tau_1]$.
Applying Theorem~\ref{thm01:walls} and a change of variables $x\mapsto T(x,t)$, then using the fact that   $T(\cdot, t): V_i(\bsi)\to V_i(\bsi+t\dbsi)$ is bi-Lipschitz, we get
$$G_1(\bsi+t\dbsi) := \int_{V_i(\bsi+t\dbsi)} f(x) dx
= \int_{T(V_i(\bsi),t)} f(x) dx
= \int_{V_i(\bsi)} f(T(x,t)) |\det T(x,t)| dx.$$
This yields
$$ \nabla G_1(\bsi)\cdot\dbsi = \int_{V_i(\bsi)} \divv(f(x) \theta(x))dx,$$
where $\theta: =  \partial_t T(\cdot,0)$.
Since $V_i(\bsi)$ is Lipschitz, applying the divergence theorem we get
$$ \nabla G_1(\bsi)\cdot\dbsi = \int_{\partial V_i(\bsi)} f(x) \theta(x)\cdot\nu(x) dx.$$
Let $\hme$ denote the set of interior edges of the cell $V_i(\bsi)$, i.e., edges that are included in $A$.
Then, applying \eqref{der_normal:walls} and \eqref{der_normal:walls2} we get
\begin{align*} 
\nabla G_1(\bsi) \cdot\dbsi
& = \sum_{E\in\hme}\int_{E} f(x) \frac{ \nabla_{\si_{k(i,E)}} \lsf_{k(i,E)}(x,\si_{k(i,E)})\cdot\delta\si_{k(i,E)} - \nabla_{\si_i} \lsf_i(x,\si_i)\cdot\delta\si_i}{\|\nabla_x\lsf_{k(i,E)}(x,\si_{k(i,E)}) - \nabla_x\lsf_i(x,\si_i)\|} dx,
\end{align*}
where $k(i,E)$ is the index such that $E = \overline{V_i(\bsi)}\cap \overline{V_{k(i,E)}(\bsi)}$.
Note that Assumption~\ref{assump1:MDbd} and Remark~\ref{rem:assump1:MDbd}
imply $\|\nabla_x\lsf_{k(i,E)}(x,\si_{k(i,E)}) - \nabla_x\lsf_i(x,\si_i)\|>0$.

\subsection{Application to  integrals on edges}\label{sec:bd_int}

Let us consider another standard cost functional defined as an integral over an edge of the minimization diagram $\mv(\bsi)$:
$$G_2(\bsi) := \int_{E(\bsi)} f(x) dx,$$
where  $f\in C^1(\overline{A},\R^2)$.
Here $E(\bsi)$ can either be an interior edge given by $E(\bsi) = \overline{V_i(\bsi)}\cap \overline{V_k(\bsi)}$
or a boundary edge given by $E(\bsi) = \overline{V_i(\bsi)}\cap \partial A$.
To compute the gradient of $G_2$ we recall the following basic results. 

\begin{thm}[tangential divergence theorem]\label{lem:tandivpol}
Let $\Gamma\subset\R^2$ be a $C^k$ open curve, $k\geq 2$, with a parameterization $\xi$, and denote $(v,w)$ the starting and ending points of $\Gamma$, respectively, with respect to $\xi$.
Let $\tau$ be the unitary-norm  tangent vector to $\Gamma$, $\nu$ the unitary-norm normal vector to $\Gamma$, and $\mathcal{H}$ the mean curvature of $\Gamma$, with respect to the parameterization $\xi$.
Let $F\in W^{1,1}(\Gamma,\R^2)\cap C^0(\overline{\Gamma},\R^2)$,
then we have 
$$ \int_{\Gamma} \divv_\Gamma(F(x))dx 
= \int_{\Gamma} \mathcal{H}(x) F(x)\cdot \nu(x)dx +  \llbracket F(x)\cdot \tau(x)\rrbracket_v^w,$$
where  
$\divv_\Gamma(F)$ is the tangential divergence of $F$ on $\Gamma$, and 
$$\llbracket F(x)\cdot \tau(x)\rrbracket_v^w :=  F(w)\cdot \tau(w) - F(v) \cdot \tau(v).$$
\end{thm}
\begin{proof}
The result follows from  \cite[\S~7.2]{MR756417} and \cite[Ch.~9, \S~5.5]{MR2731611}.
\end{proof}

\begin{lemma}[change of variables for line integrals]\label{lem:intsub}
Let $\Gamma\subset\R^2$ be a $C^k$ open curve, $k\geq 2$, and  $\nu$ a unitary-norm normal vector to $\Gamma$.
Let $F\in C^0(\overline{\Gamma},\R^2)$ and $T(\cdot, t):\overline{\Gamma}\to T(\overline{\Gamma},t)$ be a bi-Lipschitz mapping. 
Then  
$$ \int_{T(\Gamma,t)} F(x)\, dx 
= \int_{\Gamma} F(T(x,t)) \zeta(x,t)dx,
$$
where 
\begin{equation}\label{density:changevar}
\zeta(x,t) := \|\det(D_xT(x,t)) D_xT(x,t)^{-\top}\nu(x)\|
\end{equation}
and $\det(D_xT(x,t)) D_xT(x,t)^{-\top}$ is the cofactor matrix of $D_xT(x,t)$.
Furthermore, we have 
\begin{equation}\label{density:der}
\partial_t\zeta(x,0) = \divv_\Gamma \theta(x) \text{ with } \theta: = \partial_t T(\cdot,0) \text{ on } \Gamma.
\end{equation}
\end{lemma}
\begin{proof}
See \cite[Prop.~5.4.3]{MR3791463}.
\end{proof}

Applying Theorem~\ref{thm01:walls}, Lemma~\ref{lem:intsub}  and a change of variables $x\mapsto T(x,t)$, then using the fact that   $T(\cdot, t): E(\bsi)\to E(\bsi+t\dbsi)$ is bi-Lipschitz, we get
$$G_2(\bsi+t\dbsi) := \int_{E(\bsi+t\dbsi)} f(x) dx
= \int_{T(E(\bsi),t)} f(x) dx
= \int_{E(\bsi)} f(T(x,t)) \zeta(x,t) dx.$$
This yields, using \eqref{density:der},
$$ \nabla G_2(\bsi)\cdot\dbsi 
= \int_{E(\bsi)} \nabla_x f(x) \cdot \theta(x) 
+ f(x)\divv_\Gamma(\theta(x))dx
= 
\int_{E(\bsi)} \partial_\nu f(x)  \theta(x)\cdot\nu(x) 
+ \divv_\Gamma(f(x)\theta(x))dx,$$
where $\theta: =  \partial_t T(\cdot,0)$, $\nu$ is the outward unit normal vector to $V_i(\bsi)$ and $\partial_\nu f(x):= \nabla_x f(x)\cdot \nu(x)$.
According to Lemma~\ref{lem3b}, $E(\bsi)$ is a finite union of smooth, connected arcs. 
Let $\mE_{E(\bsi)}$ be the set of these arcs, then we have
$$E(\bsi) = \bigcup_{\gamma\in \mE_{E(\bsi)}} \overline{\gamma}. $$
Thus we can write
$$ \nabla G_2(\bsi)\cdot\dbsi 
= 
\int_{E(\bsi)} \partial_\nu f(x)  \theta(x)\cdot\nu(x) dx
+ \sum_{\gamma\in \mE_{E(\bsi)}} \int_\gamma \divv_\Gamma(f(x)\theta(x))dx.$$
Since $V_i(\bsi)$ is Lipschitz, applying Theorem~\ref{lem:tandivpol} on each integral over $\gamma$ we get
\begin{equation}\label{grad:G2}
\nabla G_2(\bsi)\cdot\dbsi = \int_{E(\bsi)} (\partial_\nu f(x) + \mathcal{H}(x)) \theta(x)\cdot\nu(x) dx + \sum_{\gamma\in \mE_{E(\bsi)}}\llbracket f(x)\theta(x)\cdot \tau(x)\rrbracket_v^w, 
\end{equation}
where $(v,w)$ denote the starting and ending points of $\gamma$, with respect to a counterclockwise orientation on $\partial V_i(\bsi)$.
Finally, in \eqref{grad:G2}, $\theta(x)\cdot\nu(x)$ is given by \eqref{der_normal:walls} if $E(\bsi) = \overline{V_i(\bsi)}\cap \overline{V_j(\bsi)}$ is an interior edge and by  \eqref{der_normal:walls2}  if $E(\bsi) = \overline{V_i(\bsi)}\cap \partial A$ is a boundary edge.
Also, $\theta(x)\cdot \tau(x)$ is given by 
\eqref{der_tangential:walls} if $x\in A$ and by  \eqref{der_tangential:walls2} if $x\in\partial A$.
\section{The particular case of Euclidean Voronoi diagrams}\label{sec:vordiag}

Voronoi diagrams are the simplest example of minimization diagrams, corresponding to $q=2$ and $\lsf_i(x,\si_i) = \|x - \si_i\|^2$ for all $i\in\mL$, and have applications in many fields such as natural sciences, engineering and computer sciences.
Therefore it is both relevant and helpful, for a deeper understanding of the perturbation theory for minimization diagrams, to interpret and discuss the results and formulas of Sections~\ref{sec:abstract} and \ref{sec:MinDiags} in the particular case of Voronoi diagrams, which is the purpose of this section.
The obtained formulas will be the basis for the calculation of the gradients used in the numerical experiments of Section~\ref{numexp}.

Throughout this section we always assume that $q=2$ and $\lsf_i(x,\si_i) = \|x - \si_i\|^2$ for all $i\in\mL$.
For $x\in \overline{A}$ let us introduce the set of indices
$$ P(x) := \left\{ i \in\mL \text{ such that } x\in \overline{V_i(\bsi)}\right\}.$$
We clearly have $|P(x)|\geq 1$ for all $x\in \overline{A}$.
If $|P(x)|= 1$, then either $x$ belongs to some cell $V_i(\bsi)$ or $x\in\partial A$.

First of all we observe that Assumption~\ref{assump1:MDbd}, considering Remark~\ref{rem:assump1:MDbd}, reduces to  $\|\nabla_x \vp(x)\|>0$ for all $x\in \partial A$ and to the   condition $\|\si_i -\si_j\|>0$ for all $\{i,j\}\subset\mL$, which is independent of $x$.
The latter condition of well-separated sites also derives from Assumption~\ref{assump2:MDbd}, as can be seen in the following result.
\begin{lemma}\label{lem:vor1}
Suppose that Assumption~\ref{assump2:MDbd} holds, then the sites $\{\si_i\}_{i\in\mL}$ are pairwise distinct. In addition, we have  $|P(x)|\leq 3$ for all $x\in A$ and $|P(x)|\leq 2$ for all $x\in \partial A$. 
\end{lemma}
\begin{proof}
Suppose $\si_i =\si_j$ for some $i\neq j$, then we have $\nabla_x \lsf_i(x,\si_i) -\nabla_x \lsf_j(x,\si_j) = 0$ for all $x\in\R^2$ and  Assumption~\ref{assump2:MDbd} could not hold. 

Now suppose that $|P(x)|> 3$ for some $x\in A$. Then there exist indices $\{i,j,k,\ell\}\subset\mL$ such that 
$\nabla_x \lsf_i(x,\si_i) 
= \nabla_x \lsf_j(x,\si_j)
= \nabla_x \lsf_k(x,\si_k)
= \nabla_x \lsf_\ell(x,\si_\ell)$, but this is incompatible with condition \eqref{199}.
In a similar way, $|P(x)|> 2$ for some $x\in \partial A$ is incompatible with condition \eqref{199}.
\end{proof}

Lemma~\ref{lem:vor1} provides an  interpretation of Assumptions~\ref{assump2} and \ref{assump2:MDbd} in the context of Voronoi diagrams.
It basically states that Assumption~\ref{assump2} (or Assumption~\ref{assump2:MDbd}) eliminates the trivial situations where two cells are identical. 
The second result of Lemma~\ref{lem:vor1} shows that the vertices of the Voronoi diagram belong to no more than three cells.
If Assumption~\ref{assump2:MDbd} does not hold, then vertices may belong to four or more cells and a new edge may appear after a small perturbation, a singular case that requires a specific asymptotic analysis. 
These configurations are ``rare'' in the sense that they represent a set of zero measure in $\R^2$, and an arbitrary small perturbation of the sites allows to avoid them when they occur.

In Section~\ref{sec:sensitivity_vertices} we have introduced the set of interior vertices $\trp_{ijk}$  and the set of boundary vertices  $\dop_{ij}$ of the diagram.
The following result is an immediate consequence of Lemma~\ref{lem:vor1}. 
\begin{lemma}\label{lem:009}
Suppose that Assumption~\ref{assump2:MDbd} holds, then for all  $\{i,j,k\}\subset \mL$ we have $|\trp_{ijk}|\leq 1$ and  $|\dop_{ij}|\leq 1$.
\end{lemma}
Lemma \ref{lem:009} states that in the particular case of Voronoi, the intersection of three cells is at most one point, and the intersection of two cells and of the boundary of $A$ is at most one point. 
This well-known fact illustrates  Assumptions~\ref{assump2} and \ref{assump2:MDbd} in a particular case.
Note that the results of Lemma \ref{lem:009} do not hold in general for minimization diagrams.

Now we describe Theorems~\ref{thm:perturbMD1} and \ref{thm:perturbMD2} for the particular case of Voronoi diagrams.

\begin{thm}
Suppose Assumption \ref{assump2:MDbd} holds and $|\trp_{ijk}|=1$ for some  $\{i,j,k\}\subset \mL$.
Then, denoting  $v= \trp_{ijk}$, there exists $\tau_1>0$ and  a unique smooth function $z_v:[0,\tau_1]\to \R^{2}$ satisfying $z_v(0)=v$ and
\begin{align}\label{triple_perturb:vor} 
z_v'(0)  
& =
M(j,k,i) \dsi_i + M(k,i,j) \dsi_j + M(i,j,k) \dsi_k  
\end{align}
where
\begin{align}\label{M:vor} 
 M(i,j,k) := \frac{(\si_i -\si_j)^\perp\otimes (v -\si_k)^\top}{Q(i,j,k) } 
\end{align}
and
$$ Q(i,j,k) := 
\det\begin{pmatrix}
(\si_j -\si_i)^\top\\
(\si_k -\si_i)^\top
\end{pmatrix}.$$
\end{thm}
\begin{proof}
The result follows by applying  Theorem~\ref{thm:perturbMD1} with $\lsf_\ell(x,\si_\ell) = \|x - \si_\ell\|^2$, $\ell=i,j,k$.
\end{proof}

\begin{thm}
Suppose Assumption \ref{assump2:MDbd} holds and $|\dop_{ij}|=1$ for some  $\{i,j\}\subset \mL$.
Then, denoting  $v= \dop_{ij}$, there exists $\tau_1>0$ and a unique smooth function $z_v:[0,\tau_1]\to \R^{2}$ satisfying $z_v(0)=v$ and
\begin{align}\label{wprime_bd:vor}
z_v'(0) = \mm(j,i) \dsi_i + \mm(i,j) \dsi_j 
\end{align}
with 
\begin{equation} \label{calM:vor}
\mm(j,i): = \frac{-\nabla_x \vp(v)^\perp \otimes (v-\si_i)^\top}{\det\begin{pmatrix}
 (\si_j -\si_i)^\top\\
 \nabla_x\vp(v)^\top
\end{pmatrix}
}
.
\end{equation}
\end{thm}
\begin{proof}
The result follows by applying  Theorem~\ref{thm:perturbMD2} with $\lsf_\ell(x,\si_\ell) = \|x - \si_\ell\|^2$, $\ell=i,j$. 
\end{proof}

Next, we compute the gradient of $G_1(\bsi)$ of Section~\ref{sec:G1} in the particular case of Voronoi diagrams. 
Taking $\lsf_\ell(x,\si_\ell) = \|x - \si_\ell\|^2$, $\ell=i,k$, we obtain
\begin{align*}
\nabla G_1(\bsi) \cdot\dbsi
& = \sum_{E\in\hme}  \frac{\delta\si_i}{\|\si_i - \si_{k(i,E)}\|}\cdot\int_{E} f(x) (x-\si_i)dx 
-\frac{\delta\si_{k(i,E)}}{\|\si_i - \si_{k(i,E)}\|}\cdot
\int_E f(x)(x-\si_{k(i,E)}) dx.
\end{align*}
Let $v_E$ and $w_E$ denote the vertices of $E$ with respect to a counterclockwise orientation on $V_i(\bsi)$. 
In the particular case $f \equiv  1$ we compute
\begin{align*} 
\nabla G_1(\bsi) \cdot\dbsi &= \sum_{E\in\hme}  \frac{\delta\si_i}{\|\si_i - \si_{k(i,E)}\|}\cdot\int_{E}  (x-\si_i)dx 
-\frac{\delta\si_{k(i,E)}}{\|\si_i - \si_{k(i,E)}\|}\cdot
\int_E (x-\si_{k(i,E)}) dx\\
& = \sum_{E\in\hme}  \frac{|E|\delta\si_i}{\|\si_i - \si_{k(i,E)}\|}\cdot \left[ \frac{v_E +w_E}{2} - \si_i  \right] 
-\frac{|E|\delta\si_{k(i,E)}}{\|\si_i - \si_{k(i,E)}\|}\cdot \left[ \frac{v_E +w_E}{2} - \si_{k(i,E)}  \right].
\end{align*}
Introducing the midpoint $p_E := (v_E + w_E)/2$ of $E$, this yields
\begin{align*} 
\nabla G_1(\bsi) \cdot\dbsi & = \sum_{E\in\hme}  \frac{|E|}{\|\si_i - \si_{k(i,E)}\|} [\delta\si_i\cdot ( p_E- \si_i  ) 
-\delta\si_{k(i,E)}\cdot ( p_E - \si_{k(i,E)} )].
\end{align*}


Next, we compute the gradient of $G_2(\bsi)$ of Section~\ref{sec:bd_int} in the particular case of Voronoi diagrams and $f\equiv 1$. 
In this case $\mathcal{H}=0$, $\partial_\nu f \equiv 0$ and each edge $E$ has only one connected component, hence 
\begin{equation*}
\nabla G_2(\bsi)\cdot\dbsi = \theta(w_E)\cdot \tau(w_E) - \theta(v_E)\cdot \tau(v_E).
\end{equation*}
Considering that $\theta(v)\cdot \tau(v)$ is given by 
\eqref{der_tangential:walls} if $v\in A$ and by  \eqref{der_tangential:walls2} if $v\in\partial A$, we get
\begin{equation}\label{grad:G2_vor}
\nabla G_2(\bsi)\cdot\dbsi = \mathcal{F}(i,w_E) \cdot\tau(w_E) - \mathcal{F}(i,v_E)\cdot\tau(v_E),
\end{equation}
where
\begin{equation}\label{grad:G2_vor:F}
\mathcal{F}(i,v) :=\left\{
\begin{array}{ll}
 M_v(j,k,i) \dsi_i 
+ M_v(k,i,j) \dsi_j 
+ M_v(i,j,k) \dsi_k  & \mbox{if } v \in \trp_{ijk},\\
 \mm_v(j,i)  \dsi_i + \mm_v(i,j) \dsi_j & \mbox{if } v \in \dop_{ij}.
\end{array}
\right.
\end{equation}
Note that in \eqref{grad:G2_vor:F}, the indices $j,k$ in $\trp_{ijk}$ and the index~$j$ in $\dop_{ij}$ actually depend on the index $i$ and on the vertex~$v$. 
These indices may be uniquely determined  by choosing a counterclockwise orientation of the cells around the vertex $v$.

\section{Numerical experiments} \label{numexp}
In this section we assume that $A$ is open and polygonal, and we show numerical experiments to illustrate the application of the developed theory to the specific case of Voronoi diagrams. 
The optimization of Voronoi diagrams is highly relevant in applications, in particular for mesh optimization \cite{MR1722997,NMTMA-3-119,sieger}. 
The advantage of our approach is to provide a general framework for computing  derivatives for a wide class of cost functions and generalized Voronoi diagrams, and also to provide a sensitivity analysis for cells, edges and vertices on the boundary of $A$.
In Section~\ref{samevol}, we addressed the problem of constructing Voronoi diagrams with cells of equal volume. In Section~\ref{smalledges}, we show how to avoid cells with very small edges. In Section~\ref{sharpangles}, we show how to avoid sharp angles. In Section~\ref{midpoints}, we deal with approximating the midpoint of a Voronoi edge with the midpoint of the corresponding Delaunay edge. In Section~\ref{J2}, we show how to get cells with different pre-specified sizes for different parts of the region~$A$.

Considering that $A$ is a polygon, define $\me_i^{\partial A}$  the set of edges of $\overline{ V_i(\bsi)}\cap\partial A$.
Let $\hme$ denote the set of interior edges of the cell $\overline{V_i(\bsi)}$, i.e., edges that are included in $A$.
Define $\me_i := \hme \cup \me_i^{\partial A}$ as the set of edges of the cell $\overline{V_i(\bsi)}$.  
Note that this definition of edges of $\overline{V_i(\bsi)}$ is not exactly the same as the definition given in Section~\ref{sec:bd_int}.
Indeed, in the previous sections we have assumed that $A$ was smooth for convenience, and boundary edges  could be defined directly as $\overline{ V_i(\bsi)}\cap\partial A$, whereas $\overline{ V_i(\bsi)}\cap\partial A$ may have several edges when  $A$ is a polygon. 
Also, the edges of $A$ need to be taken into account for defining the cost functions and their gradients in this section.

So far we have defined the cells $V_i(\bsi)$ as subset of $A$. 
For the numerical implementation we also introduce Voronoi cells relative to the plane: 
$$W_i(\bsi):= \inte\left\{x \in \R^2 \text{ such that } \| x - \si_i\|^2  \leq  \| x - \si_k\|^2\text{ for all } k\in\mL\setminus\{i\}  \right\}.$$ 
Clearly, $V_i(\bsi) = W_i(\bsi) \cap A$.

\subsection{Identical volume cells} \label{samevol}

Initially, we consider the merit function given by
\[
J^1(\bx) := \frac{1}{\nsites} \sum_{i=1}^{\nsites} \left[ J^1_i(\bx) \right]^2,
\]
where
\[
J^1_i(\bx) := \left( \int_{V_i(\bx)} dx \right) \big/ 
\left( \frac{1}{\nsites}\int_A dx \right) - 1,
\]
that measures the deviation of the area of the Voronoi cell respect to the  average area of the cells in the domain~$A$. For future reference, note that~$J^1(\bx) \leq 10^{-8}$ means that, on average, $|J^1_i(\bx)| \leq 10^{-4}$. This means that, on average, the relative error of the area of a cell in relation to the ideal area of a cell is less than 0.01\%.

As the Voronoi diagrams are not well defined if two sites coincide, it is useful in practice to ask the sites to keep a certain distance between each other. 
Therefore, the minimization of~$J^1$ can be combined with the minimization of~$J^0$ given by
\[
J^0(\bx) := \sum_{i=1}^{\nsites} \sum_{j=i+1}^{\nsites} \max \left\{ 0, J^0_{ij}(\bx) \right\}^2
\;\mbox{ with }\; J^0_{ij}(\bx) := \delta^2 - \| a_i - a_j \|^2,
\]
where $\delta>0$ is a small pre-established tolerance. $J^0$ nullifies when all pairs of sites are at least $\delta$ apart, i.e., when they represent the centers of non-overlapping balls of radius $\delta/2$. At first glance, the computation of~$J^0$ has time complexity~$O(\nsites^2)$. However, it is expected that sites of non-neighboring cells have distance greater than~$\delta$ and, therefore, do not contribute to the computation of~$J^0$. Thus, in practice, it is reasonable to compute the terms in~$J^0$ that correspond only to pairs~$(i,j)$ of neighbor cells, reducing its evaluation cost to~$O(\nsites)$. 

Following \cite{coveringsecond,coveringfirst}, we consider examples of regions~$A$ given by the union of disjoint convex polygons $A_1, \dots, A_p$. In that way, we can represent non-convex regions and regions with holes. Given the sites~$\bx \in \mathbb{R}^{2\nsites}$, we first compute the Delaunay diagram using the \textsc{Dtris2} subroutine from Geompack~\cite{geompack} (available at \url{https:// people.math.sc.edu/Burkardt/f_src/geompack2/geompack2.html}) and, from that, the Voronoi diagram $\mathcal{W}(\bx)=\{ W_i(\bx) \mbox{ for } i=1,\dots,\nsites\}$. Each cell of the Voronoi diagram is a polyhedron (which can be unbounded). For each $W_i(\bx)$, we compute $V_i(\bx) = W_i(\bx) \cap A$ as $V_i(\bx) = \cup_{j=1}^p V_{ij}(\bx)$, where $V_{ij}(\bx) = W_i(\bx) \cap A_j$. Each $V_{ij}(\bx)$ is the intersection of a polyhedron with a convex polygon and is computed using an adaptation of Sutherland-Hodgman algorithm~\cite{sutherland1974}. With this information, the area of $V_i(\bx)$ in $J^1_i$ is trivially calculated as the sum of the areas of the corresponding polygons $V_{ij}(\bx)$. At this point it is worth noting that the interpretation of~$J^1$ may not correspond precisely to what one imagines at first, since for non-convex regions~$A$ some cells $V_i(\bx)$ may be disconnected. 

Code was written in Fortran~90. All tests were conducted on a computer with a 3.4 GHz Intel Core i5 processor and 8GB 1600 MHz DDR3 RAM memory, running macOS Mojave (version 10.14.6). Code was compiled by the GFortran compiler of GCC (version 8.2.0) with the -O3 optimization directive enabled. In $J^0$, we considered $\delta=0.1$. 
This value is appropriate since the considered regions~$A$ were scaled so that $|A| \approx \nsites$, i.e., at a solution~$\bx$ it is expected that $|V_i(\bx)| \approx 1$ for all~$i$. The optimization problems were solved using the Spectral Projected Gradient (SPG)  method~\cite{bmr,bmracmtoms,bmrenc,bmrreview}.
As anticipated in the discussion above, the stopping criterion was to achieve a value of the objective function less than or equal to $10^{-8}$. The initial point~$\bx^0$ of the optimization process was constructed by drawing points in~$A$ with uniform distribution. 
In fact, from the description of~$A$, it is possible to establish the smallest rectangle~$D$ such that $A \subseteq D$. Uniformly distributed points are drawn in~$D$ until~$\nsites$ points $a_1,\dots,a_{\nsites} \in \mathbb{R}^2$ belonging to $A$ are obtained. These points constitute the initial guess~$\bx^0$.

Figures~\ref{fig1}--\ref{fig3} show the result of minimizing $f(\bx) := 10 J^0(\bx) + J^1(\bx)$ subject to $\bx \in \mathbb{R}^{2\nsites}$ with $\nsites \in \{100, 1000\}$ for some non-convex regions~$A$. The behavior of the optimization method varies slightly depending on the relative weight attributed to~$J^0$ and ~$J^1$. The weight~10 for~$J^0$ was obtained empirically. Non-convex regions ``America'', ``Ces\`aro Fractal'', ``Minkowski Fractal'', ``Star'', and ``Polygon with Holes'' were already considered in~\cite{coveringsecond,coveringfirst}, where a detailed description can be found. Non-convex regions ``Letter A'' and ``Key'' were inspired by~\cite[Fig.2]{sieger}. The detailed description of our free interpretation of these figures can be found in Appendix~B, as well as the description of the other two simple convex regions considered. The figures show that many cells $V_i(\bx)$ are non-convex, mainly near the borders of $A$. 
Disconnected cells $V_i(\bx)$ are rare, but do exist. See for example the Gulf of California or Tierra del Fuego regions in the map of America with $\nsites=100$. In the second case, two disconnected parts of~$A$ are covered by the same Voronoi cell, resulting, evidently, in a disconnected $V_i(\bx)$. These two cases correspond to $\nsites=100$. When $\nsites=1000$, the cells are small compared to the edges of the polygons defining $A$ and, therefore, the cases are rarer or nonexistent.

Table~\ref{tab1} shows some details of the optimization process. In the table, \textbf{scaling factor} corresponds to the scale factor the region $A$ was multiplied by,  so that its volume is approximately equal to $\nsites$, i.e., that the cells have ideal area $|A|/\nsites$ of approximately~1, while ${\bf |A|}$ corresponds to the actual volume of $A$. Column $\boldsymbol{p}$ corresponds to the number of convex polygons defining $A$. Column~$\boldsymbol{\nsites}$ corresponds to the considered number of sites and column \textbf{ntrials} corresponds to the number of different initial guesses (limited to~10) that were used in the optimization process until a final iterate with a functional value smaller than or equal to $10^{-8}$ was found. \textbf{it} identifies the number of iterations, \textbf{fcnt} identifies the number of evaluations of the objective function, and \textbf{Time} identifies the elapsed CPU time in seconds. The number of gradient evaluations coincides with the number of iterations plus~1. The columns $f(\bx^*)$ and $\| \nabla f(\bx^*) \|_{\infty}$ identify the value of the objective function and the sup-norm of the gradient at the final iterate~$\bx^*$. 

The figures in Table~\ref{tab1} show that the considered problems could be solved using a simple, available and well-established optimization algorithm with an acceptable effort. Moreover, the target functional value was found in 7 out of the 9 considered problem starting from a single initial guess. In one problem (``Letter A'' with $\nsites=100$), a functional value smaller than $10^{-8}$ was found in the third trial; while in one problem (``America'' with $\nsites=1000$) the smallest functional value was found in the fourth trial, but no value below $10^{-8}$ was found. In any case, small functional values were found in all problems with an affordable computational effort. This performance is in accordance with the solution of a large practical problem for which multiple runs might be unaffordable.

\begin{figure}[ht!]
\begin{center}
\begin{tabular}{cc}
\includegraphics{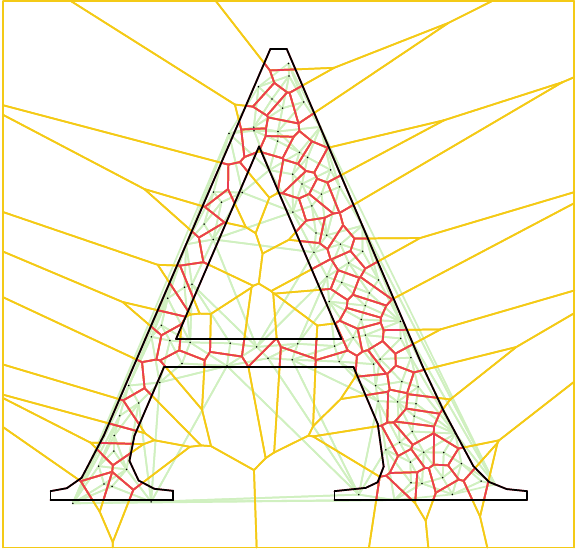} &
\includegraphics{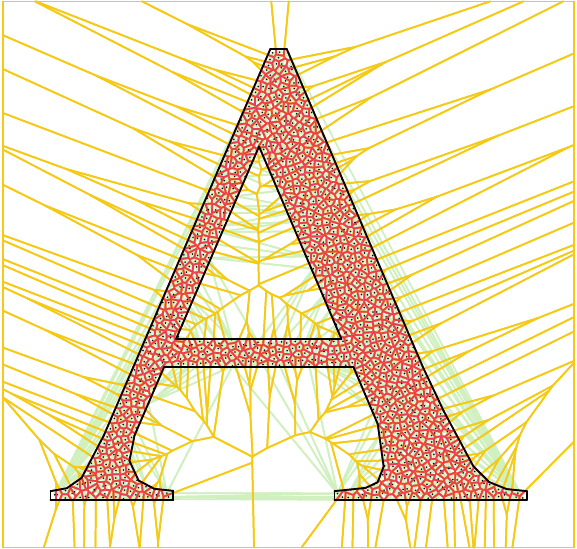} \\
\includegraphics{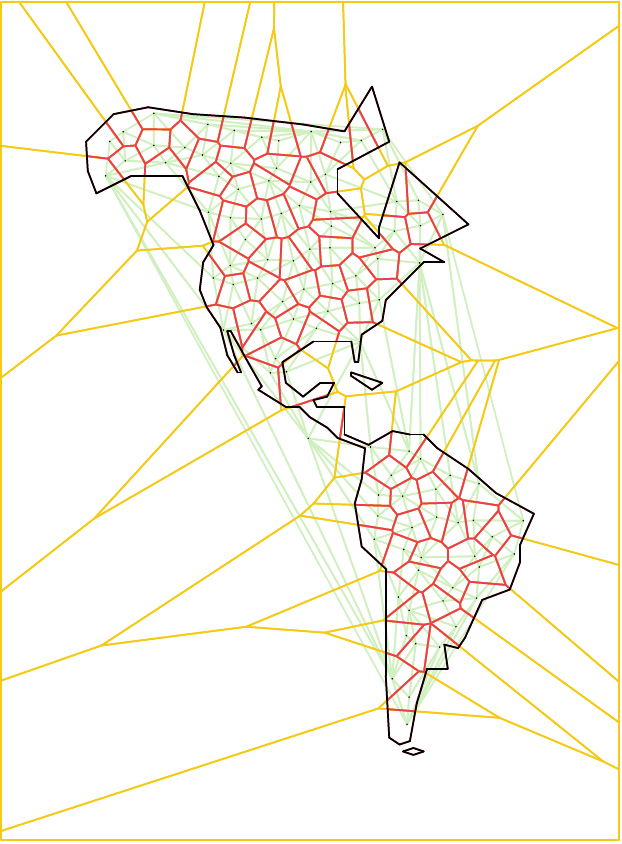} &
\includegraphics{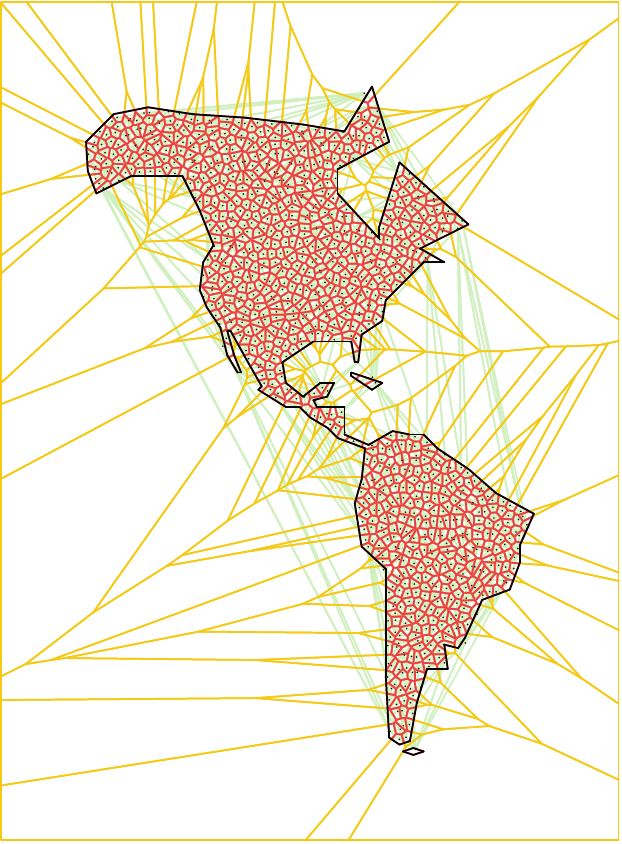} \\
\includegraphics{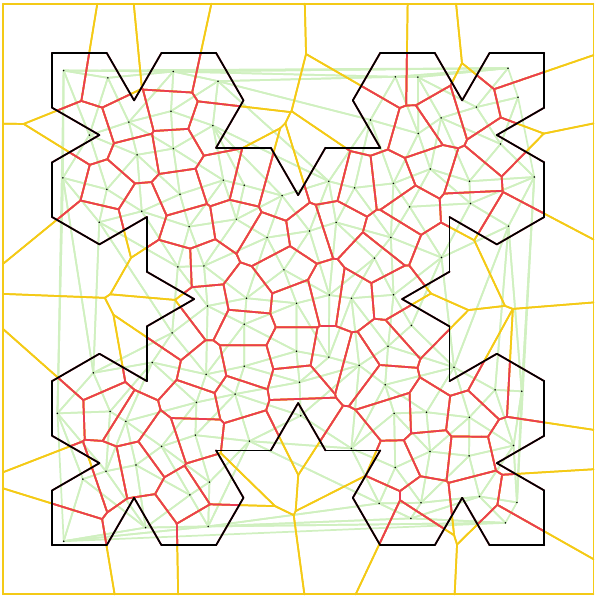} &
\includegraphics{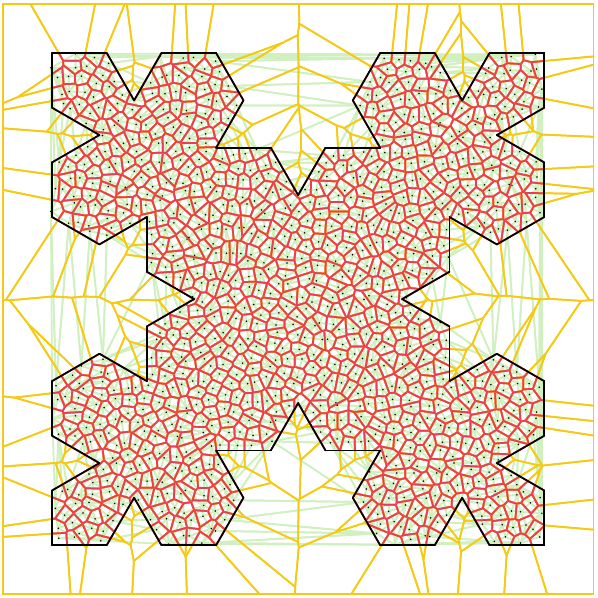} \\
\end{tabular}
\end{center}
\caption{Voronoi diagrams with $\nsites \in \{100, 1000 \}$ cells of identical area. Regions ``Letter A'', ``America'', and ``Ces\`aro Fractal''.}
\label{fig1}
\end{figure}

\begin{figure}[ht!]
\begin{center}
\begin{tabular}{cc}
\includegraphics{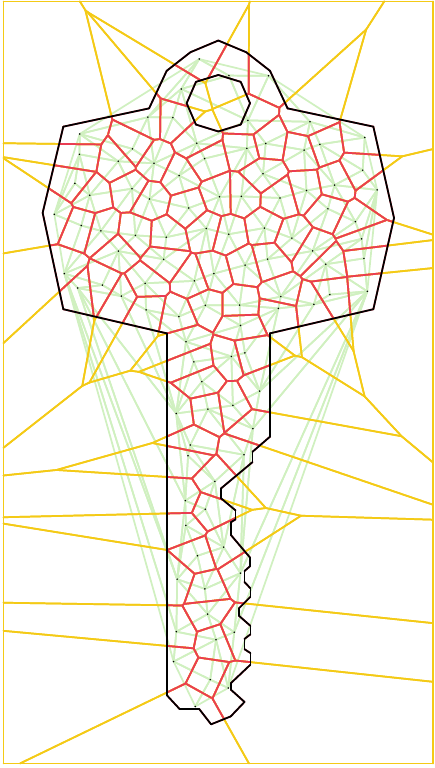} &
\includegraphics{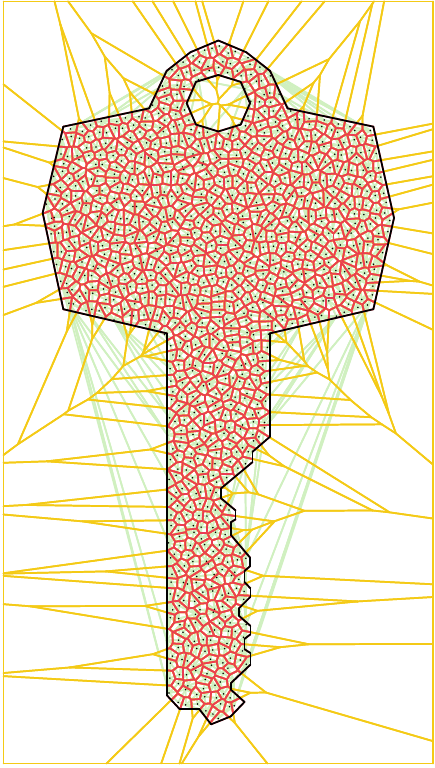} \\
\includegraphics{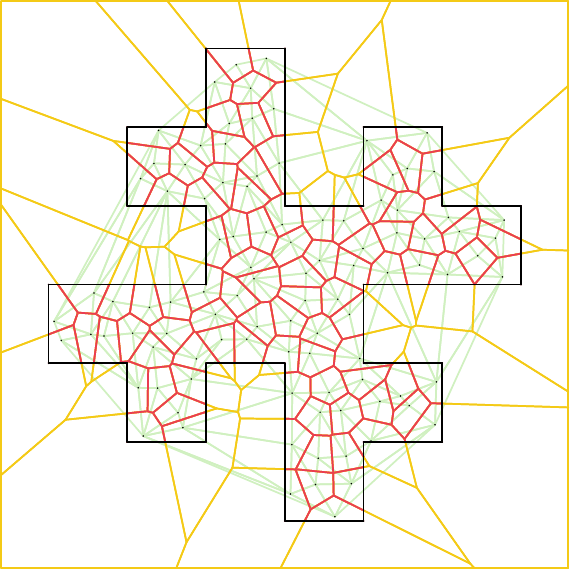} &
\includegraphics{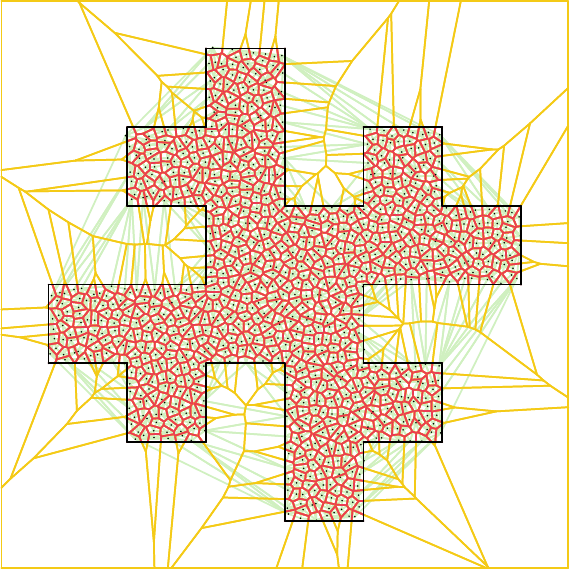} \\
\includegraphics{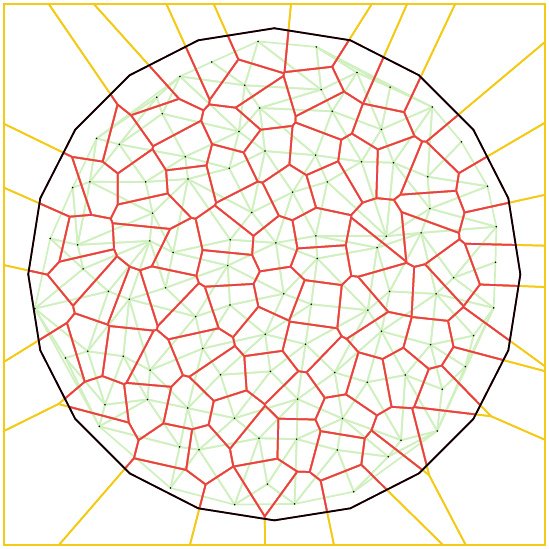} &
\includegraphics{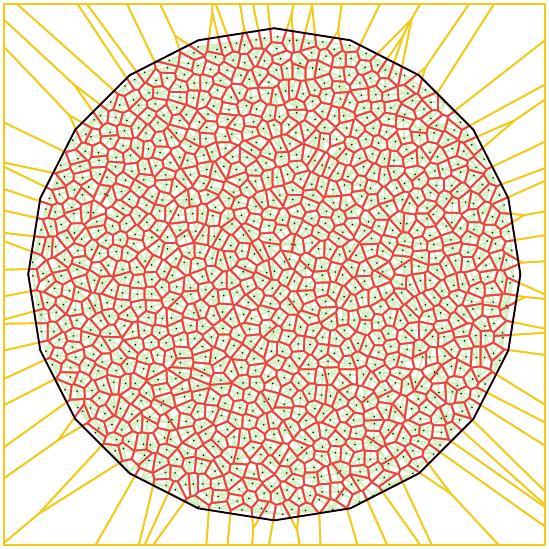} \\
\end{tabular}
\end{center}
\caption{Voronoi diagrams with $\nsites \in \{100, 1000 \}$ cells of identical area. Regions ``Key'', ``Minkowski Fractal'', and ``Regular Polygon''.}
\label{fig2}
\end{figure}

\begin{figure}[ht!]
\begin{center}
\begin{tabular}{cc}
\includegraphics{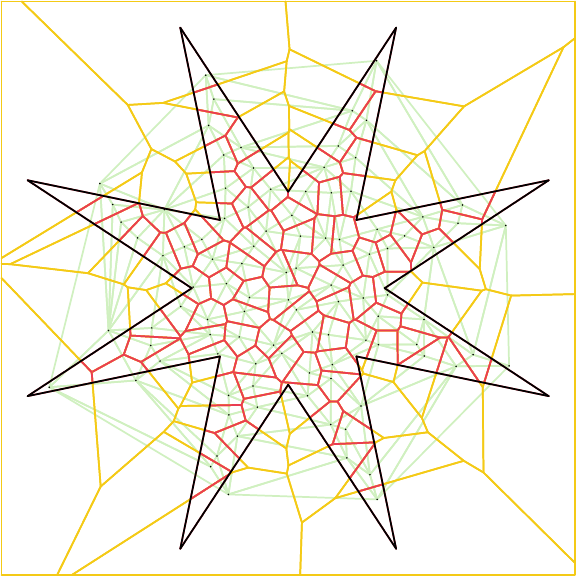} &
\includegraphics{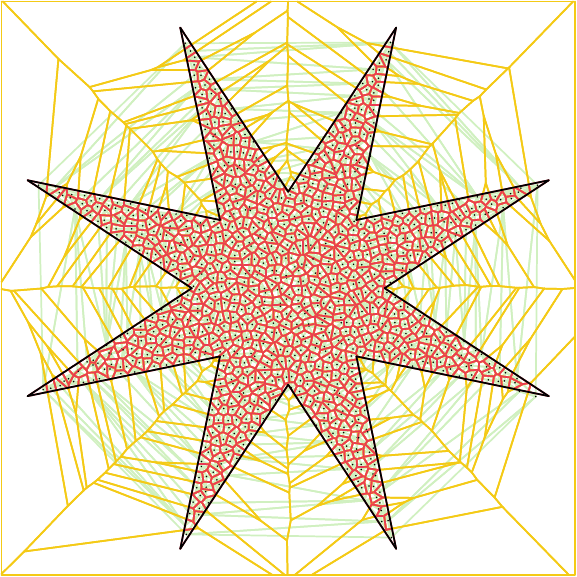} \\
\includegraphics{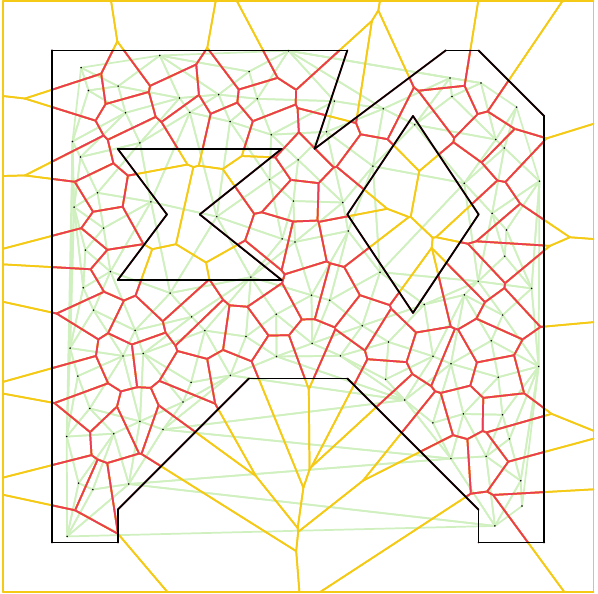} &
\includegraphics{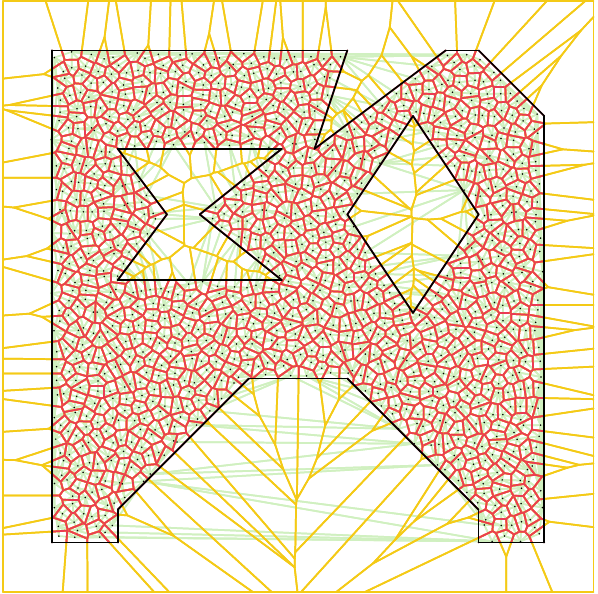} \\
\includegraphics{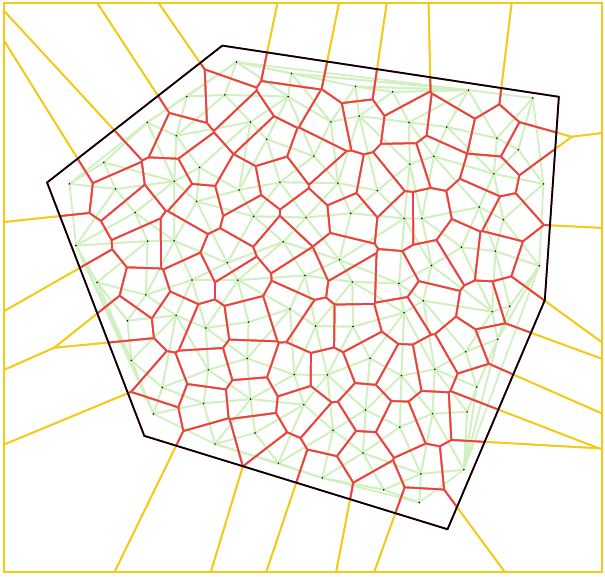} &
\includegraphics{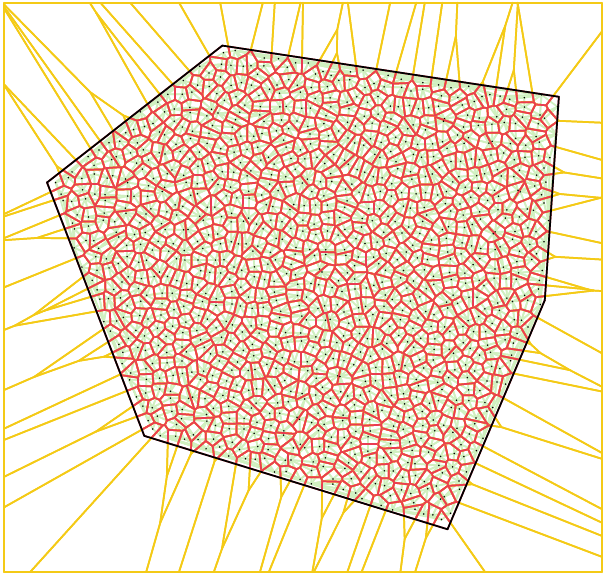} \\
\end{tabular}
\end{center}
\caption{Voronoi diagrams with $\nsites \in \{100, 1000 \}$ cells of identical area. Regions ``Star'', ``Polygon with Holes'', and ``Convex Polygon''.}
\label{fig3}
\end{figure}

\begin{table}[ht!]
\begin{center}
\resizebox{\textwidth}{!}{
\begin{tabular}{|c|cccccccccr|}
\hline
\multicolumn{1}{|c|}{Problem} &
scaling factor & $|A|$ & $p$ & $\nsites$ & ntrials & $f(\bx^*)$ & $\| \nabla f(\bx^*) \|_{\infty}$ & it & fcnt & Time\\
\hline
\hline
\multirow{2}{*}{\begin{tabular}{c}Convex\\[-0.5mm]Polygon\end{tabular}} 
& 3.06250E$+$00 & 1.00582E$+$02 &  1 &  100 & 1 & 9.09410E$-$09 & 5.2E$-$06 &   39 &   41 &   0.02\\
& 9.67188E$+$00 & 1.00320E$+$03 &  1 & 1000 & 1 & 9.84241E$-$09 & 5.8E$-$07 &  258 &  346 &   1.68\\
\hline
\multirow{2}{*}{\begin{tabular}{c}Regular\\[-0.5mm]Polygon\end{tabular}} 
& 5.70312E$+$00 & 1.00510E$+$02 &  1 &  100 & 1 & 9.62359E$-$09 & 9.5E$-$07 &   54 &   57 &   0.03\\
& 1.79531E$+$01 & 9.96007E$+$02 &  1 & 1000 & 1 & 7.37442E$-$09 & 2.5E$-$06 &  228 &  287 &   1.64\\
\hline
\multirow{2}{*}{Letter A} 
& 6.56250E$-$01 & 1.00143E$+$02 & 16 &  100 & 3 & 9.76781E$-$09 & 1.7E$-$06 &  350 &  450 &   0.67\\
& 2.07812E$+$00 & 1.00421E$+$03 & 16 & 1000 & 1 & 9.16932E$-$09 & 7.4E$-$07 & 1209 & 1849 &  27.37\\
\hline
\multirow{2}{*}{America} 
& 1.18750E$+$00 & 9.91234E$+$01 & 34 &  100 & 1 & 9.51486E$-$09 & 6.0E$-$07 &  866 & 1368 &   3.98\\
& 3.76562E$+$00 & 9.96743E$+$02 & 34 & 1000 & 4 & 2.14619E$-$05 & 9.8E$-$09 & 5603 & 9241 & 267.93\\
\hline
\multirow{2}{*}{\begin{tabular}{c}Ces\`aro\\[-0.5mm]Fractal\end{tabular}} 
& 1.17812E$+$01 & 1.00214E$+$02 & 21 &  100 & 1 & 9.66216E$-$09 & 2.9E$-$06 &   58 &   68 &   0.13\\
& 3.71758E$+$01 & 9.97855E$+$02 & 21 & 1000 & 1 & 9.98247E$-$09 & 7.3E$-$08 &  485 &  706 &  13.79\\
\hline
\multirow{2}{*}{Key} 
& 1.06250E$+$00 & 9.95155E$+$01 & 22 &  100 & 1 & 6.63147E$-$09 & 1.0E$-$06 &  135 &  165 &   0.37\\
& 3.36719E$+$00 & 9.99464E$+$02 & 22 & 1000 & 1 & 9.86242E$-$09 & 1.5E$-$07 &  728 & 1101 &  22.28\\
\hline
\multirow{2}{*}{\begin{tabular}{c}Minkowski\\[-0.5mm]Fractal\end{tabular}} 
& 2.50000E$+$00 & 1.00000E$+$02 & 16 &  100 & 1 & 8.02930E$-$09 & 1.7E$-$06 &   77 &   85 &   0.14\\
& 7.89062E$+$00 & 9.96191E$+$02 & 16 & 1000 & 1 & 6.94410E$-$09 & 3.2E$-$06 &  528 &  752 &  11.59\\
\hline
\multirow{2}{*}{Star} 
& 2.54688E$+$00 & 9.91209E$+$01 &  9 &  100 & 1 & 5.52897E$-$09 & 8.7E$-$06 &  119 &  133 &   0.15\\
& 8.08984E$+$00 & 1.00007E$+$03 &  9 & 1000 & 1 & 9.76350E$-$09 & 3.5E$-$07 &  467 &  621 &   6.80\\
\hline
\multirow{2}{*}{\begin{tabular}{c}Polygon\\[-0.5mm]with Holes\end{tabular}} 
& 1.20156E$+$01 & 9.97793E$+$01 & 14 &  100 & 1 & 8.39601E$-$09 & 7.0E$-$06 &  103 &  123 &   0.19\\
& 3.79141E$+$01 & 9.93456E$+$02 & 14 & 1000 & 1 & 8.43135E$-$09 & 7.2E$-$07 &  419 &  576 &   8.57\\
\hline
\end{tabular}}
\end{center}
\caption{Details of the optimization process and the solutions found for the problem of finding Voronoi diagrams with cells of equal volume.}
\label{tab1}
\end{table}

We close this section by showing how the method behaves when the problem size increases, i.e. when the number of cells grows. Table~\ref{tab1b} shows details of the solutions obtained and the performance of the method when applied to the region $A$ given by a regular polygon with $\nsites \in \{ 100, 500, 1{},000, 5{,}000, 10{,}000$, $20{,}000, 30{,}000, 40{,}000, 50{,}000 \}$. Column \textbf{fcnt/it} shows that the number of function evaluations per iteration is, on average, smaller than $1.5$, regardless of~$\nsites$. At this point it is perhaps interesting to mention that the computation of the objective function and its gradient share many operations, among them the construction of the Voronoi diagram, which is the dominant cost. Therefore, among the several possible options, we opted for computing them together. At iteration $k$, being at the current point $\bx^k$, the SPG method calculates $\bx^{k,\mathrm{trial}}$ and, if the value of the merit function at that point is considered acceptable, it defines $\bx^{k+1} := \bx^{k,\mathrm{trial}}$. (Otherwise, the method starts a backtracking process to calculate a new point closer to $\bx^k$.) Whenever a new iteration starts, the gradient at the new current point $\bx^{k+1}$ is necessary; see~\cite{bmr,bmracmtoms,bmrenc,bmrreview} for details. If, when calculating the merit function at $\bx^{k,\mathrm{trial}}$, we already calculated the gradient together, then the joint calculation of function and gradient can be used. The value, in average, smaller than 1.5 in column \textbf{fcnt/it} suggests, as already known in the literature, that the method goes from $\bx^k$ to $\bx^{k+1}$ with a single function evaluation in about half of the iterations, making the joint evaluation of function and gradient profitable. Regarding the cost, as a function of $\nsites$, of the routine that evaluates function and gradient, it follows, as expected, the cost $O( \nsites \log(\nsites))$ related to the computation of the Voronoi diagram; see, for example, \cite{deberg}. This is shown in the last column, which presents the value of $c$ given by \textbf{Time/fcnt} divided by $\nsites \log_{10}(\nsites)$. The column shows that $c \approx 2 \times 10^{-6}$ seconds independently of~$\nsites$. (The case $\nsites=100$ should be ignored, since the measurement of such small times is subject to large relative measurement errors.)

\begin{table}[ht!]
\begin{center}
\resizebox{\textwidth}{!}{
\begin{tabular}{|ccc|ccccr|cc|}
\hline
$\nsites$ & scaling factor & $|A|$ & $f(\bx^*)$ & $\| \nabla f(\bx^*) \|_{\infty}$ & it & fcnt & Time &
fcnt/it & $c$ \\
\hline
\hline
  100 & 5.70312E$+$00 & 1.00510E$+$02 & 9.62359E$-$09 & 9.5E$-$07 &   54 &   57 &   0.03 & 1.06 & 2.94E$-$06\\
  500 & 1.27188E$+$01 & 4.99886E$+$02 & 9.16129E$-$09 & 2.6E$-$06 &  137 &  156 &   0.43 & 1.14 & 2.05E$-$06\\
 1000 & 1.79531E$+$01 & 9.96007E$+$02 & 7.37442E$-$09 & 2.5E$-$06 &  228 &  287 &   1.62 & 1.26 & 1.88E$-$06\\
 5000 & 4.03750E$+$01 & 5.03741E$+$03 & 1.13295E$-$07 & 8.0E$-$09 &  638 &  925 &  30.01 & 1.45 & 1.75E$-$06\\
10000 & 5.71094E$+$01 & 1.00785E$+$04 & 6.91039E$-$07 & 8.6E$-$09 & 1131 & 1725 & 114.84 & 1.53 & 1.66E$-$06\\
20000 & 8.03750E$+$01 & 1.99629E$+$04 & 5.40638E$-$06 & 8.9E$-$09 &  927 & 1385 & 188.07 & 1.49 & 1.58E$-$06\\
30000 & 9.82344E$+$01 & 2.98201E$+$04 & 1.16004E$-$05 & 9.8E$-$09 &  673 &  969 & 216.83 & 1.44 & 1.67E$-$06\\
40000 & 1.14109E$+$02 & 4.02369E$+$04 & 1.36545E$-$05 & 8.6E$-$09 &  785 & 1134 & 342.17 & 1.44 & 1.64E$-$06\\
50000 & 1.27008E$+$02 & 4.98475E$+$04 & 3.66230E$-$05 & 8.3E$-$09 &  668 &  965 & 363.70 & 1.44 & 1.60E$-$06\\
\hline
\end{tabular}}
\end{center}
\caption{Details of the optimization process and the solutions found for the problem of finding Voronoi diagrams with cells of equal volume in the region $A$ given by a regular polygon with increasing values~of~$\nsites$.}
\label{tab1b}
\end{table}

\subsection{Avoiding cells with relatively small edges} \label{smalledges}

In this section, we consider convex regions\footnote{There would be, a priori, no limitation to apply 
the content of this section to non-convex regions. 
However, due to the way we compute $V_i(\bx) := W_i(\bx) \cap A$, when $A$ is non-convex, we have direct access to the edges of each $V_{ij}(\bx) := W_i(\bx) \cap A_j$ for $j=1,\dots,p$ instead of having access to the edges of $V_i(\bx)$. When $A$ is convex (in which case $p=1$), $V_i(\bx)$ coincides with $V_{i1}(\bx)$ for all $i$ and, thus, we have direct access to the edges of $V_i(\bx)$. This is a technical limitation that could be overcome by re-implementing this part of the software.}. 
If we analyze the cells of the regular polygon with $\nsites=100$ in Figure~\ref{fig2}, we can see that there are cells with small edges. Specifically, given a fraction $c_2 \in (0,1)$, we say that an edge $E$ of a cell $V_i$ is small if its size $|E|$ is smaller than $c_2 \; (P_i / n_i)$, where $P_i$ is the perimeter of the cell $V_i$, $n_i$ is the number of edges of the cell $V_i$, and, therefore, $P_i / n_i$ is the average size of the cell's edges. To construct Voronoi diagrams that do not have cells with relatively small edges, given a tolerance $c_2 \in (0,1)$, we consider the merit function given by
\[
J^2(\bx) := \sum_{i=1}^{\nsites} J^2_i(\bx) \;\mbox{ with }\;
J^2_i(\bx) := \frac{1}{n_i} \sum_{E \in \me_i} \min \left\{ 0, \frac{|E|}{\bar E_i} - c_2 \right\}^2,
\]
where $\me_i$ is the set of edges of the cell $V_i$, $n_i = |\me_i|$, and $\bar E_i = P_i / n_i$. Given $c_2 \in (0,1)$, if all edges $E \in \me_i$ of a cell $V_i$ satisfy $|E| \geq c_2 \; \bar E_i$, i.e., if they are at least $100\% \times c_2$ larger than the average, then $|E| / \bar E_i - c_2 \geq 0$ and, therefore, $J^2_i$ vanishes. In general, $J^2_i$ measures the average violation of the size of the edges of $V_i$ relative to the minimum desired size.

For the Voronoi diagram with $\nsites=1000$ of the Regular Polygon region shown in Figure~\ref{fig2}, Figure~\ref{fig4}a shows, painted with different tones of blue, the cells $V_i$ that satisfy  $J_i^2 > 0$ for different values of $c_2 \in \{0.1, 0.2, \dots, 0.5\}$. The darker the color of the cell $V_i$, the smaller the maximum value of $c_2$ for which $J_i^2 > 0$, i.e., the more unbalanced are the edge sizes of the cell $V_i$. The uncolored cells $V_i$ satisfy $J_i^2=0$ for the considered values of $c_2$  and are, therefore, deemed balanced. Preserving the meaning of the colors, Figures~\ref{fig4}b--f show the diagrams obtained by minimizing
\begin{equation} \label{withJ2}
f(\bx) := 10 J^0(\bx) + J^1(\bx) + J^2(\bx)
\end{equation} 
with $c_2 \in \{ 0.1, 0.2, \dots, 0.5 \}$, respectively. In the cases with~$c_2$ up to~$0.4$, the unwanted unbalanced cells were eliminated. In the case $c_2=0.5$, a single run of the optimization method was not able to find a global minimizer of $f$ and, therefore, cells $V_i$ with $J^2_i(\bx) > 0$ remained. Regardless of that, in the solutions found by minimizing~(\ref{withJ2}) with~$c_2=0.4$ and~$c_2=0.5$ (Figures~\ref{fig4}e and~\ref{fig4}f), all cells~$V_i$ satisfy $J^2_i(\bx)=0$ with $c_2=0.4$, i.e., they are considered balanced with tolerance~$c_2=0.4$ and, thus, no cell has an edge whose size is less than 40\% of the average size of the cell edges.

\begin{figure}
\begin{center}
\begin{tabular}{cc}
\includegraphics{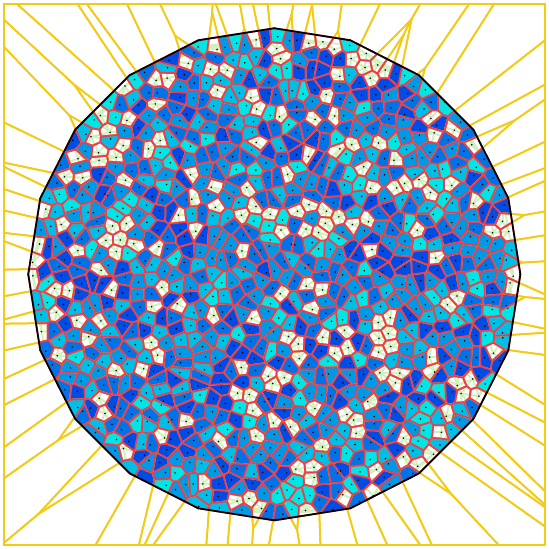} &
\includegraphics{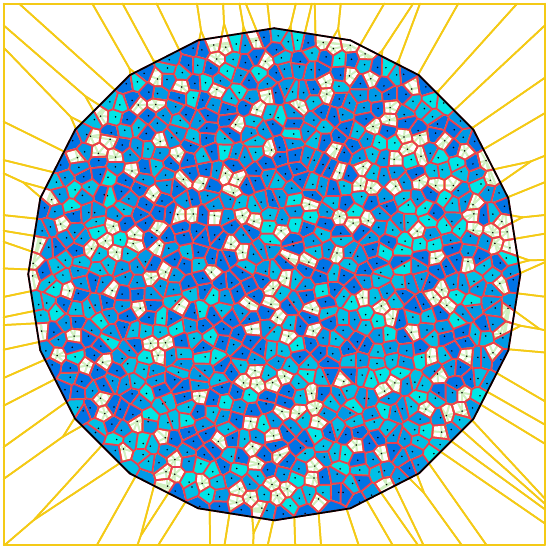} \\
(a) & (b)\\
\includegraphics{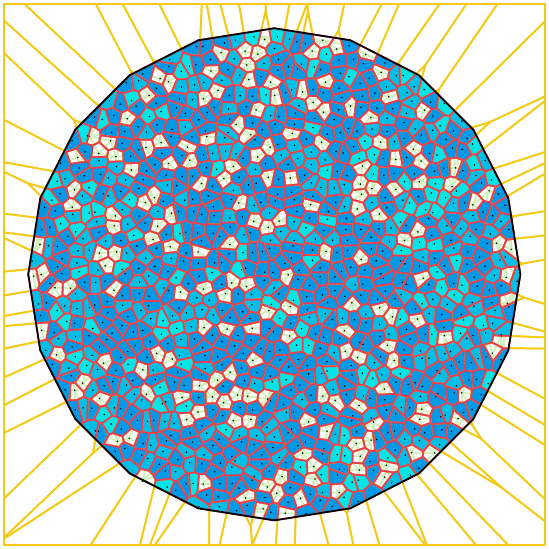} &
\includegraphics{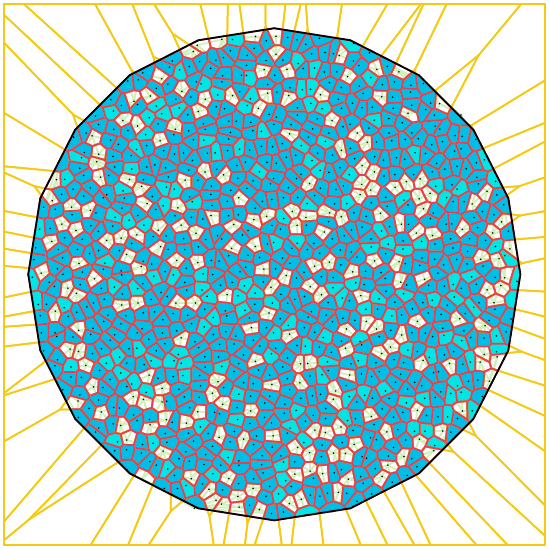} \\
(c) & (d)\\
\includegraphics{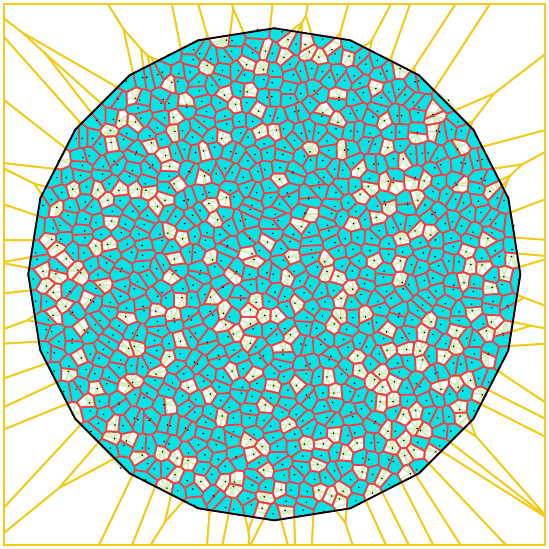} &
\includegraphics{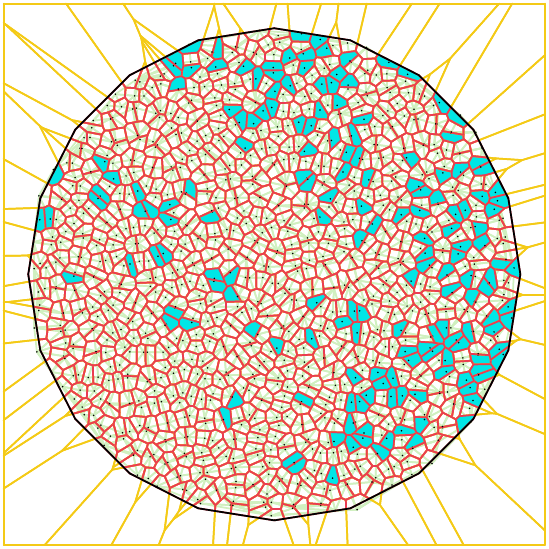} \\
(e) & (f)\\
\end{tabular}
\end{center}
\caption{Voronoi diagram with $\nsites=1000$ for the region given by a regular polygon. In (a) we show the Voronoi diagram obtained in Section~\ref{samevol}, minimizing $f(\bx) := 10 J^0(\bx) + J^1(\bx)$. The darker the cell, the more unbalanced the sizes of its edges.  In (b), (c), (d), (e) and~(f), preserving the meaning of the colors, we show the diagrams obtained by minimizing $f(\bx) := 10 J^0(\bx) + J^1(\bx) + J^2(\bx)$ with $c_2 \in \{ 0.1, 0.2, \dots, 0.5 \}$, respectively.}
\label{fig4}
\end{figure}

Figure~\ref{fig5} analyzes the six different solutions $\bx^*$ (depicted in Figure~\ref{fig4}) found by minimizing~(\ref{withJ2}) with $c_2 \in \{0.0, 0.1, \dots, 0.5\}$. For a given solution $\bx^*$, the figure shows the proportion of cells $V_i$ satisfying $J_i^2(\bx^*) = 0$ as a function of $c_2 \in [0,1]$. The case in which~(\ref{withJ2}) is minimized with~$c_2=0$ is identical to minimizing $10 J^0(\bx) + J^1(\bx)$, since $J^2(\bx)$ is identically null when $c_2=0$. For the solution found in this case, the figure shows, for example, that the statement ``all my edges are at least 20\% the average size of my edges'' is true for 60\% and that the statement ``all my edges are at least 40\% the average size of my edges'' is true for slightly more than 30\% of the cells. The figure also shows that when we minimize~(\ref{withJ2}) with $c_2=0.4$ or $c_2=0.5$, the statement ``all my edges are at least 40\% the average size of my edges'' is true for all the cells.
Table~\ref{tab2} shows details of the solutions found and the optimization process. The first column corresponds to the value of $c_2$ considered in~(\ref{withJ2}). The other columns contain the same information as Table~\ref{tab1}. The numbers in the table show that the problems that correspond to minimizing~(\ref{withJ2}) with $c_2 \in \{0.1, 0.2, 0.3\}$ were easily solved. When $c_2=0.4$, solving the problem was more expensive; and the method failed to find a solution with a value of a merit function less than $10^{-8}$ within a limit of $50{,}000$ iterations when we minimized~(\ref{withJ2}) with $c_2=0.5$.

\begin{figure}
\begin{center}
\input{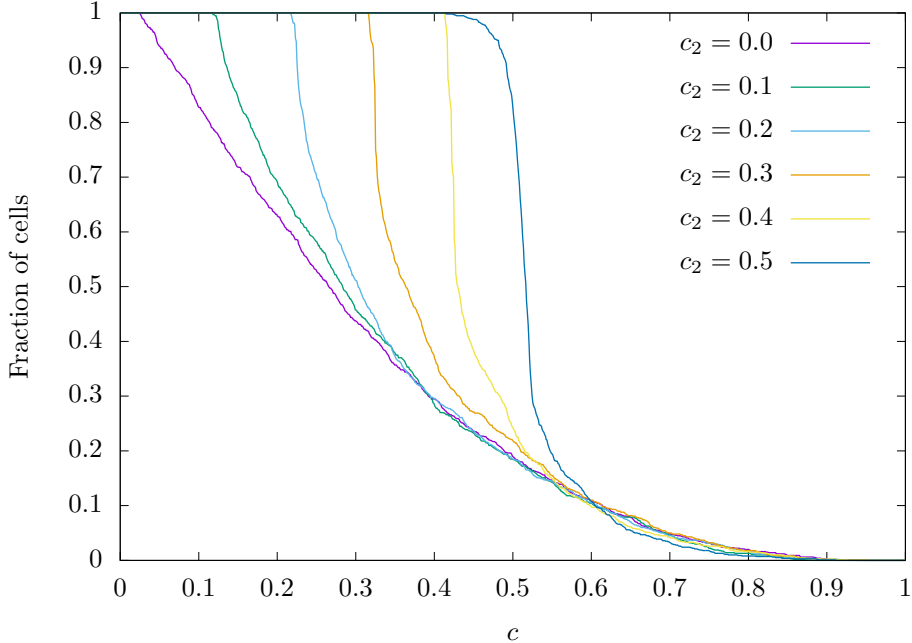}
\end{center}
\caption{This figure analyzes the solutions found when minimizing~(\ref{withJ2}) with $c_2 \in \{0.0, 0.1, \dots, 0.5\}$. (The case in which $c_2=0$ is identical to minimizing $10 J^0(\bx) + J^1(\bx)$, i.e., ignoring $J^2$.) For each solution, the figure shows, as a function of $c$, the proportion of cells that satisfy the statement ``all my edges are at least $100\% \times c$ the average size of my edges''.}
\label{fig5}
\end{figure}

\begin{table}[ht!]
\begin{center}
\begin{tabular}{|cccccr|}
\hline
$c_2$ & $f(\bx^*)$ & $\| \nabla f(\bx^*) \|_{\infty}$ & it & fcnt & Time\\
\hline
\hline
0.1 & 8.98014E$-$09 & 3.7E$-$07 &    288 &    357 &   2.07\\
0.2 & 6.17602E$-$09 & 4.3E$-$06 &    362 &    448 &   2.77\\
0.3 & 9.86248E$-$09 & 4.5E$-$08 &    519 &    628 &   3.83\\
0.4 & 9.99953E$-$09 & 5.4E$-$08 &  30304 &  46154 & 309.08\\
0.5 & 1.15592E$-$04 & 8.2E$-$06 &  50000 &  83594 & 608.91\\
\hline
\end{tabular}
\end{center}
\caption{Details of the optimization process and the solutions found for the problem of finding Voronoi diagrams with cells of equal volume and avoiding cells with relatively small edges.}
\label{tab2}
\end{table}

\subsection{Avoiding sharp-angled cells} \label{sharpangles}

The solution illustrated in Figure~\ref{fig4}e has cells of identical size for which the statement ``my edges are all at least 40\% of the average size of my edges'' holds true. In this section we focus on the balancing of the internal angles of the cells.  To construct Voronoi diagrams that do not have cells with sharp-angled cells, given a tolerance $c_3 \in (0,1)$, we consider the merit function given by
\[
J^3(\bx) := \sum_{i=1}^{\nsites} J_i^3(\bx) \;\mbox{ with }\; 
J^3_i(\bx) := \frac{1}{|\widetilde \me_i|} \sum_{E\in \widetilde \me_i} 
\min \left\{0, \frac{\theta_E}{\bar \theta_i} - c_3 \right\}^2,
\]
where \textbf{(i)} for a given edge $E \in \me_i$, $v_E$ and $w_E$ represent its vertices in counterclockwise order, \textbf{(ii)}
\[
\tau_E := \frac{w_E - v_E}{\|w_E - v_E\|}
\]
is the tangential vector on $E$ pointing counterclockwise, \textbf{(iii)} $\theta_E := \arccos(- \tau_E \cdot \tau_{\hat E})$ is the interior angle formed by the edge $E$ and the edge before $E$, denoted~$\hat E$, considering a counterclockwise orientation, \textbf{(iv)} $\widetilde \me_i := \{E \in \me_i \text{ such that } v_E \notin \mathcal{T}_{\partial A}\}$, where $\mathcal{T}_{\partial A}$  denotes the set of vertices of $\partial A$, and \textbf{(v)} 
\[
\bar \theta_i := \frac{1}{|\widetilde \me_i|}\sum_{E \in \widetilde \me_i} \theta_E
\]
is the average angle value in the cell $V_i$, excluding the angles of $V_i$ that are also vertices~of~$A$.

The five black dashed lines in Figure~\ref{fig6} show, for each of the five solutions described in Section~\ref{smalledges}, the proportion of cells satisfying $J^3 \equiv 0$ as a function of $c_3$. It is worth noting that the merit function~$J^3$ was not considered in Section~\ref{smalledges}. However, the black dashed lines in Figure~\ref{fig6} show that, somehow, trying to balance the size of the edges produced cells more or less well balanced in relation to their angles as well. The three solid lines in the figure show the same property in relation to solutions obtained by minimizing
\begin{equation} \label{withJ3}
f(\bx) := 10 J^0(\bx) + J^1(\bx) + J^2(\bx) + J^3(\bx)
\end{equation} 
with $c_2=0.4$ and $c_3 \in \{ 0.5, 0.6, 0.7 \}$. Details of these solutions and the optimization process are shown in Table~\ref{tab3}. The numbers in the table show that, in the problems with $c_3=0.5$ and $c_3=0.6$, it was possible to find a solution with $f(\bx^*) \leq 10^{-8}$, while the same was not possible with~$c_3=0.7$, considering a single attempt and a limit of $50{,}000$ iterations. Figure~\ref{fig7} shows the solutions found.

\begin{figure}
\begin{center}
\input{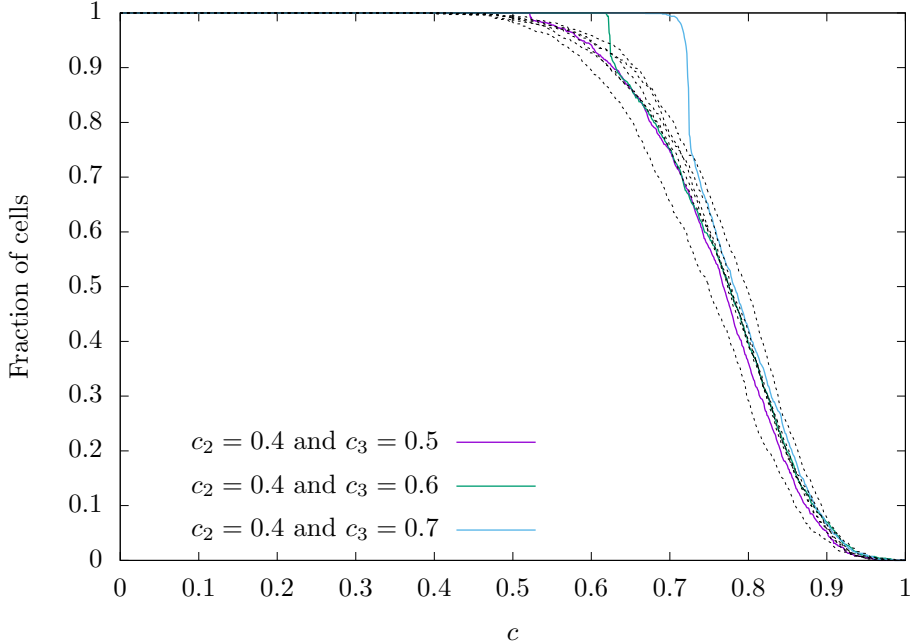}
\end{center}
\caption{This figure analyzes (solid lines) the solutions found when minimizing~(\ref{withJ3}) with $c_2=0.4$ and $c_3 \in \{0.5, 0.6, 0.7\}$. For each solution, the figure shows, as a function of $c$, the proportion of cells that satisfy the statement ``all my angles are at least $100\% \times c$ my average angle''. The five black dashed lines show, for each of the five solutions described in Section~\ref{smalledges}, in which $J^3$ was not considered, the proportion of cells satisfying $J^3 \equiv 0$ as a function of $c$.}
\label{fig6}
\end{figure}

\begin{table}[ht!]
\begin{center}
\begin{tabular}{|cccccr|}
\hline
$c_3$ & $f(\bx^*)$ & $\| \nabla f(\bx^*) \|_{\infty}$ & it & fcnt & Time\\
\hline
\hline
0.5 & 9.93676E$-$09 & 6.6E$-$06 & 20514 & 33342 & 266.28\\
0.6 & 9.99877E$-$09 & 3.2E$-$07 & 36549 & 61748 & 490.28\\
0.7 & 3.05663E$-$06 & 1.2E$-$07 & 50000 & 81335 & 654.68\\
\hline
\end{tabular}
\end{center}
\caption{Details of the optimization process and the solutions found for the problem of finding Voronoi diagrams with cells of equal volume and avoiding, simultaneously, cells with relatively small edges and sharp-angles.}
\label{tab3}
\end{table}

\begin{figure}
\begin{center}
\begin{tabular}{cc}
\includegraphics{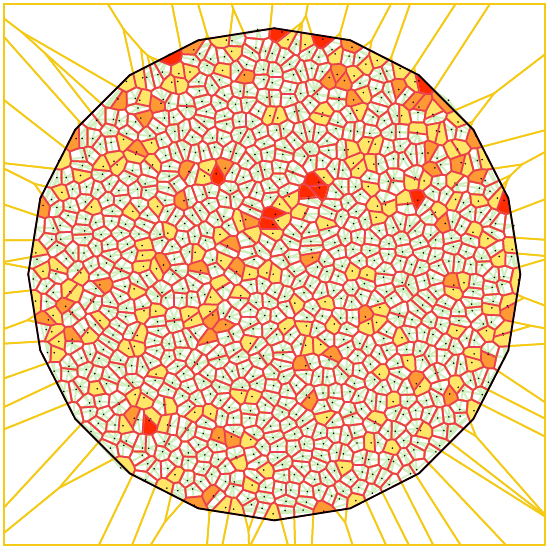} &
\includegraphics{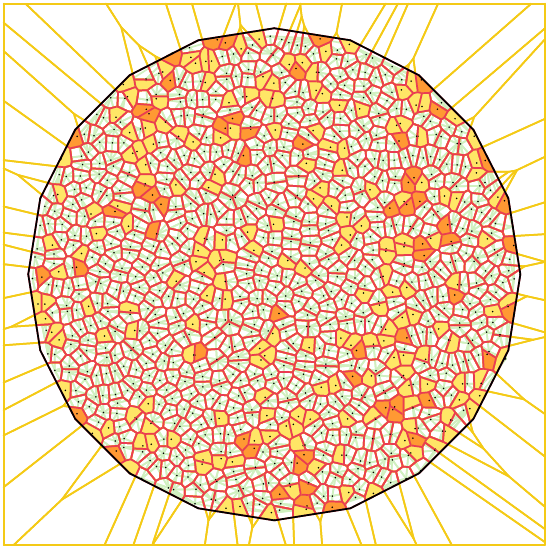} \\
(a) & (b)\\
\includegraphics{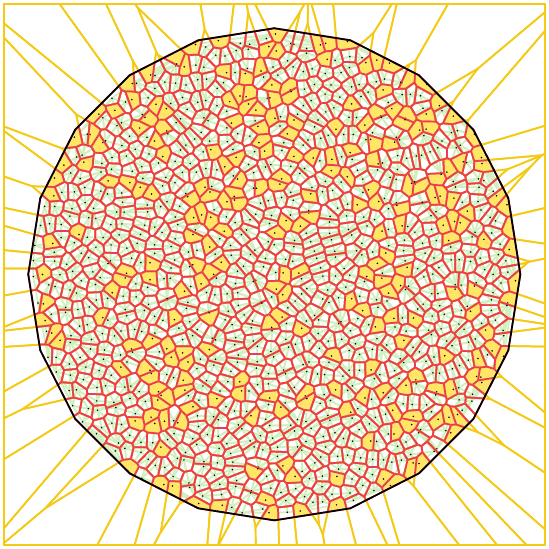} &
\includegraphics{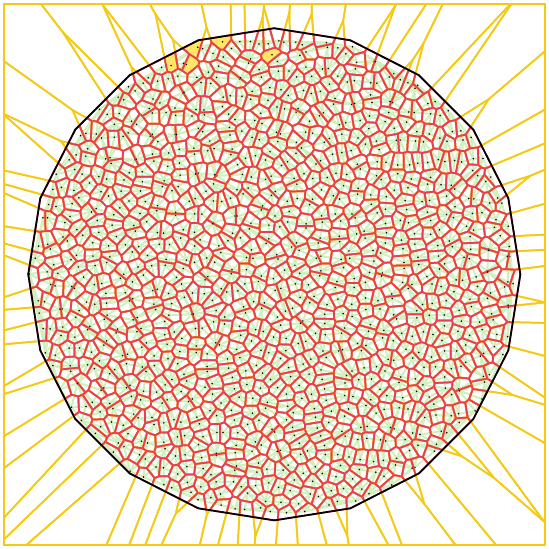} \\
(c) & (d) \\
\end{tabular}
\end{center}
\caption{Voronoi diagram with $\nsites=1000$ for the region given by a regular polygon. In (a) we show the Voronoi diagram obtained in Section~\ref{smalledges}, minimizing $f(\bx) := 10 J^0(\bx) + J^1(\bx) + J^2(\bx)$ with $c_2=0.4$. The darker the cell, the more unbalanced the angles. In (b), (c), and (d), preserving the meaning of the colors, we show the diagrams obtained by minimizing $f(\bx) := 10 J^0(\bx) + J^1(\bx) + J^2(\bx) + J^3(\bx)$ with $c_2=0.4$ and $c_3 \in \{ 0.5, 0.6, 0.7 \}$, respectively.}
\label{fig7}
\end{figure}

\subsection{Balancing Delaunay and Voronoi edges} \label{midpoints}

So far we have shown that, in association with looking for cells of equal size, we can try to balance the edges and angles of each cell. 
On the other hand, it is a fact that one cannot optimize everything at the same time, since some objectives may conflict and, in the end, one could get a result that does not minimize anything.
In this section, we return to looking at cells of equal size and try to build cells such that the midpoint of each edge is contained in the associated Delaunay edge, i.e., that the midpoints of the associated Voronoi and Delaunay edges coincide. 
This is a relevant objective  to improve the  accuracy of discrete differential operators for grid optimization; see the tweaking optimization algorithm in \cite{HRK}.
In this case, the considered merit function is given by
\begin{equation} \label{J4}
J^4(\bx) := \sum_{i=1}^{\nsites} J_i^4(\bx) \;\mbox{ with }\; 
J^4_i(\bx) := \frac{1}{|\hme|} \sum_{E \in \hme} 
\frac{\| p_E - q_E \|^2}{|E|^2},
\end{equation}
where $\hme \subseteq \me_i$  denotes the set of all edges of the cell $V_i$ that are contained in $A$, i.e., edges on the boundary of $A$ are excluded, $p_E := \frac{1}{2} (v_E + w_E)$ is the midpoint of edge~$E$, $q_E := \frac{1}{2} (a_i + a_{k(i,E)})$ is the midpoint of the edge of the Delaunay triangulation that joins $a_i$ with $a_{k(i,E)}$, and $a_{k(i,E)}$ is the site of the Voronoi cell $V_{k(i,E)}$ that shares edge $E$ with $V_i$.

The merit function $J^4_i$ measures, for a cell $V_i$, the relative mean deviation between the midpoints of associated Voronoi and Delaunay edges; while the merit function $J^4$ measures the average of that metric over all cells in the diagram. Note that, in this case, when defining $J^4_i$, we made a different choice (relative to the choice made when defining $J^2_i$ and $J^3_i$); $J^4_i$ measures the mean of the desired metric over a cell and it is nullified only if the midpoints coincide in all pairs of edges. There is no desired upper bound on the relative distance $\|p_E-q_E\|/|E|$ which causes the function to nullify if that bound is honored.

It should also be noted that it may be impossible to obtain a solution with cells $V_i$ of equal volume that satisfies $p_E=q_E$ for all $E \in \cup_{i=1}^{\nsites} \hme$. That is, in this problem, the stopping criterion should not be to get a solution that nullifies the objective function with a certain tolerance. The remaining criterion in this case is to find a stationary point of the merit function, i.e., a solution that cancels the gradient of the objective function (with a tolerance $\varepsilon_{\mathrm{opt}}>0$). Independently of that, if what we want is, among the solutions with cells of equal size, a solution that minimizes $J^4$, we must find a weight $\rho$  so that we can get such a solution by minimizing
\[
f(\bx) := 10 J^0(\bx) + J^1(\bx) + \rho J^4(\bx).
\]
A numerical experiment with $\rho=1$ shows that $J^4$ dominates $f(\bx)$ and that the obtained solution  does not have cells of the same size. On the other hand, the desired result is obtained with $\rho=10^{-4}$.

Figure~\ref{fig8} shows the solution (corresponding to the eighth of ten attempts, which was the best in terms of lowest objective function value), obtained using $\rho=10^{-4}$ and $\varepsilon_{\mathrm{opt}}=10^{-8}$. In fact, despite the $\varepsilon_{\mathrm{opt}}=10^{-8}$, the optimization method stopped due to lack of progress in the objective function in an iterate $\bx^*$ with $f(\bx^*) \approx 10^{-6}$ and $\nabla f(\bx^*) \approx 10^{-6}$, using 2039 iterations, 3310 functional evaluations, and 24.10 seconds of CPU time. Importantly, at this point, we have $J^0(\bx^*)=0$,  $J^1(\bx^*) \approx 10^{-5}$, and $J^4(\bx^*) \approx 84.40$. Moreover, $\min_{i\in\mL} \{ J^4_i(\bx^*) \} \approx 10^{-4}$, $\max_{i\in\mL} \{ J^4_i(\bx^*) \} \approx 0.89$, and the average $\frac{1}{\nsites} \sum_{i=1}^{\nsites} J^4_i(\bx^*) = J^4(\bx^*) / m \approx 0.08$, meaning that, in average over all $E \in \cup_{i=1}^{\nsites} \hme$, $\| p_E - q_E \|^2$ is smaller than 10\% of $|E|^2$. As a reference, the solution $\bar \bx$ illustrated in Figure~\ref{fig4}a, obtained by minimizing $10 J^0(\bx) + J^1(\bx)$, has $\min_{i\in\mL} \{ J^4_i(\bar \bx) \} \approx 10^{-3}$, $\max_{i\in\mL} \{ J^4_i(\bar \bx) \} \approx 184323.03$, and the average $J^4(\bar \bx) / m \approx 504.63$. These relatively ``large'' values must in part correspond to edges $E$ with ``small'' $|E|$. This shows that the inclusion of the merit function $J^4$ has the desirable side effect of avoiding ``small'' edges. In Figure~\ref{fig8}, the blue segments correspond to segments of the form $[p_E,q_E]$. In most cases, it is valid that $[p_E,q_E] \subseteq [v_E,w_E]$, where $v_E$ and $w_E$ are the vertices of the edge $E$. The cases where $[p_E,q_E] \not\subseteq [v_E,w_E]$ correspond to cases where a Voronoi edge and its associated Delaunay edge do not even intersect.

\begin{figure}
\begin{center}
\includegraphics{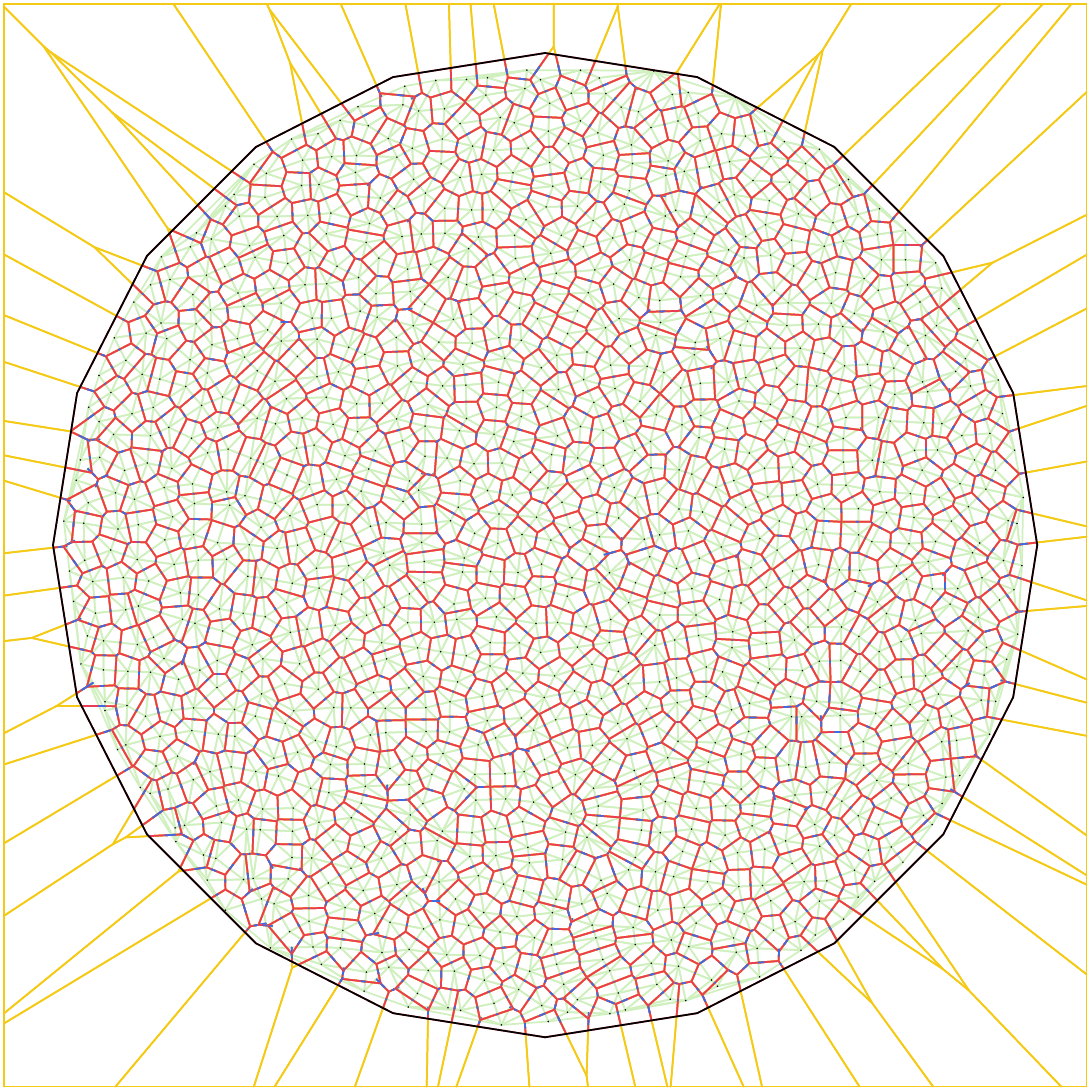}
\end{center}
\caption{Voronoi diagram with $\nsites=1000$ for the region given by a regular polygon constructed by seeking cells of the same size such that the midpoints of the edges of the Voronoi cells and their associated Delaunay edges coincide. Segments $[p_E,q_E]$, which ideally should be null or small relative to~$|E|$, appear painted blue.}
\label{fig8}
\end{figure}

\subsection{Seeking cells of varied sizes} \label{J2}

Many possible merit functions may be defined in order to achieve cells with different sizes.
For instance in the framework of centroidal Voronoi tessellations, nonconstant density functions are used to obtain  nonuniform cell sizes \cite{MR1722997}.
In this work, we opted for a small variation of the function $J^1_i$ defined in Section~\ref{samevol}.  
One could generalize $J^1_i$ by introducing a density in the integral on the cell $V_i(\bsi)$, but this would require to use a quadrature to compute the integral. In order to preserve  the exactness of the gradient, we avoid using a quadrature rule and we simply replace the constant $1$ in $J^1_i$ by a function  of the cell site~$a_i$. In this way, the desired cell size is governed by a function $\psi : A \to \mathbb{R}$ that dictates the desired value for the ratio of the cell's volume divided by the ``ideal size'' $|A|/\nsites$. The merit function follows:
\[
J^5(\bx) := \frac{1}{\nsites} \sum_{i=1}^{\nsites} \left[ J^5_i(\bx) \right]^2
\;\mbox{ with }\;
J^5_i(\bx) := \left( \int_{V_i(\bx)} dx \right) \big/ 
\left( \frac{1}{\nsites}\int_A dx \right) - \psi(a_i).
\]

A difficulty of the merit function~$J^5$ thus defined is that the sum of the desired areas does not necessarily coincide with the total area of the region~$A$. As a consequence, there is no global minimizer in which the merit function cancels out. Therefore, the stopping criterion of the optimization method must again depend on the merit function gradient norm and there is no simple way to identify whether a global minimizer has been found. As seen in the previous example, in which the stopping criterion depended on the gradient norm, the method stopped due to ``lack of progress''. That is, the method continued as long as a decrease in the objective function was observed. If, in a successive number of iterations, progress is no longer observed, the method stops. In fact, this is not an issue in practice and this stopping criterion is as valid as any other, since the tolerance $\varepsilon_{\mathrm{opt}}$ used to stop by the gradient rule is in general arbitrary.

Figure~\ref{fig9} shows two examples of minimizing $f(\bx) := J^5(\bx)$. In both cases the method stopped due to lack of progress and the solution found corresponds to the best among ten attempts in terms of lowest objective function value. In the case depicted in Figure~\ref{fig9}a, the method used~58 iterations, 439 function evaluations, and 2.29 seconds of CPU time to find a solution $\bx^*$ with $f(\bx^*) \approx 10^{-2}$ and $\nabla f(\bx^*) \approx 10^{-4}$. That merit function value corresponds to an average deviation relative to the desired area of 17\%, with a maximum deviation of 44\%. In the case depicted in Figure~\ref{fig9}b, the method used~116 iterations, 575 function evaluations, and 3.05 seconds of CPU time to find a solution $\bx^*$ with $f(\bx^*) \approx 10^{-2}$ and $\nabla f(\bx^*) \approx 10^{-4}$. That merit function value corresponds to an average deviation relative to the desired area of 8\%, with a maximum deviation of 33\%. Recall that a solution that satisfies all desired areas does not exist with very high probability. On the other hand, solutions look exactly as expected.

\begin{figure}
\begin{center}
\begin{tabular}{cc}
\includegraphics{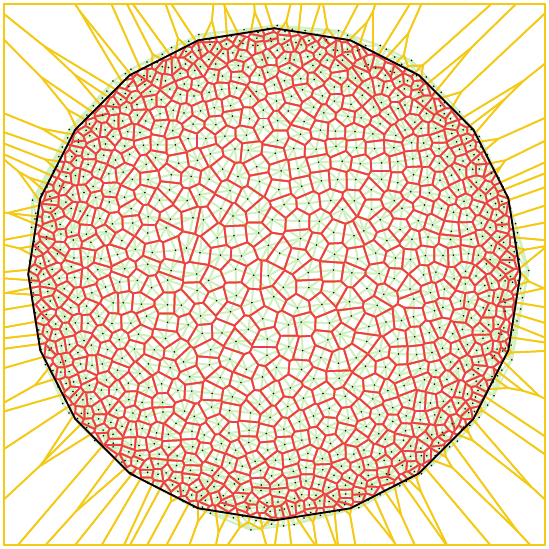} &
\includegraphics{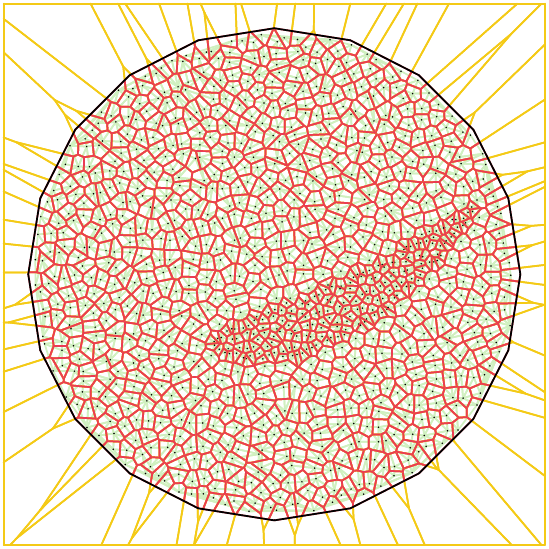} \\
(a) & (b)
\end{tabular}
\end{center}
\caption{Voronoi diagram with $\nsites=1000$ for the region given by a regular polygon constructed by seeking cells of varied sizes. In (a), $\psi(z):= 2.5 - 2 \| z - c \|^2 / r^2$, where~$c$ and~$r$ are the center and the radius of the circle circumscribing the polygon, respectively. In (b), $\psi(z)=0.25$ if $(\bar z_2-(\bar z_1/4)^2)^2 + (\bar z_1/4-1)^2 \leq 1$, where $\bar z = (2,2)^T + \frac{2}{5} z$, and $\psi(z)=1.075$, otherwise; i.e., the region to be covered by smaller cells is a scaled and translated level set of the famous Rosenbrock function.}
\label{fig9}
\end{figure}

\section{Discussion} \label{discussion}

In Section~\ref{numexp}, we showed that Voronoi diagrams with certain pre-specified desired characteristics, formally described by differentiable merit functions, can be obtained by solving an optimization problem. Using the chain rule and the first order derivatives obtained for the minimization diagrams in previous sections, the gradients of the merit functions were obtained. With the merit functions and their gradients in hand, an off-the-shelf first-order optimization method with a well-established convergence theory was employed.

On the one hand, the use of Voronoi diagrams was an illustrative example, and other types of diagrams with desired characteristics could also be constructed following the same procedure. On the other hand, even in the case of Voronoi diagrams, many of the decisions made could have been different. For example:
\begin{itemize}
\item Depending on the merit functions considered, a better-than-random starting point could be considered, such as, for example, the sites of a Central Voronoi Tessellation (CVT).

\item Instead of minimizing the mean of the desired measure for each cell, a constraint on the merit measure of each cell could have been imposed. In Section~\ref{numexp} we described merit functions $J_i^\ell$ ($\ell=0,\dots,5$) that apply to the cells $V_i$ ($i=1,\dots,\nsites$) of a Voronoi diagram. The described merit functions can be combined in a flexible way in the determination of an optimization problem whose solution is a Voronoi diagram satisfying desired pre-specified properties. Considered merit functions are written in such a way that the closer they are to zero, the better; but some are always non-negative while others are not. If a non-negative merit function $J_i^\ell$ is used to impose a constraint to the desired Voronoi diagram, then constraints of the form
\[
J_i^\ell(\bx) \leq \varepsilon_\ell, \; i=1,\dots,\nsites
\]
should be considered in the optimization problem, where $\varepsilon_\ell > 0$ is an ad-hoc given constant; while constraints of the form
\[
-\varepsilon_\ell \leq J_i^\ell(\bx) \leq \varepsilon_\ell, \; i=1,\dots,\nsites
\]
should be considered if the merit function $J_i^\ell(\bx)$ can assume negative values as well. When the closest-to-zero possible value of a merit function is sought, a term of the form
\[
J^\ell(\bx) := \frac{1}{\nsites} \sum_{i=1}^{\nsites} (J^\ell_i(\bx))^\zeta
\]
should be included in the objective function, with $\zeta=2$ or $\zeta=1$ depending on whether the merit function assumes non-negative values only or not, respectively. Summing up, given merit functions $J^{\ell}_i$ with $\ell \in {\cal J}_{F_1} \cup {\cal J}_{F_2} \cup {\cal J}_{C_1} \cup {\cal J}_{C_2}$, the nonlinear programming problem to be solved could be given by
\begin{equation} \label{obj}
\mbox{Minimize } 
\sum_{\ell \in {\cal J}_{F_1}} \rho_\ell J^\ell(\bx) + 
\sum_{\ell \in {\cal J}_{F_2}} \rho_\ell (J^\ell(\bx))^2
\end{equation}
subject to
\begin{equation} \label{constr}
J^\ell_i(\bx) \leq \varepsilon_\ell \mbox{ for } \ell \in {\cal J}_{C_1} \mbox{ and }
-\varepsilon_\ell \leq J^\ell_i(\bx) \leq \varepsilon_\ell \mbox{ for } \ell \in {\cal J}_{C_2}, \; i=1,\dots,\nsites,
\end{equation}
where $0 < \rho_\ell$ with $\ell \in {\cal J}_{F_1} \cup {\cal J}_{F_2}$ are given weights and $0 < \varepsilon_\ell$ with $\ell \in {\cal J}_{C_1} \cup {\cal J}_{C_2}$ are given tolerances. For practical purposes, it is important to note that, for the merit functions considered in Section~\ref{numexp}, the Jacobian of the constraints in~(\ref{constr}) is a sparse matrix. It is also important to remark that nonlinear programming problems of the form~(\ref{obj},\ref{constr}) can be solved with, for example, Augmented Lagrangian methods~\cite{abmstango,bmbook,bmcomper}.

\item The experiments in Section~\ref{numexp} and the remark in the item above correspond to the situation in which the user knows in which way he/she wants to combine the various possible objectives. When this is not the case, the problem is a multi-objective optimization problem, which may or may not have constraints. If the different objectives are in fact conflicting, then there is no solution that minimizes all of them at the same time and the Pareto set of solutions or at least a Pareto solution should be computed. For that purpose, classical scalarization techniques, such as linear scalarization and $\epsilon$-constraint minimization can be used; see, for example, \cite{miettinen1999nonlinear}. In all situations, the merit functions presented in Section~\ref{numexp}, their gradients, and the aforementioned optimization methods (Spectral Projected Gradients and Augmented Lagrangians) can still be used.

\item Once a problem is formulated and solved, new smaller optimization problems can be solved to re-optimize or refine details of specific parts of the solution. Sites whose position must remain fixed can be defined as constants or as variables with additional constraints.

\end{itemize}

What this discussion shows is that the proposed methodology provides a flexible way to construct Voronoi diagrams with desired characteristics, which could be extended to other types of minimization diagrams. The point being made is that the approach is flexible and that optimization problems can be solved with off-the-shelf methods with well-established convergence theory. It is important to observe that many known and already in use ways of constructing Voronoi diagrams with desired characteristics fit into the proposed framework.

\section{Conclusions}

In this paper we have developed a perturbation theory and performed a sensitivity analysis for sets defined as intersections of smooth sets and for minimization diagrams.
This contributes to consolidate the theory developed in \cite{laurainLEM} for dynamic minimization diagrams.
This also contributes to advance the theory of nonsmooth shape optimization and calculus \cite{coveringsecond,coveringfirst,LAURAIN2020328,MR3535238}.
This sensitivity analysis allowed us to obtain general formulas for computing derivatives of functions depending on the cells, edges and vertices  of minimization diagrams.
Our approach has the advantage of treating interior and boundary edges and vertices in a unified way.
Using the chain rule, we have applied these general formulas to compute the first derivatives of different merit functions related to the cells of a Voronoi diagram. 
With this tool, using established off-the-shelf optimization algorithms, we constructed Voronoi diagrams satisfying pre-established desired properties. Numerical experiments showed that the proposed techniques work well in practice. 

The generality of the perturbation theory developed in this work opens various perspectives for future research. 
The optimization of generalized Voronoi diagrams of interest for applications, such as  multiplicatively weighted Voronoi diagrams  and power diagrams, seems to be a natural consideration. 
Problems involving partial differential equations will also be considered, and present interesting challenges from the shape optimization point of view, see \cite{laurainLEM}.
Due to its significance for computing numerical solutions of partial differential equations, grid generation and optimization should be one of the main focus for future research.
A natural step would be to extend and apply this theory for grid optimization and manifolds and in particular on spheres \cite{HRK}.
Following the line of \cite{coveringsecond,coveringfirst}, second derivatives could also be computed using similar techniques. However, in the context of grid generation and optimization, the problems to be solved are large and, therefore, it is not clear whether optimization techniques using second order derivatives can be used. Perhaps they can be used in the context of re-optimizing small regions of an already computed mesh, but that practical utility would need to be established. This will be the subject of future work.

\vspace{0.5cm}
\noindent{\bf Acknowledgments.}
The authors would like to thank Pedro da Silva Peixoto for the inspiring discussions and for sharing with us  his knowledge and useful references about grid generation, optimization, quality measures and  the tweaking optimization algorithm.


\section*{Appendix~A: Gradients of the considered merit functions} \label{appA}

We start by defining some useful notations for writing the gradients calculated in this section.
Recall that $A$ is open, polygonal and that $\mathcal{T}_{\partial A}$ denotes  the set of vertices of $\partial A$.
The Voronoi diagram associated with $A$ is denoted  $\mv(\bsi)$.
The set $\mathcal{T}_{\rm{int}}$ denotes the set of vertices of $\mv(\bsi)$ belonging to the open set $A$.
The set $\mathcal{T}_{\rm{bd}}$ denotes the set of vertices of $\mv(\bsi)$ belonging to $\partial A\setminus \mathcal{T}_{\partial A}$.
If $v\in \mathcal{T}_{\rm{int}}\cap \overline{V_i(\bsi)}$, then $v$ belongs to exactly three cells due to Lemma \ref{lem:009}, and we denote $\ell_1(i,v)$ and  $\ell_2(i,v)$ the indices of these cells with respect to a counterclockwise orientation around $v$.
In the case $v\in \mathcal{T}_{\rm{bd}}\cap \overline{V_i(\bsi)}$, $v$ belongs to exactly two cells due to Lemma \ref{lem:009}, and we denote by  $\ell(i,v)$
the index of the other cell.
In the case $v\in \mathcal{T}_{\partial A}$, $v$ belongs to only one cell and is fixed, thus the contribution of such points to the gradient is zero.
Recall that $\me_i$ is the set of edges of the cell $V_i(\bsi)$   and   $\hme$ denotes the set of interior edges of the cell $V_i(\bsi)$, i.e., edges that are included in $A$.
For an edge $E\in\me_i$, $v_E$ and $w_E$ denote the vertices of $E$ with respect to a counterclockwise orientation. 
Also recall that  $\widetilde \me_i := \{E \in \me_i \text{ such that } v_E \notin \mathcal{T}_{\partial A}\}$.

Then we introduce the following function which frequently appears in the gradient formulas:
\[
F(i,v,\zeta) :=
\left\{
\begin{array}{ll}
\displaystyle  
  M(\ell_1(i,v),\ell_2(i,v),i)^\top \zeta \cdot \dx_i
+ M(\ell_2(i,v),i,\ell_1(i,v))^\top \zeta \cdot \dx_{\ell_1(i,v)}\phantom{,} & \\[2mm]
\hfill + M(i,\ell_1(i,v),\ell_2(i,v))^\top \zeta \cdot \dx_{\ell_2(i,v)}
, & \mbox{if } v \in \mathcal{T}_{\rm{int}},\\[2mm]
\displaystyle 
 \mm(\ell(i,v),i)^\top \zeta \cdot \dx_i 
+\mm(i,\ell(i,v))^\top \zeta \cdot \delta x_{\ell(i,v)}, &
\mbox{if } v \in \mathcal{T}_{\rm{bd}},\\[2mm]
\displaystyle 0, & \mbox{if } v \in \mathcal{T}_{\partial A}. 
\end{array}
\right.
\]
Note that $F(i,v,\zeta) = \mathcal{F}(i,v) \cdot \zeta$, where $\mathcal{F}(i,v) $ is given by  \eqref{grad:G2_vor:F}.
Note that the case $v \in \mathcal{T}_{\partial A}$ does not appear in  \eqref{grad:G2_vor:F} since $A$ was assumed to be smooth in Section~\ref{sec:vordiag}.

\subsection*{Gradient of $J^1$}

We compute
$$\nabla J^1(\bx) = \frac{2}{\nsites} \sum_{i=1}^{\nsites} J^1_i(\bx)\nabla J^1_i(\bx)$$
with 
\begin{align*} 
\nabla J_i^1(\bsi) \cdot\dbsi
& = \frac{\nsites}{|A|}\sum_{E\in\hme} \frac{|E|}{\|\si_i - \si_{k(i,E)}\|} [\delta\si_i\cdot ( p_E- \si_i  ) 
-\delta\si_{k(i,E)}\cdot ( p_E - \si_{k(i,E)} )],
\end{align*}
where $p_E := (v_E +w_E)/2$ and $k(i,E)$ is the index such that $E = \overline{V_i(\bsi)}\cap \overline{V_{k(i,E)}(\bsi)}$.

\subsection*{Gradient of $J^2$}
We compute 
$$\nabla J^2(\bx) =  \sum_{i=1}^{\nsites} \nabla J^2_i(\bx).$$
Recalling that $\bar E_i = P_i / n_i$ and $P_i = \sum_{\tilde E\in \me_i} |\tilde E|$, we obtain
\begin{align*}
\nabla J^2_i(\bx)&=  \frac{2}{n_i}  \sum_{E\in \me_i}  \min\left\{ 0,  \frac{|E|}{\bar E_i} - c \right\}
\left( \frac{\nabla |E|}{\bar E_i}  -  \frac{|E|}{n_i\bar E_i^2} \left( \sum_{\tilde E\in \me_i}\nabla |\tilde E| \right) \right) \\
&=  \frac{2}{P_i}  \left(\sum_{E\in \me_i}  \min\left\{ 0, \frac{|E|}{\bar E_i} - c \right\}
\nabla |E| \right)
 - \frac{2}{P_i}  \left( \sum_{\tilde E\in \me_i} \nabla |\tilde E| \right) \left( \sum_{E\in \me_i} \min\left\{ 0, \frac{|E|}{\bar E_i} - c \right\} \frac{|E|}{P_i} \right)\\
&=  \frac{2}{P_i} \sum_{E\in \me_i} 
\left( \min\left\{ 0,  \frac{|E|}{\bar E_i} - c \right\} 
-  \sum_{\tilde E\in \me_i}  \frac{|\tilde E|}{P_i}  \min\left\{ 0, \frac{|\tilde E|}{\bar E_i} - c \right\}  
\right)
 \nabla |E|  .
\end{align*}
Here $\nabla |E|\cdot \delta\bx$ is given by \eqref{grad:G2_vor}, thus we obtain, using \eqref{grad:G2_vor:F} and the property $F(i,v,\zeta) = \mathcal{F}(i,v) \cdot \zeta$,  
$$\nabla J^2_i(\bx)\cdot\delta\bx 
= \sum_{E\in \me_i} \mu(E) (F(i,w_E,\tau_E) - F(i,v_E,\tau_E) )$$
with 
$$\mu(E) := \frac{2}{P_i}\left( \min\left\{ 0, \frac{|E|}{\bar E_i} - c \right\} -  \sum_{\tilde E\in \me_i} \frac{|\tilde E|}{P_i}  \min\left\{ 0,  \frac{|\tilde E|}{\bar E_i} - c \right\}  \right) .$$

\subsection*{Gradient of $J^3$}

We compute 
$$\nabla J^3(\bx) =  \sum_{i=1}^{\nsites} \nabla J^3_i(\bx).$$
The calculation of $\nabla J^3_i$ is similar to the calculation of $\nabla J^2_i$ and yields
$$
\nabla J^3_i(\bx) =  \frac{2}{\bar \theta_i |\tilde \me_i|} \sum_{E\in \tilde \me_i} 
\left( \min\left\{ 0, \frac{\theta_E}{\bar\theta_i} - c \right\} -  \sum_{\tilde E\in \tilde \me_i} \frac{\theta_{\tilde E}}{\bar\theta_i |\tilde \me_i|} \min\left\{ 0, \frac{\theta_{\tilde E}}{\bar\theta_i} - c \right\} 
\right)
\nabla\theta_E.
$$
Now recall that 
$
\tau_E = (w_E - v_E)/|E|
$
is the tangential vector on $E$ pointing counterclockwise, and $\theta_E := \arccos(- \tau_E \cdot \tau_{\hat E})$ is the interior angle formed by the edge $E$ and the edge before $E$, denoted~$\hat E$, considering a counterclockwise orientation.
Then we compute
$$\nabla\theta_E\cdot\delta\bx 
= - \frac{\nabla [-\tau_E\cdot\tau_{\hat E}]\cdot \delta\bx}{(1-(\tau_E\cdot\tau_{\hat E})^2 )^{1/2}} 
= \frac{D\tau_E \delta\bx\cdot \tau_{\hat E} + D\tau_{\hat E} \delta\bx\cdot \tau_E}{(1-(\tau_E\cdot\tau_{\hat E})^2 )^{1/2}}. 
$$
In view of \eqref{triple_perturb:vor}, \eqref{wprime_bd:vor},  \eqref{grad:G2_vor:F} we have $D v_E \delta\bx = \mathcal{F}(i,v_E)$ and $D w_E \delta\bx = \mathcal{F}(i,w_E)$.
Using $\nabla |E|\cdot \delta\bx$  given by \eqref{grad:G2_vor}, we get
\begin{align*}
D\tau_E \delta\bx\cdot \tau_{\hat E}
& =
\left[\frac{\mathcal{F}(i,w_E) - \mathcal{F}(i,v_E)}{|E|} 
- \frac{(w_E -v_E)}{|E|^2}(\mathcal{F}(i,w_E)\cdot\tau_E - \mathcal{F}(i,v_E)\cdot\tau_E)
\right] \cdot \tau_{\hat E}\\
& =
(\mathcal{F}(i,w_E) - \mathcal{F}(i,v_E))\cdot \left(\frac{\tau_{\hat E}}{|E|}  - \frac{\tau_E\cdot \tau_{\hat E}}{|E|}  \tau_E\right)\\
& = 
(\mathcal{F}(i,w_E) - \mathcal{F}(i,v_E))\cdot \nu_E \frac{(\nu_E\cdot \tau_{\hat E})}{|E|},
\end{align*}
and similarly
\begin{align*}
D\tau_{\hat E} \delta\bx\cdot \tau_E
& = 
(\mathcal{F}(i,w_{\hat E}) - \mathcal{F}(i,v_{\hat E}))\cdot \nu_{\hat E} \frac{(\nu_{\hat E}\cdot \tau_E)}{|{\hat E}|}.
\end{align*}
By definition we have, using $0<\theta_E<\pi$ for $E\in\tilde \me_i$, 
$$\left(  1-(\tau_E\cdot\tau_{\hat E})^2 \right)^{-1/2} = \left(  1-(\cos \theta_E)^2 \right)^{-1/2} = (\sin\theta_E)^{-1}$$
and also $\nu_E\cdot \tau_{\hat E} = \sin \theta_E$,  $\nu_{\hat E}\cdot \tau_E = -\sin \theta_E$.
Gathering these results we get
$$\nabla\theta_E\cdot\delta\bx 
= (\mathcal{F}(i,w_E) - \mathcal{F}(i,v_E))\cdot \frac{\nu_E}{|E|}
- (\mathcal{F}(i,w_{\hat E}) - \mathcal{F}(i,v_{\hat E}))\cdot \frac{\nu_{\hat E}}{|\hat E|}. 
$$
This yields
$$\nabla J^3_i(\bx)\cdot\delta\bx 
= \sum_{E\in  \tilde \me_i} \frac{\eta(E)}{|E|} (F(i,w_E,\nu_E) - F(i,v_E,\nu_E) )
- \frac{\eta(E)}{|\hat E|}(F(i,w_{\hat E},\nu_{\hat E}) - F(i,v_{\hat E},\nu_{\hat E}))
$$
with 
\begin{align*}
\eta(E) & := 
\frac{2}{\bar \theta_i |\tilde \me_i|} 
\left( \min\left\{ 0, \frac{\theta_E}{\bar\theta_i} - c \right\} -  \sum_{\tilde E\in \tilde \me_i} \frac{\theta_{\tilde E}}{\bar\theta_i |\tilde \me_i|} \min\left\{ 0, \frac{\theta_{\tilde E}}{\bar\theta_i} - c \right\} 
\right).
\end{align*}

\subsection*{Gradient of $J^4$}
Recall that 
\[
J^4_i(\bx) := \frac{1}{|\hme|}\sum_{E\in\hme} \frac{\| d_E \|^2}{|E|^2}, \mbox{ where } d_E := p_E - q_E.
\]
We compute 
\begin{align*}
\nabla J^4_i(\bx) \cdot  \delta\bx &= \frac{1}{|\hme|}\sum_{E\in\hme} \nabla \left( \frac{\| d_E \|^2}{|E|^2} \right)  \cdot  \delta\bx.\\
&= \frac{1}{|\hme|} \sum_{E\in\hme} \left( \frac{2 d_E \cdot\nabla d_E }{|E|^2} - 2 \frac{\| d_E \|^2}{|E|^3} \nabla |E|\right)  \cdot  \delta\bx.\\
&= \frac{1}{|\hme|} \sum_{E\in\hme} \left( \frac{d_E \cdot( D v_E\dbsi + D w_E\dbsi - \delta\si_i - \delta\si_{k(i,E)}) }{|E|^2} - 2 \frac{\| d_E \|^2}{|E|^3} \nabla |E|\cdot  \delta\bsi  \right) .
\end{align*}
In view of \eqref{triple_perturb:vor}, \eqref{wprime_bd:vor},  \eqref{grad:G2_vor:F} we have $D v_E \delta\bx = \mathcal{F}(i,v_E)$ and $D w_E \delta\bx = \mathcal{F}(i,w_E)$.
Using \eqref{grad:G2_vor} we have $\nabla |E|\cdot \delta\bx = F(i,w_E,\tau_E) - F(i,v_E,\tau_E)$.
Considering that $\tau_E = (w_E - v_E)/|E|$, we get
\begin{align*}
\nabla J^4_i(\bx) \cdot  \delta\bx  
&= \frac{1}{|\hme|} \sum_{E\in\hme} \left( \frac{d_E \cdot(  \mathcal{F}(i,v_E) +  \mathcal{F}(i,w_E) - \delta\si_i - \delta\si_{k(i,E)})) }{|E|^2} 
- 2 \frac{\| d_E \|^2}{|E|^3} \nabla |E|\cdot  \delta\bsi\right)\\
&= \frac{1}{|\hme|} \sum_{E\in\hme} F(i,v_E,\mu_E) + F(i,w_E,\eta_E) 
- \frac{d_E\cdot \delta \si_i}{|E|^2}  - \frac{d_E\cdot \delta \si_{k(i,E)}}{|E|^2},
\end{align*}
where
\begin{align*}
\mu_E &:= \frac{d_E}{|E|^2} + 2 \frac{\|d_E\|^2}{|E|^4} (w_E - v_E) \quad \text{ and }
\quad  
\eta_E := \frac{d_E}{|E|^2} - 2 \frac{\|d_E\|^2}{|E|^4} (w_E - v_E).
\end{align*}

\subsection*{Gradient of $J^5$}

We compute
$$\nabla J^5(\bx) = \frac{2}{\nsites} \sum_{i=1}^{\nsites} J^5_i(\bx)\nabla J^5_i(\bx)$$
with 
$
\nabla J_i^5(\bsi) \cdot\dbsi
= \nabla J_i^1(\bsi) - \nabla_{\si_i}\psi(\si_i)\cdot \delta\si_i 
$, see the calculation of the gradient of $J^1$.

\section*{Appendix~B: Description of ``Letter A'' and ``Key'' regions}

The description of each region~$A$ consists in the list of the vertices, in counterclockwise order, of the convex polygons~$A_j$ that constitute the partition of the problem. Both regions being described here were inspired by \cite[Fig.2]{sieger}.

The non-convex polygon in the form of the letter ``A'' shown, with $\mathrm{Vol}(A) \approx 232.5318$, is composed by $p=16$ convex polygons. The vertices of polygons $A_1,\dots,A_{16}$ are given below:
\begin{description}[topsep=0.1cm,itemsep=0.1cm,partopsep=0.2cm, parsep=0.2cm]
\item $\mathcal{V}(A_1) = \{(-1,0), (8.2,0), (8.2,0.62), (6.92,0.76), (1,0.8), (-0.1,0.6)\}$,
\item $\mathcal{V}(A_2) = \{(1,0.8), (6.92,0.76), (5.86,1.32), (2,1.5)\}$,
\item $\mathcal{V}(A_3) = \{(2,1.5), (5.86,1.32), (5.24,2.65), (3.5,4.36)\}$,
\item $\mathcal{V}(A_4) = \{(5.24,2.65), (5.58,4.36), (3.5,4.36)\}$,
\item $\mathcal{V}(A_5) = \{(3.5,4.36), (5.58,4.36), (7.58,9), (5.5,9)\}$,
\item $\mathcal{V}(A_6) = \{(5.5,9), (7.58,9), (8.4,10.91), (6.32,10.91)\}$,
\item $\mathcal{V}(A_7) = \{(6.32,10.91), (8.4,10.91), (14.02,23.95), (11.94,23.95)\}$,
\item $\mathcal{V}(A_8) = \{(11.94,23.95), (18.72,23.95), (15.89,30.56), (14.79,30.56)\}$,
\item $\mathcal{V}(A_9) = \{(19.6,10.91), (24.3,10.91), (18.72,23.95), (14.02,23.95)\}$,
\item $\mathcal{V}(A_{10}) = \{(7.58,9), (20.42,9), (19.6,10.91), (8.4,10.91)\}$,
\item $\mathcal{V}(A_{11}) = \{(20.42,9), (25.12,9), (24.3,10.91), (19.6,10.91)\}$, 
\item $\mathcal{V}(A_{12}) = \{(22.06,5.15), (26.54,6), (25.12,9), (20.42,9)\}$,
\item $\mathcal{V}(A_{13}) = \{(22.46,2.26), (28.53,2.3), (26.54,6), (22.06,5.15)\}$,
\item $\mathcal{V}(A_{14}) = \{(22.05,1.2), (29.6,1.22), (28.53,2.3), (22.46,2.26)\}$,
\item $\mathcal{V}(A_{15}) = \{(21.24,0.82), (30.79,0.74), (29.6,1.22), (22.05,1.2)\}$,
\item $\mathcal{V}(A_{16}) = \{(19.13,0), (32.15,0), (32.15,0.6), (30.79,0.74), (21.24,0.82), (19.13,0.6)\}$.
\end{description}

The non-convex polygon in the shape of a key, with $\mathrm{Vol}(A) \approx 88.15209$, is composed by $p=22$ convex polygons. The vertices of polygons $A_1,\dots,A_{22}$ are given below:
\begin{description}[topsep=0.1cm,itemsep=0.1cm,partopsep=0.1cm, parsep=0.1cm]
\item $\mathcal{V}(A_1) = \{(0,0), (0,-3.44), (2.49,-3.44), (3,-3), (3,0)\}$,
\item $\mathcal{V}(A_2) = \{(0,-3.44), (0,-4.5), (1.58,-4.5), (2.49,-3.74), (2.49,-3.44)\}$,
\item $\mathcal{V}(A_3) = \{(0,-4.5), (0,-4.79), (1.58,-4.79), (1.58,-4.5)\}$,
\item $\mathcal{V}(A_4) = \{(0,-4.79), (0,-5.48), (1.87,-5.48), (2,-5.4), (2,-5.14), (1.58,-4.79) \}$,
\item $\mathcal{V}(A_5) = \{(0,-5.48), (0,-5.86), (1.87,-5.86), (1.87,-5.48)\}$,
\item $\mathcal{V}(A_6) = \{(0,-5.86), (0,-6.9), (2.26,-6.9), (2.42,-6.76), (2.42,-6.51), (1.87,-5.86)\}$,
\item $\mathcal{V}(A_7) = \{(0,-6.9), (0,-7.22), (2.26,-7.22), (2.26,-6.9)\}$,
\item $\mathcal{V}(A_8) = \{(0,-7.22), (0,-7.98), (2.1,-7.98), (2.43,-7.65), (2.43,-7.4), (2.26,-7.22)\}$,
\item $\mathcal{V}(A_9) = \{(0,-7.98), (0,-8.2), (2.1,-8.2), (2.1,-7.98)\}$,
\item $\mathcal{V}(A_{10}) = \{(0,-8.2), (0,-8.87), (2.26,-8.87), (2.43,-8.74), (2.43,-8.49), (2.1,-8.2)\}$,
\item $\mathcal{V}(A_{11}) = \{(0,-8.87), (0,-9.17), (2.26,-9.17), (2.26,-8.87)\}$,
\item $\mathcal{V}(A_{12}) = \{(0,-9.17), (0,-10.15), (1.87,-10.15), (2.43,-9.62), (2.43,-9.28), (2.26,-9.17)$,
\item $\mathcal{V}(A_{13}) = \{(0,-10.15), (0,-10.5), (0.37,-10.9), (0.94,-10.9),(1.87,-10.35), (1.87,-10.15)\}$,
\item $\mathcal{V}(A_{14}) = \{(0.94,-10.9), (1.29,-11.35), (1.86,-11.12), (2.26,-10.7), (1.87,-10.35)\}$,
\item $\mathcal{V}(A_{15}) = \{(0.85,6.06), (0.58,6.68), (-0.51,6.53), (-3,6), (-3.6,3.5), (-3,0.7), (0,0)\}$,
\item $\mathcal{V}(A_{16}) = \{(1.5,5.86), (0.85,6.06), (0,0), (3,0)\}$,
\item $\mathcal{V}(A_{17}) = \{(1.5,5.86), (3,0), (2.15,6.06)\}$,
\item $\mathcal{V}(A_{18}) = \{(2.15,6.06), (3,0), (6,0.7), (6.6,3.35), (6,6), (3.51,6.53), (2.42,6.68)\}$,
\item $\mathcal{V}(A_{19}) = \{(0.58,6.68), (0.85,7.3), (0.69,8.16), (0,7.62), (-0.51,6.53)\}$,
\item $\mathcal{V}(A_{20}) = \{(0.85,7.3), (1.5,7.5), (1.5,8.5), (0.69,8.16)\}$,
\item $\mathcal{V}(A_{21}) = \{(1.5,7.5), (2.15,7.3), (2.31,8.16), (1.5,8.5)\}$,
\item $\mathcal{V}(A_{22}) = \{(2.42,6.68), (3.51,6.53), (3,7.62), (2.31,8.16), (2.15,7.3)\}$.
\end{description}

The regular polygon has $n_{\mathrm{vert}}=20$ vertices of the form $(\cos(\theta_i),sin(\theta_i))$ with $\theta_i=2\pi (i-1)/n_{\mathrm{vert}}$ for $i=1,\dots,n_{\mathrm{vert}}$. The vertices of the convex polygon with six edges are given by:
\begin{description}[topsep=0.1cm,itemsep=0.1cm,partopsep=0.1cm, parsep=0.1cm]
\item $\mathcal{V}(A_1) = \{(0.65,2.31), (-1.98,2.71), (-3.35,1.64), (-2.59,-0.34), (-0.22,-1.07), (0.54,0.72)\}$.
\end{description}

\bibliographystyle{plain}
\bibliography{main}

\end{document}